\documentclass[11pt]{amsart}
\usepackage{amsmath, amsthm, amsfonts, amsbsy, amssymb, upref} 
\usepackage{epsfig, graphicx}
\usepackage{pdfsync, hyperref}
\newcommand{\vol}{\mathrm{Vol}}
\newcommand{\I}{\mathrm{i}}

\newcommand{\ep}{\epsilon}
\newcommand{\abs}[1]{| #1|}

\newcommand{\mm}{\mathcal{M}}
\newcommand{\mr}{\mathcal{R}}
\newcommand{\MU}{\mathcal{U}}
\newcommand{\mc}{\mathcal{C}}
\newcommand{\ms}{\mathcal{S}}
\newcommand{\tf}{\tilde{f}}
\newcommand{\tl}{\tilde{L}}

\newcommand{\hq}{\widehat{Q}}

\newtheorem{thm}{Theorem}
\newtheorem{lmm}[thm]{Lemma}
\newtheorem{cor}[thm]{Corollary}
\newtheorem{prop}[thm]{Proposition}

\newcommand{\cc}{\mathbb{C}}

\newcommand{\ee}{\mathbb{E}}

\newcommand{\cp}{\mathcal{P}}

\newcommand{\pp}{\mathbb{P}}

\newcommand{\ra}{\rightarrow}
\newcommand{\rr}{\mathbb{R}}

\newcommand{\var}{\mathrm{Var}}

\newcommand{\zz}{\mathbb{Z}}
 
\newcommand{\fpar}[2]{\frac{\partial #1}{\partial #2}}

\begin{document}
\title{Invariant measures and the soliton resolution conjecture}
\author{Sourav Chatterjee}
\address{Courant Institute of Mathematical Sciences, New York University, 251 Mercer Street, New York, NY 10012}
\email{sourav@cims.nyu.edu}
\thanks{Research partially supported by  NSF grant DMS-1005312}

\begin{abstract}
The soliton resolution conjecture for the focusing nonlinear Schr\"odinger equation (NLS) is the vaguely worded claim that a global solution of the NLS, for generic initial data, will eventually resolve into a radiation component that disperses like a linear solution, plus a localized component that behaves like a soliton or multi-soliton solution. Considered to be one of the fundamental open problems in the area of nonlinear dispersive equations, this conjecture has eluded a proof or even a precise formulation till date. 

This paper proves a ``statistical version'' of this conjecture at mass-subcritical nonlinearity, in the following sense. The uniform probability distribution on the set of all functions with a given mass and energy, if such a thing existed, would be a natural invariant measure for the NLS flow and would reflect the long-term behavior for ``generic initial data'' with that mass and energy. Unfortunately, such a probability measure does not exist. We circumvent this problem by constructing a sequence of discrete measures that, in principle, approximate this fictitious probability distribution as the grid size goes to zero. We then show that a continuum limit of this sequence of probability measures does exist in a certain sense, and in agreement with the soliton resolution conjecture, the limit measure concentrates on the unique ground state soliton. Combining this with results from ergodic theory, we present a tentative formulation and proof of the soliton resolution conjecture in the discrete setting. 

The above results, following in the footsteps of a program of studying the long-term behavior of nonlinear dispersive equations through their natural invariant measures initiated by Lebowitz, Rose and Speer, and carried forward by Bourgain, McKean, Tzvetkov, Oh and others, is proved using a combination of techniques from large deviations, PDE, harmonic analysis  and bare hands probability theory. It is valid in any dimension. 
\end{abstract}
\maketitle

\setcounter{tocdepth}{1}

\tableofcontents{}

\section{Introduction}\label{intro}
\subsection{Probabilistic motivation}
Suppose that we are asked to choose a function $f:\rr^d \ra\cc$ uniformly at random from the set of all $v:\rr^d \ra\cc$ satisfying $M(v)=m$ for some given constant $m$, where
\begin{equation}\label{massdef}
M(v) := \int_{\rr^d} |v(x)|^2 dx. 
\end{equation}
While this question does not make sense mathematically, the only reasonable answer that one can give is that $f$ must be equal to zero almost everywhere. Paradoxically, this $f$ does not satisfy $M(f)=m$. The paradox is resolved if we view this question as the limit of a sequence of discrete questions: First approximate $\rr^d$ by a large box $[-L, L]^d$; then discretize this box by splitting it as a union of many small cubes; finally, choose a function $f:\rr^d \ra\cc$ uniformly from the set of all functions $v$ that are piecewise constant in these small cubes and zero outside the box $[-L,L]^d$, and satisfy $M(v)=m$. This is a probabilistically sensible question; the resulting $f$ approaches zero in the $L^\infty$ norm as the box size goes to infinity.

Now suppose that we add one more constraint, namely, that $f$ should satisfy $H(f)=E$, where $H$ is the functional
\begin{equation}\label{energydef}
H(v):=\frac{1}{2}\int_{\rr^d} |\nabla v(x)|^2dx - \frac{1}{p+1}\int_{\rr^d} |v(x)|^{p+1} dx,
\end{equation}
and $p>1$  and $E\in \rr$ are given constants. The motivation for adding this second constraint comes from the study of  microcanonical invariant measures of nonlinear Schr\"odinger equations (more on this later). One problem that arises immediately is that if $v$ satisfies $M(v)=m$ and $H(v)=E$, so does the function $u(x):= \alpha_0 v(x+x_0)$ for any $x_0\in \rr^d$ and $\alpha_0\in S^1$, where $S^1$ is the unit circle in $\cc$. Thus, it is reasonable to first quotient the function space by the equivalence relation $\sim$, where $u\sim v$ means that $u$ and $v$ are related in the above manner. 

When $p$ satisfies the ``subcriticality'' condition  $p<1+4/d$, standard results from the theory of nonlinear Schr\"odinger equations imply that the set of functions $v$ that minimize $H(v)$ given $M(v)=m$ form a unique equivalence class of the relation~$\sim$. This equivalence class is known as the ``ground state soliton'' of mass $m$. The main result of this manuscript (Theorem \ref{ourmain}) says that if we attempt to choose an equivalence class uniformly at random from all classes satisfying $M(v)=m$ and $H(v)=E$, by first discretizing the problem and then passing to the continuum limit, then we end up choosing this ground state soliton. As before, there is no paradox in the fact that the ground state soliton may not satisfy the constraint~$H(v)=E$. While the problem is quite simple for the single constraint $M(v)=m$, the addition of the second constraint $H(v)=E$ somehow renders it unreasonably difficult; indeed, nearly the entirety of this long manuscript is devoted to the proof of Theorem \ref{ourmain}.

The above result is a small step towards understanding  uniform probability distributions on manifolds in function spaces that are defined by a finite number of constraints. These distributions arise as ``microcanonical'' invariant measures for Hamiltonian flows on such manifolds. The conserved quantities for the flow give the constraints defining the manifold. A preliminary attempt with a simpler problem was made in~\cite{chatterjee10}. All of this, and how it connects to  the behavior of nonlinear Schr\"odinger flows and ideas from statistical physics, will be discussed in greater detail in the remainder of this section.

\subsection{The nonlinear Schr\"odinger equation}
A complex-valued function $u$ of two variables $x$ and $t$, where $x\in \rr^d$ is the space variable and $t\in \rr$ is the time  variable, is said to satisfy a $d$-dimensional nonlinear Schr\"odinger equation (NLS) if 
\begin{equation}\label{nlsequation}
\I\partial_t u = - \Delta u + \kappa |u|^{p-1}u,
\end{equation}
where $\Delta$ is the Laplacian operator in $\rr^d$, $p>1$ is the nonlinearity parameter, and $\kappa$ is a parameter which is either $+1$ or $-1$. When $\kappa=1$, the equation is called ``defocusing'', and when $\kappa=-1$ it is called ``focusing''.

The study of the NLS and other nonlinear dispersive equations is a large and growing area in the analysis of PDE, with numerous open questions and conjectures. For a very readable general introduction, see Tao~\cite{tao06}. For a more specialized account of the state of affairs in the study of NLS, see the lecture notes of Rapha\"el \cite{raphael08}. The NLS arises in many areas of the pure and applied sciences, including Bose-Einstein condensation, Langmuir waves in plasmas, nonlinear optics, and a number of other fields \cite{ESY2, ESY,KSS,Z, BKA, FKM, W, rumpf04}.

The NLS is an infinite dimensional Hamiltonian flow, with Hamiltonian given by 
\begin{equation*}
H(v) = \frac{1}{2}\int_{\rr^d} |\nabla v(x)|^2dx + \frac{\kappa}{p+1}\int_{\rr^d} |v(x)|^{p+1} dx. 
\end{equation*}
(Note that in the focusing case ($\kappa=-1$), this is just the function $H$ defined in \eqref{energydef}.) Consequently, if $u$ is a solution to \eqref{nlsequation}, then $H(u(t,\cdot))$ is the same for all $t$. Since $H(v)$ is commonly called the energy of $v$ in the context of Hamiltonian flows, the previous sentence can be restated as: ``The NLS flow conserves energy''. Another important conserved quantity is the mass $M(v)$, defined in \eqref{massdef}.
 
A significant amount of information is known about the defocusing NLS; in particular, it is known that in many situations, solutions of the defocusing equation disperse like solutions of the linear Schr\"odinger equation (see \cite[p.~154]{tao06}). Here ``dispersion'' means that while $M(u(t,\cdot))$ remains conserved, for every compact set $K\subseteq \rr^d$,
\[
\lim_{t\ra\infty} \int_K |u(x,t)|^2 dx = 0. 
\]
In the focusing case, however, dispersion may not occur. This is demonstrated quite simply by a special class of solutions called ``solitons'' or ``standing waves''. These are solutions of the form $u(x,t) = v(x)e^{\I \omega t}$, where $\omega$ is a positive constant and the function $v$ is a solution of  the soliton equation 
\begin{equation}\label{solitonequation}
-\omega v  = -\Delta v - |v|^{p-1} v. 
\end{equation}
Often, the function $v(x)$  is also called a soliton. Such functions are known to be smooth and exponentially rapidly decreasing (see e.g. \cite[Section 8.1]{cazenave89}), and if one makes the further assumption that $v$ is non-negative and spherically symmetric then there is a unique solution to \eqref{solitonequation} for each $\omega > 0$ \cite{coffman72, strauss89, bl79}; we refer to this $v$ as the ``ground state''. There also exist radial solutions which change sign, see \cite{bgk83}; such solutions are called ``excited states''.

The focusing equation is said to have mass-subcritical nonlinearity if the nonlinearity parameter $p$ satisfies the subcriticality condition 
\[
1<p< 1+\frac{4}{d}.
\]
Mass-subcritical nonlinearity has important consequences. For instance, if $p < 1+4/d$, then it is easy to show that all solutions with initial data in $H^1$ are global and bounded in $H^1$ (see \cite[Section 1.1]{raphael08}). Another important feature of this regime is that for any $m > 0$, 
\begin{equation}\label{emindef}
E_{\min}(m) := \inf_{v\;:\; M(v) = m} H(v) \in ( -\infty, 0),
\end{equation}
and the infimum is achieved at the ground state soliton with mass $m$. In fact, it is simple to prove by a scaling argument that when $p < 1+4/d$, the function $E_{\min}$ has  the form 
\begin{equation}\label{eminform}
E_{\min}(m)= m^\alpha E_{\min}(1),
\end{equation}
where $E_{\min}(1) \in (-\infty,0)$ and $\alpha>1$ is a constant that is explicitly determined by $p$ and $d$ (see \cite[Section 1.4]{raphael08}). 
The infimum is achieved uniquely: any energy minimizing function $v$ must be of the form 
\[
v(x) = Q_{\lambda(m)}(x-x_0) e^{\I \gamma_0},
\]
where $x_0\in \rr^d$ and $\gamma_0\in \rr$, and $Q_{\lambda(m)}$ is the unique ground state soliton with mass $m$. The ground state soliton $Q_{\lambda(m)}$ has the following explicit form: Let $Q$ be the unique positive and radially symmetric solution of the equation 
\begin{equation}\label{solitondefine}
-Q = -\Delta Q - |Q|^{p-1} Q.
\end{equation}
For each $\lambda> 0$, let 
\begin{equation}\label{solitonscaling}
Q_\lambda(x) := \lambda^{2/(p-1)}Q(\lambda x). 
\end{equation}
Then for  each $m >0$, there is a unique $\lambda(m) > 0$ such that $Q_{\lambda(m)}$ is the ground state soliton of mass $m$. For all of the above claims about ground state solitons in the mass-subcritical regime, see \cite[Sections 1.2 and 1.3]{raphael08}.  The uniqueness of the ground state is a deep result. See  \cite[Appendix B]{tao06} for details. 

When $p \ge 1+4/d$, much less is known; it is currently an area of active research (see \cite{kenigmerle06, killipvisan10} for recent developments and pointers to the literature). 

Even in the mass-subcritical case, little is known about the long-term behavior of solutions. One particularly important conjecture, sometimes called the ``soliton resolution conjecture'' (see Tao~\cite[p.~154]{tao06}), claims (vaguely) that as $t\ra\infty$, the solution $u(\cdot, t)$ would look more and more like a soliton, or a union of a finite number of receding solitons. The claim may not hold for all initial conditions, but is expected to hold for ``most'' (i.e.\ generic) initial data. In the critical and supercritical regimes, the conjecture is still supposed to be true, but with the additional imposition that the solution does not blow up. The conjecture is based mainly on numerical simulations, although there has been a limited amount of progress towards a proof (see \cite{nakanishischlag11, soffer06, tao04,  tao07, tao09} and references therein). The only case where one can give a heuristic treatment is when $d=1$ and $p=3$, where the NLS is completely integrable (see \cite{segurablowitz, novoksenov80, zakharovshabat72}). The soliton resolution conjecture has been investigated for other dispersive systems, with partial results~\cite{miura76, schuur86, segur73,  eckhaus85, eckhausschuur83}. For significant recent progress on the soliton resolution conjecture for the energy-critical wave equation and a far more extensive survey of the literature around the conjecture, see~\cite{dkm}.

\subsection{Invariant measures for the NLS}
One approach to understanding the long-term behavior of global solutions is through the study of invariant Gibbs measures. Roughly, the idea is as follows. Since the NLS is a Hamiltonian flow, one might expect by Liouville's theorem that Lebesgue measure on the space of all functions of suitable regularity, if such a thing existed, would be an invariant measure for the flow (see, e.g., \cite[p.\ 68]{arnold89} for a statement of Liouville's theorem in the finite dimensional setting). Since the flow preserves energy, this would imply that Gibbs measures that have density  proportional to 
\begin{equation}\label{invmeasure0}
\exp(-\kappa\beta H(v))
\end{equation}
with respect to this fictitious Lebesgue measure (where $\beta$ is arbitrary) would also be invariant for the flow. One way to make this rigorous is to first restrict the system to the unit torus $\mathbb{T}^d$ and then consider Gibbs measures that have density proportional to 
\begin{equation}\label{invmeasure}
\exp\biggl(\kappa\beta\int_{\mathbb{T}^d} |v(x)|^{p+1} dx\biggr)
\end{equation}
with respect to the free-field Gaussian measure (see \cite{lmw11}) on the appropriate space of distributions on $\mathbb{T}^d$. This is the pioneering idea of Lebowitz, Rose and Speer \cite{lrs88}. For such a thing to make sense in $d\ge 2$, one has to interpret the integral in the Wick-ordered sense. 

These Gibbs measures exist for the defocusing case ($\kappa=1$) for all $p$ in $d = 1$ (without Wick ordering) and for $p \le 5$ in $d = 2$, and $p \le 3$ in $d = 3$~\cite{GJ}. Furthermore, despite the fact that this measure is supported on rough functions, Bourgain showed that it is invariant under the dynamics given by~\eqref{nlsequation} for $d \le 2$ \cite{bourgain96}. This means that the dynamics can be defined (after Wick ordering modification in $d = 2$) on a set of full measure with respect to this Gibbs measure.

The focusing case ($\kappa=-1$) is more delicate. Since $H$ is unbounded from below, it is obvious that the Gibbs measure cannot exist without some restrictions on its domain. It was shown in \cite{lrs88} that in $d = 1$, the Gibbs measure exists for $p = 3$ when restricted to $L^2$ balls and that it exists for $p = 5$ with the additional condition of small $\beta$. The development was continued by Bourgain \cite{bourgain94}, McKean \cite{mckean95}, McKean and Vaninsky in \cite{mckean94, mckean97a, mckean97b} and Zhidkov \cite{zhidkov91}. In $d = 2$, Jaffe showed that the measure exists for $p = 2$ for real $u$ when restricted to $L^2$ balls and after Wick ordering (see \cite{lmw11}); while Brydges and Slade \cite{brydgesslade96} showed that this does not work when $p = 3$.

Invariant measures coupled with Bourgain's development \cite{bourgain94, bourgain96, bourgain98, bourgain00} of the so-called $X^{s,b}$ spaces (``Bourgain spaces'') for constructing global solutions has led to important developments in this field. Recently, striking advances have been made by Tzvetkov and coauthors \cite{tzvetkov06, burqtzvetkov07, tzvetkov08, burqtzvetkov08, burqtzvetkov08b, tzvetkov10, thomanntzvetkov10} and Oh and coauthors \cite{collianderoh09, oh09, oh09b, oh10, oh-etal10, ohquastel10} and others (e.g.\ \cite{nahmod-etal11}) who use invariant measures and Bourgain's method to construct global solutions of the NLS and other nonlinear dispersive equations with random initial data. 

Qualitative features of the infinite volume limit of Gibbs measures were studied  by  Brydges and Slade \cite{brydgesslade96} in $d=2$ and  Rider \cite{rider02, rider03} in~$d=1$. Invariant Gibbs measures for the cubic discrete nonlinear Schr\"odinger equation (DNLS) in $d\ge 3$ were studied in \cite{ck10}. 

An idea that is gaining traction in the physics circles in recent years is that of considering microcanonical ensembles (see e.g.\ \cite{rumpfnewell01, rumpf04} and references therein). The general idea -- which has already been discussed at the beginning of this section -- is to consider an abstract manifold of functions satisfying certain constraints (usually two) and then trying to understand the characteristics of a function picked uniformly at random from this manifold. Often, the physicists alternately characterize the uniform distribution as the ``maximum entropy'' distribution. For example, in the context of the NLS, one looks at the ``uniform distribution'' on the space of all functions with a given mass and energy. The relevance of this to the long term behavior of NLS flows is heuristically justified through Liouville's theorem; we have more on this in the next section. It is in an attempt to understand these physical heuristics that I got interested in this line of research (and I thank Persi Diaconis -- who heard about it from Julien Barr\'e -- for communicating these problems to me a few years ago). In an early paper \cite{chatterjee10}, I tried to understand the behavior of functions chosen uniformly from all functions satisfying $\int |u(x)|^2dx = m$ and $\int |u(x)|^{p+1} dx = -E$, completely ignoring the gradient term in the Hamiltonian. Already in this simplified situation one can prove interesting phase transitions and localization phenomena. In a later paper with Kay Kirkpatrick \cite{ck10}, the gradient term was added to the analysis, but the nonlinearity parameter was taken to be so large that the gradient term became practically unimportant. In both \cite{chatterjee10} and~\cite{ck10}, the settings were discrete and too crude to allow passage to a continuum limit. The purpose of the current manuscript is to undertake the more serious task of analyzing regimes where the gradient term actually matters, and a continuum limit can be taken.

\section{The main result}\label{basic}
Assume that $\kappa = -1$ for the rest of this manuscript. Given $m > 0$ and $E> E_{\min}(m)$ (where $E_{\min}(m)$ is the minimum energy for mass $m$, as defined in \eqref{emindef}), let 
\begin{equation}\label{sem}
S(E,m) := \{v\in H^1(\rr^d): M(v)=m \text{ and } H(v)=E\}
\end{equation}
be the set of all $H^1$ functions of mass $m$ and energy $E$. 
Since the NLS flow~\eqref{nlsequation} preserves mass and energy, the same heuristic via Liouville's theorem that led to \eqref{invmeasure0} would imply that a ``uniform distribution'' on $S(E, m)$, if such a thing existed, would be an invariant measure for the flow. In physics parlance, these measures would be the ``Microcanonical Ensembles'' corresponding to the  ``Canonical Ensembles'' given by \eqref{invmeasure0}.

If the soliton resolution conjecture is indeed true, and an invariant measure like the microcanonical ensemble suggested  above indeed exists and describes the long-term behavior of the typical NLS flow with a given mass and energy, then it should put all its mass on soliton or multi-soliton functions. This may seem like a contradiction since such functions may not have energy $E$  required for membership in $S(E,m)$. However there is no actual contradiction since $S(E,m)$ is not compact under any reasonable metric.

Our goal is to go ahead and try to give a meaning to the abstract nonsense outlined above. To give a meaning to the notion of a uniform probability distribution on the set of all functions with a given mass and energy, we restrict ourselves first to a finite region of space, and then to a discretization of it. Instead of $\rr^d$, therefore, our space would be the discrete grid 
\[
V_n = \{0,1,\ldots, n-1\}^d = (\zz/n\zz)^d.
\]
We imagine this set embedded in $\rr^d$ as $hV_n$, where $h>0$ is a parameter representing the grid size. Note that  $hV_n$ is a discrete approximation of the box $[0,nh]^d$. We would eventually want to send $h$ to zero and $nh$ to $\infty$. 

The mass and energy of a function $v:V_n \ra \cc$ at grid size $h$ and box size $n$ are defined in analogy with \eqref{massdef} and \eqref{energydef} as 
\begin{equation}\label{mhndef}
M_{h,n}(v) := h^d\sum_{x\in V_n} |v(x)|^2,  
\end{equation}
and 
\begin{equation}\label{ehndef}
H_{h,n}(v) := \frac{h^d}{2} \sum_{x,y\in V_n\atop x\sim y} \biggl|\frac{v(x) - v(y)}{h}\biggr|^2 - \frac{h^d}{p+1}\sum_{x\in V_n} |v(x)|^{p+1},
\end{equation}
where $x\sim y$ means that $x$ and $y$ are neighbors in $V_n$. For simplicity, we endow $V_n$ with the graph structure of a discrete torus, i.e.\ identifying $n$ with~$0$. Let $\cc^{V_n}$ denote the set of all functions from $V_n$ into $\cc$. Take any $\ep>0$, $m > 0$ and $E\in \rr$, and define the set 
\begin{align}\label{ssem}
S_{\ep, h, n}(E, m) &:= \{v\in \cc^{V_n}: |M_{h,n}(v) - m|\le \ep, |H_{h,n}(v)-E|\le \ep\}. 
\end{align}
Clearly, $S_{\ep,h,n}(E,m)$ is a finite volume subset of the finite dimensional space $\cc^{V_n}$. This set is a ``manageable'' version of the set $S(E,m)$ defined in \eqref{sem}. Indeed, as $\ep \ra 0$, $h \ra 0$ and $nh\ra \infty$, the set $S_{\ep, h,n}(E,m)$ may be imagined as tending to the limit set $S(E,m)$. 

We have chosen $\ep >0$ to ensure that the volume is nonzero whenever the set is non-empty. In this situation, the uniform probability distribution on this set is well-defined. This uniform probability distribution, besides being an approximation to our abstract object of interest, has also a concrete interpretation as a natural invariant measure for an appropriate discrete NLS evolution on $V_n$, to be discussed in the next section.

Fix $E$ and $m$ such that $E> E_{\min}(m)$. Given  $\ep$, $h$ and $n$, let $\mu_{\ep, h,n}$ be the uniform probability distribution on $S_{\ep,h,n}(E,m)$. Let  $f_{\ep, h, n}$ be a random function on $V_n$ with law $\mu_{\ep,h,n}$. Our main result, stated below, is that when the nonlinearity is mass-subcritical,   the random function $f_{\ep, h, n}$ converges in a certain sense to  the unique ground state soliton $Q_{\lambda(m)}$ of mass $m$ defined in Section \ref{intro}, as $(\ep, h, nh)$ goes to $(0,0,\infty)$ in a certain manner.  

To define the notion of convergence, we first need a way of comparing functions on $V_n$ with functions on $\zz^d$ and $\rr^d$. Given $v:V_n \ra\cc$, first define its extension $v^e$ to $\zz^d$ by simply defining $v^e$ as equal to $v$ on $V_n$ and zero outside. Next, given a function $w:\zz^d \ra \cc$, define its ``continuum image at grid size $h$'' as the function $\tilde{w}:\rr^d \ra\cc$, defined as follows. Given $y= (y_1,\ldots, y_d)\in \rr^d$, let $x = (x_1,\ldots, x_d)$ be the unique point in $\zz^d$ such that for each $i$,
\[
x_i \le y_i/h < x_i+1, 
\]
and let $\tilde{w}(y) := w(x)$. (In other words, $x_i = [y_i/h]$.) Lastly, given $v:V_n\ra\cc$, define its continuum image $\tilde{v}:\rr^d \ra\cc$ at grid size $h$ as the function $\tilde{v}^e$, that is, the continuum image of the extended function $v^e$.

For each $q\in [1,\infty]$ define a pseudometric $\tl^q$ on the set of measurable complex-valued functions on $\rr^d$ as
\begin{equation}\label{tlq1}
\tl^q(u,v) := \inf_{x_0\in \rr^d, \; \alpha_0\in S^1} \|u(\cdot)- \alpha_0v(\cdot+x_0)\|_q, 
\end{equation}
where $S^1$ is the unit circle in the complex plane and $\|\cdot \|_q$ denotes the usual $L^q$ norm of a complex-valued function on $\rr^d$ with respect to Lebesgue measure. (Note that $L^\infty$ is the essential supremum norm and not the supremum norm.) This is a pseudometric since $\tl^q(u, v)$ may be zero even if $u$ and $v$ are not equal, but $v$ is of the form $v(x)= \alpha_0 u(x + x_0)$ 
for some $x_0\in \rr^d$ and $\alpha_0\in S^1$. It is necessary to work with pseudometrics since the law of $f_{\ep,h,n}$ is invariant under translations and  multiplication by scalars of unit modulus.
\begin{thm}\label{ourmain}
Suppose that $1< p < 1+4/d$. Fix $E$ and $m$ such that $E> E_{\min}(m)$, where $E_{\min}$ is defined in \eqref{emindef}. Let $f_{\ep, h,n}$ be a uniform random choice from the set $S_{\ep,h,n}(E,m)$ defined in \eqref{ssem}, and let $\tf_{\ep,h,n}$ denote its continuum image at grid size $h$, as defined above. Let $\tl^q$ be the pseudometric defined above. Then for any $\delta > 0$ and any $q\in (2,\infty]$, 
\[
\lim_{h\ra0} \limsup_{\ep\ra0} \limsup_{n \ra\infty} \pp(\tl^q(\tf_{\ep, h,n}, Q_{\lambda(m)}) > \delta ) = 0,
\]
where $Q_{\lambda(m)}$ is the unique ground soliton of mass $m$ defined in Section \ref{intro}. Moreover, there is a sequence $(\ep_k, h_k, n_k)$ with $\ep_k \ra 0$, $h_k \ra 0$ and $n_k h_k \ra \infty$ as $k\ra\infty$, such that  for any fixed $\delta > 0$ and $q\in (2,\infty]$,
\[
\lim_{k\ra\infty} \pp(\tl^q(\tf_{\ep_k, h_k, n_k}, Q_{\lambda(m)}) > \delta) =0.
\]
\end{thm}
Note that the energy of $f_{\ep,h,n}$ is converging to $E$, whereas the energy of $Q_{\lambda(m)}$ is $E_{\min}(m)$, which is strictly less than $E$. There is no contradiction here, since the metrics of convergence are  not strong  enough to ensure convergence of the Hamiltonian. Nor should they be,  since the difference $E-E_{\min}(m)$ denotes the amount of energy that has ``escaped'' to infinity when the NLS has flowed for a long (infinite) time. 

Let me emphasize here that Theorem \ref{ourmain} {\it does not}  say that a certain Gibbs measure concentrates on its lowest energy state in the infinite volume limit. In fact, the soliton $Q_{\lambda(m)}$ is not in the support of the measure $\mu_{\ep,h,n}$ at all, even in the limit. What Theorem \ref{ourmain} says is more subtle: In the infinite volume continuum limit, a typical function with mass $m$ and energy $E$ decomposes into an ``invisible'' or ``radiating'' part that is small in $L^\infty$ norm but contains a significant amount of energy due to microscopic fluctuations, and a ``visible part'' that is close to  the soliton $Q_{\lambda(m)}$ in the $L^\infty$ distance.

It may seem strange that while a certain amount of energy escapes to infinity, there is no escape of mass. Again there is no contradiction, since functions of arbitrarily small $L^2$ mass on $\rr^d$ can hold arbitrarily large amounts of energy by being very wiggly.

Theorem \ref{ourmain} does not model the full
dynamics of NLS; in particular, it does not model stable multisoliton solutions consisting of two or more receding solitons which do not collapse into a single ground state. This is possibly  because the
effect of recession ``outruns'' the thermodynamic convergence to equilibrium in the infinite volume setting, whereas multisoliton solutions eventually merge into a single soliton on the finite discrete torus considered in Theorem \ref{ourmain}.

Nevertheless, in the case of a finite discrete torus, Theorem \ref{ourmain} may be used to prove a version of the soliton resolution conjecture. This is the topic of the next section.

\section{Is this a proof of the soliton resolution conjecture on a large discrete torus with small grid size?}\label{srcsec}
A problem with the soliton resolution conjecture (SRC) is that its statement is not mathematically precise. The term `generic initial data' is particularly open to interpretation. For this reason, it may never be possible to completely settle the conjecture to everyone's satisfaction. Even so,  I will now make an attempt to prove a certain formulation of the conjecture on a large discrete torus. Whether this is actually a `correct' formulation may be a matter of contention.

Let all notation be as in Section \ref{basic}. Fix a positive integer $n$, and positive real numbers $h$ and $\ep$. As in Section \ref{basic}, let $V_n$ be the discrete torus $\{0,1,\ldots, n-1\}^d$. 
The discrete Laplacian on the torus $V_n$ with grid size $h$ is defined as
\begin{equation}\label{discretelap2}
\Delta v(x) := \frac{1}{h^2}\sum_{y\; :\;y\sim x} (v(y)-v(x)), 
\end{equation}
where $y\sim x$ denotes the sum over all neighbors of $x$ in  $V_n$, and $v$ is any complex-valued function on $V_n$. 

The discrete focusing nonlinear Schr\"odinger equation (DNLS) on $V_n$ with grid size $h$ and nonlinearity parameter $p$ is a family of coupled~ODEs
\begin{equation}\label{dnlseq2}
 \I \frac{du}{dt}    = - \Delta u - \abs{u}^{p-1} u,
\end{equation} 
where $\Delta$ is the discrete Laplacian defined above and $u(x,t)$ is a function on $V_n \times \rr$. The function $u_0(x)= u(x,0)$ is called the `initial data' for the flow. It is not difficult to show that the mass $M_{h,n}$ defined in \eqref{mhndef} and the energy $E_{h,n}$ defined in \eqref{ehndef} are conserved quantities for the DNLS flow.

The DNLS has been studied widely by physicists, but not so much by mathematicians, particularly in dimensions higher than one. For a recent survey of the mix of rigorous and non-rigorous results that exist in the literature, see \cite{kevrekidis09}. 

Let us formally christen the DNLS flow as $T_t$. That is, let us denote the function $u(\cdot, t)$ by $T_t u_0$. It is easy to establish by Picard iterations and the conservation of mass that for any $p> 1$, $T_t$ is a well-defined, one-to-one and continuous map for all $t\in \rr$ and satisfies $T_{t+s}=T_t T_s$; this is because in the discrete setting the right hand side of \eqref{dnlseq2} is a Lipschitz function of $u$ (under any reasonable metric) where the Lipschitz constant may be bounded by a function of the mass of $u$ and fixed quantities like $h$ and $n$. This is vastly simpler than the continuous case, where one has to use conservation of both mass and energy, together with the Gagliardo-Nirenberg inequality, to establish global well-posedness under the mass-subcriticality condition $p < 1+4/d$. This was proved by Ginibre and Velo \cite{GV} (see also \cite{K}).

Now fix some $m > 0$ and $E> E_{\min}(m)$.  Let $S = S_{\ep,h,n}(E,m)$ denote the set of all functions on $V_n$ with mass $\in [m-\ep, m+\ep]$ and energy $\in [E-\ep, E+\ep]$ at grid size $h$, defined in \eqref{ssem}.  By the Hamiltonian nature of the DNLS and the conservation of mass and energy, for any $t\in \rr$ the function $T_t$ maps $S$ onto itself, and the uniform probability distribution $\mu$ on $S$ is an invariant measure for $T_t$. 

Suppose that $p< 1+4/d$. Given any probability measure $\nu$ on $S$, we will say that {\it $\nu$ satisfies the soliton resolution conjecture (SRC) with error $\delta$} if 
\[
\nu\biggl\{f\in \cc^{V_n}:\limsup_{t\ra\infty} \frac{1}{t}\int_0^t 1_{\bigl\{\tl^\infty(\widetilde{T_s f},\, Q_{\lambda(m)})\,>\,\delta\bigr\}}\, ds < \delta \biggr\}=1, 
\]
where, as in Section \ref{basic}, $\widetilde{T_s f}$ denotes the continuum image of $T_sf$ at grid size~$h$, $\tl^\infty$ is the pseudometric defined in \eqref{tlq1}, and $Q_{\lambda(m)}$ is the ground state soliton of mass $m$. In other words, if the initial data is chosen according to the probability measure $\nu$ and $\nu$ satisfies SRC with a small error, then the DNLS flow will stay close to the ground state soliton `most of the time'.

 Let $\cp$ be the set of probability measures on $S$ endowed  with the usual weak* topology. Let $\mm\subseteq \cp$ be the set of all ergodic invariant probability measures of the map $T_1$ restricted to the set $S$. By standard results from ergodic theory and the Choquet representation theorem (see e.g.\ Remark~(2) following Theorem 6.10 in \cite{walters82}), there is a unique probability measure $\tau$ on $\mm$ such that the uniform distribution $\mu$ on $S$ may be represented as 
\[
\mu = \int_{\mm} \nu\, d\tau(\nu),
\]
in the sense that for all continuous $\phi: S\ra\rr$, 
\begin{equation}\label{ergdec}
\int_S \phi(f)\, d\mu(f) = \int_{\mm} \biggl(\int_S \phi(f) \, d\nu(f)\biggr) d\tau(\nu). 
\end{equation}
This measure $\tau$ may be called the `natural probability measure' on $\mm$. Intuitively, it chooses ergodic components proportional to their volume. 
\begin{thm}\label{src}
Suppose that $1< p< 1+4/d$. Fixing $\ep$, $h$ and $n$, let $T_t$ be the DNLS flow defined above. Fixing $E$ and $m$, let $\nu$ be a random ergodic invariant probability measure for $T_1$ chosen according to the `natural probability measure' $\tau$ on $\mm$. Then for any $\delta > 0$,
\[
\lim_{h\ra0}\limsup_{\ep\ra0} \limsup_{n\ra\infty} \pp(\textup{$\nu$ satisfies SRC with error $\delta$}) = 1. 
\]
\end{thm}
In other words, if $n$ is large and $\ep$ and $h$ are appropriately small, then nearly all ergodic components of the DNLS flow on $S$ satisfy SRC with small error. Theorem \ref{src} is proved in Section \ref{srcproof}.

\section{Microcanonical invariant measure for the discrete NLS}
The proof of Theorem \ref{ourmain} is based on a result for the discrete lattice with fixed grid size $h$. To state this result, we need some preparation. First, define the mass $M_h(v)$ and the energy $H_h(v)$ of a function $v:\zz^d \ra\cc$ at grid size $h$ just as in \eqref{mhndef} and \eqref{ehndef}, but after replacing $V_n$ by $\zz^d$. 

Given $h > 0$ and $m \ge 0$, let $E_{\max}(m, h)$ and $E_{\min}(m, h)$ denote the supremum and infimum of possible energies  of functions with mass~$m$ at grid size $h$.
\begin{thm}\label{elimitthm}
Suppose that $1<p < 1+4/d$. With the above definitions, for any $h > 0$ and $m >0$ we have $E_{\max}(m, h) = 2dm/h^2$ and $-\infty< E_{\min}(m,h)< 0$. Moreover, the function $E_{\min}$ satisfies for all positive $m$ and $m'$ the strict subadditive inequality $E_{\min}(m+m',h) < E_{\min}(m,h)+E_{\min}(m',h)$. Lastly,  $\lim_{h\ra0}E_{\min}(m, h)= E_{\min}(m)$ 
and the convergence is uniform over compact subsets of $(0,\infty)$. 
\end{thm}
The first couple of assertions of Theorem \ref{elimitthm} are proved in Section \ref{possible}. The subadditive inequality is proved in Section \ref{conccompsec}. The convergence argument is more complicated. It follows from Corollary \ref{elimit} in Section \ref{solitonlim}. The convergence is based on the convergence of discrete solitons to continuum solitons (see Theorem \ref{solitonconv} below). These scattered results are gathered into a formal proof of Theorem \ref{elimitthm} in Section \ref{threethm}.  

As in Section \ref{basic}, we define a set of pseudometrics on the space of all complex-valued functions on $\zz^d$. For any $q\in [1,\infty]$, let $\tl^{q}$ be the pseudometric on $\cc^{\zz^d}$ defined as 
\begin{equation}\label{tlq2}
\tl^{q}(u,v) := \inf_{x_0\in \zz^d, \; \alpha_0\in S^1} \|u(\cdot)-\alpha_0 v(\cdot + x_0) \|_{q},
\end{equation}
where $\|\cdot\|_{q}$ is the usual $L^q$ norm on $\cc^{\zz^d}$. 

Let $\ms(m,h)$ be the set of all functions $f$ with $M_h(f)=m$ and $H_h(f)=E_{\min}(m,h)$. The set $\ms(m,h)$ will be called the set of discrete ground state solitons with mass $m$ at grid size $h$. Note that, as in the continuum case, a simple Euler-Lagrange argument shows that any discrete ground state soliton must necessarily satisfy the discrete soliton equation
\[
-\omega v = -\Delta v - |v|^{p-1} v
\]
for some $\omega > 0$.

Unlike the continuum case, the discrete ground state soliton for a given mass may not be unique. However, we do know from the following theorem that $\ms(m,h)$ is non-empty and compact in the $\tl^{q}$-topologies. Not only that, the set $\ms(m,h)$ also has an analog of the so-called ``orbital stability'' property (see \cite[Section 1.3]{raphael08}) of continuum solitons: any function that has near-minimal energy must be nearly a soliton. The subadditive inequality from Theorem \ref{elimitthm}, together with the classical concentration-compactness technique (\cite{lions84}; see also \cite[Section 1.4]{raphael08}), is the key to the proof of this result. (Note that the subadditive inequality is trivial in the continuous case by the formula \eqref{eminform}.)  
\begin{thm}\label{orbital}
Suppose that $1< p < 1+4/d$. Let $\ms(m, h)$ be the set of  ground state solitons of mass $m$ at grid size $h$, as defined above. Then for any $m > 0$ and $h > 0$, $\ms(m,h)$ is non-empty. Moreover, for any sequence of functions $f_k$ such that $M_h(f_k) \ra m$ and $H_h(f_k) \ra E_{\min}(m,h)$, there is a sub-sequence $f_{k_j}$ and some $f\in \ms(m,h)$ such that $f_{k_j}$ converges to $f$ in the $\tl^{q}$ pseudometric for any $q\in [2,\infty]$. 
\end{thm}
The main argument for the proof of Theorem \ref{orbital}, using a discretization of the concentration-compactness method, is presented in Section~\ref{conccompsec}. The proof is completed in Section \ref{threethm}. 

What happens to $\ms(m,h)$ as $h$ tends to zero? The next theorem answers this question. As $h\ra0$, the set $\ms(m,h)$ shrinks to a single point, namely, the unique continuum ground state soliton $Q_{\lambda(m)}$ defined in Section \ref{intro}.  For related results on continuum limits of the discrete NLS in one dimension, see \cite{kirkpatrick-etal11}.  
\begin{thm}\label{solitonconv}
Suppose that $1< p< 1+4/d$. Let $m_k$ be a positive sequence converging to some $m>0$. Let $h_k$ be a positive sequence tending to zero. For each $k$, let $f_k$ be an element of $\ms(m_k, h_k)$. Let $\tf_k$ be the continuum image of $f_k$ at grid size $h$, as defined in Section \ref{basic}. Let $\tl^q$ be the pseudometric on $L^q(\rr^d)$ defined in Section \ref{basic}. Then for any $q\in [2,\infty]$, $\lim_{k\ra\infty} \tl^q(\tf_k, Q_{\lambda(m)})  =0$. 
\end{thm}
The above theorem is proved by showing that for small $h$, discrete ground state solitons can be approximated by smooth functions -- and then applying the orbital stability of continuum solitons \cite[Section 1.3]{raphael08}. To show that discrete ground state solitons can be approximated by smooth functions, one has to prove regularity estimates that do not blow up as $h\ra 0$. To achieve this, the route taken in this paper is to translate the proof of regularity of continuum solitons (as sketched in \cite[Proposition B.7]{tao06}) to the discrete setting and obtain ``$h$-free'' estimates. This translation necessitates the development of a slew of discrete harmonic analytic results, including fine properties of discrete Green's functions, discrete Littlewood-Paley decompositions, and a discrete Hardy-Littlewood-Sobolev inequality of fractional integration. The harmonic analytic tools are developed in Section~\ref{harmonic}, and the regularity of discrete solitons is worked out in Section \ref{solitonreg}. The convergence argument is presented in Section \ref{solitonlim}.

Next, recall the random function $f_{\ep,h,n}$ from Section \ref{basic}. Observe that by Liouville's theorem, the law of $f_{\ep, h, n}$ is an invariant measure for the DNLS flow \eqref{dnlseq2} at grid size $h$. In Theorem \ref{ourmain}, we saw what happens to $f_{\ep,h,n}$ as $(\ep, h,nh) \ra(0,0,\infty)$ in a certain manner. On the way to proving Theorem~\ref{ourmain}, we first investigate what happens to $f_{\ep, h,n}$ as $(\ep, n) \ra (0,\infty)$, fixing $h>0$. What happens is the following: with high probability, $f_{\ep, h,n}$ is close to a discrete ground state soliton of mass $m'$, where $m' \in [0,m]$ is determined in a complicated manner by $E$, $m$ and $h$. The following theorem makes this precise.
\begin{thm}\label{discretelimit}
Suppose that $1<p < 1+4/d$. Take any $E\in \rr$, $m >0$ and $h > 0$ such that $E_{\min}(m,h)< E< E_{\max}(m,h)$.  Let $f_{\ep, h,n}$ be a uniform random choice from the set $S_{\ep,h,n}(E,m)$ defined in \eqref{ssem}, and extend its domain to $\zz^d$ by defining it to be zero outside $V_n$. If $E < \frac{1}{2}E_{\max}(m,h)$,  then there exists a compact set $K \subseteq [0, m]$ such that for any $\delta > 0$ and any $q \in (2,\infty]$, 
\[
\lim_{\ep \ra0} \lim_{n\ra\infty}\pp\biggl(\inf_{m'\in K} \inf_{v\in\ms(m',h)}\tl^{q}(f_{\ep, h,n},\; v) > \delta \biggr) = 0. 
\]
Furthermore, the set $K$ can be described as follows. It is the set of all $m'\in [0, m]$ that maximize
\[
\log(m-m') - \Psi_d\biggl(\frac{2h^2 (E-E_{\min}(m',h))}{m-m'}\biggr),
\]
where $\Psi_d:\rr \ra [0, \infty]$ is the function 
\begin{equation}\label{psiddef}
\Psi_d(\alpha) =  \sup_{0< \gamma< 1} \int_{[0,1]^d}\log \biggl(1-\gamma + \frac{4\gamma}{\alpha} \sum_{i=1}^d \sin^2(\pi x_i) \biggr) dx_1\cdots dx_d
\end{equation}
for $\alpha\in (0,2d)$, $\Psi_d(\alpha)=\Psi_d(4d-\alpha)$ for $\alpha \in (2d, 4d)$, $\Psi_d(2d)=0$, and $\Psi_d(\alpha)=\infty$ for $\alpha \ge 4d$ and $\alpha\le 0$. Lastly, if $E\ge \frac{1}{2}E_{\max}(m, h)$, then for any $\delta > 0$,  
\[
\lim_{\ep \ra0} \lim_{n\ra\infty}\pp(\|f_{\ep,h,n}\|_{\infty} > \delta ) = 0. 
\]
\end{thm}
In a few words, the above theorem says the following: ``For the DNLS on a large torus, a typical function with a given mass and energy is close in the $L^\infty$ distance to a soliton with a (possibly) different mass and energy.''

The proof of Theorem \ref{discretelimit} is divided into several components. The first ingredient, proved in Section \ref{largedev}, is a large deviation principle for gradients of random functions. The variational problem related to this large deviation principle is dealt with in Section \ref{variational}. The proof of the main identity in Theorem \ref{discretelimit} is separated into two pieces: first an upper bound, and then a matching lower bound. The upper bound in the case $E< \frac{1}{2}E_{\max}(m,h)$ is proved in Section~\ref{upperbound}. The matching lower bound requires us to first prove the exponential decay of discrete solitons. This is done in Section \ref{expodecay}, mostly along the lines of the proof of exponential decay of continuum solitons \cite[Proposition B.7]{tao06}, but with the crucial difference that we now have to deal with discrete Green's functions. Using the information from Section \ref{expodecay}, the lower bound is proved in Section \ref{lowerbound}. Finally, the case $E\ge \frac{1}{2}E_{\max}(m,h)$ is handled in Section \ref{radiatingsec}. Everything is formally put together to complete the proof of Theorem \ref{discretelimit} in Section \ref{discretelimitproof}. 

Last of all, let us indicate how all this leads to the proof of Theorem~\ref{ourmain}. This is quite easy, given Theorems \ref{discretelimit} and \ref{solitonconv}. We just take the limit $h\ra0$ in the explicit formula given in Theorem \ref{discretelimit}. The main step is to prove that the set $K$ in Theorem \ref{discretelimit} shrinks to the singleton set $\{0\}$ as $h$ goes to $0$. This is done in Section \ref{variationalcont}. The convergence of discrete solitons to continuum solitons as given by Theorem~\ref{solitonconv} finishes the proof. This argument is formalized in Section \ref{ourmainproof}.

\section{Main ideas in the proof}
Let $f:\rr^d \ra \cc$ be a function ``uniformly chosen'' from the set of functions satisfying $M(f)=m$ and $H(f)=E$, whatever that means. We need to show that for any set $A$ of functions that do not contain the ground state soliton, the chance of $f\in A$ is zero. 

Take any $\delta > 0$ and let $V_\delta := \{x: |f(x)|\le \delta\}$. 
Then 
\[
\int_{V_\delta} |f(x)|^{p+1} dx \le \delta^{p-1}\int_{V_\delta} |f(x)|^2 dx \le \delta^{p-1} m. 
\]
Decompose $f$ as $u+v$, where $u=f1_{V_\delta}$ and $v= f1_{\rr^d \backslash V_\delta}$. The above inequality shows that when $\delta$ is close to zero, 
\[
H(u) \approx \frac{1}{2}\int_{\rr^d}|\nabla u(x)|^2 dx.
\] 
On the other hand 
\[
\mathrm{Vol}(\rr^d \backslash V_\delta) \le \frac{1}{\delta^2} \int_{\rr^d\backslash V_\delta} |f(x)|^2 dx \le \frac{m}{\delta^2}. 
\]
Let us refer to $v$ and $u$ as the ``visible'' and ``invisible'' parts of $f$. The last two inequalities show that:
\begin{itemize}
\item The visible part is supported on a finite volume set, whose size is controlled by $\delta$.
\item The energy of the invisible part is essentially the same as the $L^2$ norm squared of its gradient, times $1/2$. 
\end{itemize}
The game now is to compute $\pp(f\in A)$ by controlling the visible and invisible parts separately. The visible part, being supported on a ``small'' set, can be analyzed directly. For the invisible part, one has to develop joint large deviations for the mass and the gradient, since the nonlinear term is negligible in the invisible part.  Solving the variational problem related to the large deviation question, one arrives at the conclusion that the visible part must be close to the ground state soliton with high probability. 

The main steps in the above program are the following.
\begin{enumerate}
\item Develop large deviation estimates for the invisible part in the finite volume discrete case. This is done in Sections \ref{gaussian}, \ref{diagonal}, \ref{largedev}, \ref{upperbound}, \ref{expodecay}, \ref{lowerbound} and~\ref{radiatingsec}. 
\item Analyze the variational problem related to this large deviation question, and thereby show that with high probability, the visible part has the minimum possible energy for its mass. This is done in Sections \ref{possible}, \ref{variational} and \ref{variationalcont}. 
\item Pass to the infinite volume limit (keeping the grid size fixed) using a discretization of the classical concentration-compactness argument, and show convergence to discrete solitons. This is done in Sections \ref{conccompsec} and \ref{threethm}. 
\item Develop discrete analogs of harmonic analytic tools (Littlewood-Paley decompositions, Hardy-Littlewood-Sobolev inequality of fractional integration, Gagliardo-Nirenberg inequality, discrete Green's function estimates, etc.) to prove smoothness estimates for discrete solitons that do not blow up as the grid size $\to 0$. This is done in Sections~\ref{harmonic} and \ref{solitonreg}. 
\item Use these smoothness estimates, together with the orbital stability of the ground state soliton, to prove convergence of discrete solitons to continuum solitons. This is done in Section \ref{solitonlim}. 
\end{enumerate}

\section{Summary of notation}\label{notation}
In this section we summarize the notation that will be used repeatedly in this manuscript. Some of it has already been introduced, and some will be  defined in later sections. The summary in this section is  for the reader's convenience. 
\subsection{Spaces and norms}
For a typical element $x\in\rr^d$, we denote the $i$th coordinate of $x$ by $x_i$. The usual Euclidean norm of a vector $x\in \rr^d$  is denoted by $|x|$, while the $\ell^1$ norm of $x$ is denoted by $|x|_1$. The same notation is used for norms of vectors in $\zz^d$.

The $L^q$ norms for functions on $\rr^d$ and $\zz^d$ are defined as usual, and they induce the pseudometrics $\tl^q$ defined in \eqref{tlq1} and \eqref{tlq2}. The $L^q$ norm of a function $v$, whether on $\rr^d$ or on $\zz^d$, is denoted by $\|v\|_q$. 

Sometimes, we use a slightly different $L^q$ norm for functions on $\zz^d$ by combining the usual $L^q$ norm with the grid size $h$ to get an $L^{q,h}$-norm:
\[
\|v\|_{q,h} := h^{d/q}\|v\|_q.
\]
These norms will be used heavily in Sections \ref{harmonic} and \ref{solitonreg}. 

\subsection{The discrete torus}
Assuming that the dimension $d$ is fixed, the torus $V_n$ in $\zz^d$ is the set $\{0,1,\ldots, n-1\}^d$. We say two elements $x$ and $y$ in $V_n$ are neighbors, and write $x\sim y$, if $|x-y|=1$, where the difference $x-y$ is computed by subtraction modulo $n$ in each coordinate.

Sometimes, we use $\partial V_n$ to denote the boundary of the torus when considered as a subset of $\zz^d$ (without the toric graph structure). 
In general $\partial U$ denotes the boundary of a set $U\subseteq \zz^d$ or $U\subseteq V_n$. That is, $\partial U$ is the set of points in $U$ that are adjacent to some point outside $U$. Note that this boundary is different if $U$ is considered as a subset of the torus $V_n$ rather than as a subset of $\zz^d$. Similarly, $U^c$ denotes the set $\zz^d \backslash U$ when $U$ is a subset of $\zz^d$, whereas $U^c$ denotes $V_n\backslash U$ when $U$ is considered as a subset of the torus $V_n$.

\subsection{Mass and energy}
The mass of a function $v:\rr^d \ra \cc$ is defined as
\[
M(v) = \int_{\rr^d} |v(x)|^2 dx,
\]
and its energy, in the context of the NLS equation \eqref{nlsequation} with $\kappa=-1$, is defined as
\[
H(v) = \frac{1}{2}\int_{\rr^d} |\nabla v(x)|^2dx - \frac{1}{p+1}\int_{\rr^d} |v(x)|^{p+1} dx. 
\]
When $v$ is a function on $\zz^d$, its mass ``at grid  size $h$'' is defined as
\[
M_h(v) = h^d \sum_{x\in \zz^d} |v(x)|^2.
\]
Similarly, the energy at grid size $h$ is defined as 
\begin{equation}\label{energydefi}
H_h(v) = \frac{h^d}{2}\sum_{x,y\in \zz^d\atop x\sim y}\biggl|\frac{v(x)-v(y)}{h}\biggr|^2 - \frac{h^d}{p+1}\sum_{x\in \zz^d} |v(x)|^{p+1}. 
\end{equation}
For a function $v$ on the torus $V_n$, the mass $M_{h,n}(v)$ and the energy $H_{h,n}(v)$ are defined exactly as in the last two displays, except that the sums are now over elements of $V_n$ instead of $\zz^d$. 

For a function $v:\zz^d \ra\cc$, the energy $H_h(v)$ can be decomposed into the gradient component $G_h(v)$ and the nonlinear component $N_h(v)$, defined as the first and second terms on the right-hand side in \eqref{energydefi}, so that $H_h(v) = G_h(v)-N_h(v)$. Similarly, we define $G_{h,n}(v)$ and $N_{h,n}(v)$.

Given a subset $U\subseteq V_n$, we define 
\[
G_{h,n}(v, U)= \frac{h^d}{2}\sum_{x,y\in U\atop x\sim y}\biggl|\frac{v(x)-v(y)}{h}\biggr|^2, 
\]
and $H_{h,n}(v, U)$, $N_{h,n}(v, U)$ and $M_{h,n}(v,U)$ are defined similarly.

\subsection{Maximum and minimum energies}
Suppose that $p$ and $d$ are given. For each $m\ge 0$, we define $E_{\max}(m)$ and $E_{\min}(m)$ to be the supremum and infimum of the set of all possible energies for functions with a given mass on $\rr^d$, where mass and energy are defined as above. It is not difficult to verify that $E_{\max}(m)=\infty$, and it is known that $E_{\min}(m)$ has the form \eqref{eminform} when $p< 1+4/d$. 

In the discrete case on $\zz^d$, we define $E_{\max}(m,h)$ and $E_{\min}(m,h)$ to be supremum and infimum of the set of all possible energies for functions with mass $m$, where both mass and energy are computed at grid size $h$. We prove a number of things about these quantities in Section \ref{possible}. 

Finally, for functions on the torus $V_n$, we similarly define $E_{\max}(m,h,n)$ and $E_{\min}(m,h,n)$.

Two related functions $E^+$ and $E^-$ are defined  Section \ref{variational}. 

The set of functions on $\zz^d$ with mass $m$ at grid size $h$ that minimize energy is denoted by $\ms(m,h)$. 

\subsection{Notation related to the variational problem}
The function $\Psi_d$ defined in \eqref{psiddef} is of fundamental importance in this manuscript. Two other functions, $\Theta$ and $\widehat{\Theta}$, and two related subsets of $\rr^2$, $\mr(E,m,h)$ and $\mm(E,m,h)$, are defined in the beginning of Section \ref{variational}. All of these are used repeatedly in the manuscript.

\subsection{Continuum image of a function}
Given a function $v:\zz^d \ra\cc$, its ``continuum image at grid size $h$'' is defined in Section \ref{basic}, but let us repeat the definition here. The continuum image at grid size $h$ is a function $\tilde{v}:\rr^d\ra\cc$ defined as follows. Given $y= (y_1,\ldots, y_d)\in \rr^d$, let $x = (x_1,\ldots, x_d)$ be the unique point in $\zz^d$ such that for each $i$,
\[
x_i \le y_i/h < x_i+1,
\]
and let $\tilde{v}(y) := v(x)$. When $v$ is a function on $V_n$, we define the continuum image by first defining the function to be zero on $\zz^d\backslash V_n$, and then defining the continuum image as above.

\section{Comparison with a Gaussian function}\label{gaussian}
Let $\phi = (\phi(x))_{x\in V_n}$ be a collection of i.i.d.\ complex Gaussian random variables, with $\ee(\phi(x))=\ee(\phi(x)^2)=0$ and 
\[
\ee|\phi(x)|^2 = \frac{1}{(nh)^d}.
\]
Recall the function $f_{\ep, h,n}$ and the set $S_{\ep,h,n}(E,m)$ defined in Section \ref{basic}. 
The following basic lemma connects the properties of $\phi$ with that of $f_{\ep, h, n}$. 
\begin{lmm}\label{basiclmm}
Take any $m > 0$ and $E\in \rr$. Let $f = f_{\ep,h,n}$ and $S= S_{\ep, h,n}(E, m)$ for simplicity. Then for any $A\subseteq S$, 
\[
\pp(f\in A) \le e^{2n^d\ep} \frac{\pp(\phi\in A)}{\pp( \phi \in S)}.
\]
\end{lmm}
\begin{proof}
For any measurable $A\subseteq \cc^{V_n}$, 
\begin{align*}
\pp(f\in A) &= \frac{\vol(A)}{\vol(S)}. 
\end{align*}
Now, $|M_{h,n}(v)- m|\le \ep$ when $v\in S$. Therefore, 
\begin{align*}
\pp(\phi\in S) &= \int_S \frac{(nh)^{dn^d}e^{-n^dM_{h,n}(v)}}{\pi^{n^d}} dv\\
&\le (n^dh^d \pi^{-1})^{n^d} e^{-n^d(m-\ep)}\vol(S). 
\end{align*}
Similarly,
\begin{align*}
\pp(\phi\in A) &= \int_A \frac{e^{-n^dM_{h,n}(v)}}{\pi^{n^d} h^{-dn^d/2}} dv\\
&\ge (n^dh^d \pi^{-1})^{n^d} e^{-n^d(m+\ep)}\vol(A). 
\end{align*}
This completes the proof. 
\end{proof}

\section{Diagonalizing the Laplacian}\label{diagonal}
Let $\Gamma = (\Gamma(x,y))_{x, y\in V_n}$ be the matrix defined as 
\[
\Gamma(x,y) =
\begin{cases}
2d &\text{ if } x=y,\\
-1 &\text{ if } x\sim y,\\
0 &\text{ in all other cases.}
\end{cases}
\]
The matrix $\Gamma$ may be viewed as an operator acting on $\cc^{V_n}$ in the natural sense. The action of $\Gamma$ on a function $f:V_n \ra\cc$ will be denoted by $\Gamma f$. 
Notice that $\Gamma = - h^2 \Delta$, where $\Delta$ is the discrete Laplacian on the torus~$V_n$ with grid size $h$, as defined in \eqref{discretelap2}, and  that $\Gamma$ is a real  symmetric matrix of order $n^d$. In this section, we will write down the spectral decomposition of~$\Gamma$. 

For two functions $u,v\in \cc^{V_n}$, let $(u,v)$ be the standard inner product, 
\[
(u,v) := \sum_{x\in V_n} u(x)\overline{v(x)}, 
\]
where $\overline{v(x)}$ is the complex conjugate of $v(x)$. Notice that for any $v$,
\[
(v,\Gamma v) = (\Gamma v, v) = \sum_{x,y\in V_n\atop x\sim y} |v(x)-v(y)|^2.
\]
We use the notation $x_i$ to denote the $i$th coordinate of a vector $x\in \rr^d$. 
\begin{lmm}\label{diaglmm}
For each $y = (y_1,\ldots, y_d)\in V_n$, let $\rho_y$ be the function
\[
\rho_y(x) := n^{-d/2}e^{\I 2\pi (y_1x_1+\cdots +y_dx_d)/n}.
\]
Then the functions $(\rho_y)_{y\in V_n}$ form a complete orthonormal system of eigenfunctions of $\Gamma$, and the eigenvalue corresponding to $\rho_y$ is 
\[
\lambda_y := 4\sum_{i=1}^d\sin^2(\pi y_i/n). 
\] 
\end{lmm}
\begin{proof}
To prove orthogonality, first notice that for any $k\in \zz$, $r := e^{\I 2\pi k/n}$ is an $n$th root of unity, and hence
\[
\sum_{j=0}^{n-1} e^{\I 2\pi kj/n} = \sum_{j=0}^{n-1} r^j = 
\begin{cases}
0 &\text{ if } k\ne 0,\\
n &\text{ if } k=0.
\end{cases}
\]
Thus, for any $y,y'\in V_n$, 
\begin{align*}
(\rho_y, \rho_{y'}) &= n^{-d} \sum_{x_1,\ldots, x_d = 0}^{n-1} e^{\I 2\pi ((y_1-y_1')x_1+\cdots +(y_d-y_d')x_d)/n}\\
&= n^{-d} \prod_{i=1}^d\sum_{x_i = 0}^{n-1} e^{\I 2\pi (y_i-y_i')x_i/n}\\
&= 
\begin{cases}
0 &\text{ if } y\ne y',\\
1 &\text{ if } y=y'. 
\end{cases}
\end{align*}
To show that $\rho_y$ is an eigenfunction of $\Gamma$ with eigenvalue $\lambda_y$, note that for any $x\in V_n$,
\begin{align*}
\Gamma \rho_y(x) &= \sum_{z \; : \; z\sim x} (\rho_y(x) - \rho_y(z)) \\
&= \sum_{i=1}^d (2 - e^{\I 2\pi y_i/n} - e^{-\I 2\pi y_i/n} ) \rho_y(x)\\
&= -\rho_y(x) \sum_{i=1}^d (e^{\I \pi y_i/n} - e^{-\I \pi y_i/n})^2 = 4\rho_y(x) \sum_{i=1}^d \sin^2(\pi y_i/n). 
\end{align*}
This completes the proof of the lemma.
\end{proof}

Let $R\in \cc^{V_n\times V_n}$ be the matrix whose $y$th column is $\rho_y$, for each $y\in V$. Note that $R$ is a unitary matrix. Let $\Lambda\in \cc^{V_n\times V_n}$ be the diagonal matrix whose $y$th diagonal element is $\lambda_y$. Then 
\[
\Gamma = R\Lambda R^*,
\]
where $R^*$ is the adjoint of $R$. This is the spectral decomposition of $\Gamma$.

\section{The set of possible energies for a given mass}\label{possible}
Let $E_{\min}(m, h, n)$ and $E_{\max}(m,h,n)$ be the minimum and maximum possible energies at grid size $h$ of a function $f: V_n \ra\cc$ with mass $m$. Since the map $f\mapsto H_{h,n}(f)$ is continuous and $\{f: M_{h,n}(f) = m\}$ is a compact connected subset of $\cc^{V_n}$, therefore for every $E\in [E_{\min}(m, h, n), E_{\max}(m,h,n)]$, there exists some $f$ on $V_n$ with mass $m$ and energy $E$ at grid size $h$. Recall that $E_{\min}(m,h)$ and $E_{\max}(m,h)$ are the infimum and supremum of the set of possible energies of functions with mass $m$ on the full lattice $\zz^d$, at grid size $h$. 
\begin{lmm}\label{emfacts3}
For any $m\ge0$, and any $E\in (E_{\min}(m,h), E_{\max}(m,h))$, there is a function $f:\zz^d\ra\cc$ such that $M_h(f)=m$ and $H_h(f)=E$. Moreover,  for any $m\ge 0$, any function on $\zz^d$ with mass $m$ and energy $E$ (at grid size $h$) satisfies 
\[
|E|\le C(p,d,h)(m + m^{(p+1)/2}).
\] 
\end{lmm}
\begin{proof}
If $M_h(f)=m$, then for any $x$, $|f(x)|^2\le m h^{-d}$. Therefore,  
\[
|f(x)|^{p+1} \le (mh^{-d})^{(p-1)/2}|f(x)|^2.
\]
Thus,  
\begin{align*}
H_h(f) &\le 2dh^{d-2}\sum_{x\in \zz^d} |f(x)|^2 + \frac{(mh^{-d})^{(p-1)/2}h^d}{p+1}\sum_{x\in \zz^d} |f(x)|^2\\
&= 2dh^{-2} m + \frac{h^{-d(p-1)/2}m^{(p+1)/2}}{p+1}. 
\end{align*}
This proves the inequality. 

Next, take any two functions $f,g\in \cc^{\zz^d}$ with $M_h(f)=M_h(g)=m>0$. If $f=g$ or $f=-g$, then $H_h(f)=H_h(g)$. Suppose that $f \ne \pm g$. For each $\theta\in [0,1]$, let $f_\theta:= \theta f+(1-\theta)g$. Since $M_h(f)=M_h(g)$ and $f\ne -g$, it is easy to see that $f_\theta$ is not the zero function for any $\theta$. In particular, $M_h(f_\theta) > 0$. Let 
\[
w_\theta(x) := \sqrt{\frac{m}{M_h(f_\theta)}} f_\theta(x)  \ \ \text{ for } x\in \zz^d. 
\]
Then $M_h(w_\theta) = m$ for all $\theta$. Moreover, it is easy to prove that $H_h(w_\theta)$ varies continuously from $H_h(g)$ to $H_h(f)$ as $\theta$ varies from $0$ to $1$.  This shows that for every $E\in (E_{\min}(m,h), E_{\max}(m,h))$, there is a function $f$ with $M_h(f)=m$ and $H_h(f)=E$. 
\end{proof}

\begin{lmm}\label{emfacts4}
As $n$ goes to infinity, $E_{\min}(m,h,n)$ tends to $E_{\min}(m,h)$ and $E_{\max}(m,h,n)$ tends to $E_{\max}(m,h)$. Moreover, in both cases, the convergence is uniform on compact subsets of $[0,\infty)$ (for the parameter $m$, keeping $h$ fixed). The functions $E_{\max}(\cdot, h)$ and $E_{\min}(\cdot, h)$ are absolutely continuous on~$[0,\infty)$. 
\end{lmm}
\begin{proof}
Take any $\ep > 0$. Let $f\in \cc^{\zz^d}$ be a function such that $M_h(f)=m$ and $H_h(f) \le E_{\min}(m,h) + \ep$. For each $n$, define a function $f_n$ on $V_n$ as simply the restriction of $f$ on $V_n$. Then it is easy to see that $M_{h,n}(f_n) \ra M_h(f)$ and $H_{h,n}(f_n) \ra H_h(f)$ as $n \ra\infty$.  Since $\ep$ is arbitrary, this shows that
\[
\limsup_{n\ra\infty} E_{\min}(m,h,n) \le E_{\min}(m,h). 
\]
Fix $n$ and a function $f\in \cc^{V_n}$  such that $M_{h,n}(f)=m$ and $H_{h,n}(f) \le E_{\min}(m,h,n) + \ep$. Define a function $g$ on $\zz^d$ as follows. For each $x\in V_n$, let $f_x$ be the translated function $f_x(y) := f(x+y)$, where the addition is modulo $n$ in each coordinate. Let $\partial V_n$ denote the boundary of $V_n$ when $V_n$ is considered as a subset of $\zz^d$. Recall the notation $M_{h,n}(f,U)$ from Section~\ref{notation}. Then 
\begin{align*}
 \sum_{x\in V_n} M_{h,n}(f_x, \partial V_n) &= \sum_{x\in V_n} \sum_{y\in \partial V_n} h^d |f(x+y)|^2 \\
&= \sum_{z\in V_n} h^d |f(z)|^2 |\partial V_n| = M_{h,n}(f)|\partial V_n|. 
\end{align*}
This shows that there exists $x\in V_n$ such that 
\[
M_{h,n}(f_x, \partial V_n) \le M_{h,n}(f)|\partial V_n|n^{-d} \le C(d, m) n^{-1}.
\]
Take such an $x$ and define $g:\zz^d \ra \cc$ as 
\[
g(y) :=
\begin{cases}
f_x(y) &\text{ if } y\in V_n,\\
0 &\text{ otherwise.}
\end{cases} 
\]
Clearly, $M_h(g)=m$. Since $x$ was chosen so that the `boundary effect' is small, $|H_h(g)- H_{h,n}(f_x)|\le C(d,m,h)n^{-1}$. Since $\ep$ is arbitrary and $H_{h,n}(f_x)= H_{h,n}(f)\le E_{\min}(m,h,n) + \ep$, this shows that 
\[
\liminf_{n\ra\infty} E_{\min}(m, h,n) \ge E_{\min}(m,h). 
\]
The proof for $E_{\max}$ is similar. 

To prove uniform convergence on compact subsets of $[0,\infty)$, we will first prove it for compact subsets of $(0,\infty)$. We will show that the collection of functions $(E_{\min}(\cdot, h,n))_{n\ge 1}$ is  equi-Lipschitz continuous on any compact subinterval of $(0,\infty)$ and apply the Arzela-Ascoli theorem. This will also prove absolute continuity of the function $E_{\min}$ on $(0,\infty)$. Continuity at zero follows from Lemma \ref{emfacts3}. The result for $E_{\max}$ follows similarly.

Take any $n$ and $0 < a < b$. Take $a \le m' < m\le b$. Let $f$ be an energy minimizing function (at grid size $h$) of mass $m$ on the torus $V_n$. Let $f' := (m'/m)^{1/2} f$. Then $M_{h,n}(f') = m'$. Note that 
\begin{align*}
H_{h,n}(f') = (m'/m)G_{h,n}(f) + (m'/m)^{(p+1)/2} N_{h,n}(f). 
\end{align*}
Thus,
\begin{align*}
&|H_{h,n}(f')-H_{h,n}(f)|\\
&\le |(m'/m)-1| G_{h,n}(f) + |(m'/m)^{(p+1)/2} -1 | N_{h,n}(f). 
\end{align*}
Since $N_{h,n}(f)$ and $G_{h,n}(f)$ can both be bounded above by 
\[
C_1(p,d,h) m^{C_2(p,d,h)},
\]
this shows that 
\[
E_{\min}(m', h,n)\le E_{\min}(m,h,n) + C(a,b, p,d,h)|m'-m|.
\]
Similarly, taking $g'$ such that $M_{h,n}(g') = m'$ and $H_{h,n}(g')= E_{\min}(m', h,n)$, and letting $g := (m/m')g'$, it follows that $E_{\min}(m, h,n)$ is bounded above by $E_{\min}(m', h,n) + C(a,b, p,d,h) |m'-m|$. This proves uniform convergence on compact subsets of $(0,\infty)$. To prove uniform convergence on compact subsets of $[0,\infty)$, one simply notices that the uniform bound on the energy given in Lemma \ref{emfacts3} holds for functions on $V_n$ as well (the bound will be same, independent of $n$).   
\end{proof}
Recall that we say $a_j\sim b_j$ as $j \ra\infty$ if $a_j$ and $b_j$ are sequences that satisfy $\lim_{j\ra\infty}a_j/b_j = 1$. 
\begin{lmm}\label{riem}
Let $c_j$ be a sequence such that for some $C>0$ and $\alpha\ge 0$, $c_j \sim Cj^\alpha$ as $j\ra\infty$. Then for any $\beta \ge 0$, 
\[
\sum_{j=0}^{k-1} c_j (k-j)^\beta \sim C k^{\alpha+\beta+1} \int_0^1 x^\alpha(1-x)^\beta dx \ \ \text{ as } k\ra\infty. 
\]
\end{lmm}
\begin{proof}
Writing 
\[
\sum_{j=0}^{k-1} c_j (k-j)^\beta = k^{\alpha+\beta+1}\frac{1}{k} \sum_{j=0}^{k-1} (c_j/k^\alpha) (1-j/k)^\beta,
\]
note that 
\begin{align*}
&\biggl|\frac{1}{k} \sum_{j=0}^{k-1} (c_j/k^\alpha) (1-j/k)^\beta - \frac{C}{k} \sum_{j=0}^{k-1} (j/k)^\alpha (1-j/k)^\beta\biggr|\\
&\le \frac{1}{k} \sum_{j=0}^{k-1} \frac{|c_j- C j^\alpha|}{k^\alpha}\\
&\le \frac{1}{k} \sum_{j=0}^{k-1} \frac{|c_j - C j^\alpha|}{(j+1)^\alpha}. 
\end{align*}
Since $c_j/Cj^\alpha \ra 1$ as $j \ra\infty$, the above bound tends to zero as $k \ra\infty$. Riemann sum approximation  gives
\[
\lim_{k\ra\infty}\frac{C}{k} \sum_{j=0}^{k-1} (j/k)^\alpha (1-j/k)^\beta = C\int_0^1 x^\alpha(1-x)^\beta dx.
\]
This completes the proof. 
\end{proof}

\begin{lmm}\label{emin0}
Suppose that $1<p < 1+4/d$. Then for any $m> 0$ and $A>0$, 
\[
\sup_{0< h\le A} E_{\min}(m,h) < 0. 
\]
\end{lmm}
\begin{proof}
Fix $m > 0$. For each positive integer $k$, define a function $f_k$ as follows. If $|x|_1 \ge k$, let $f_k(x) = 0$. If $|x|_1 < k$, let 
\[
f_k(x) = A_k(k-|x|_1),
\]
where $A_k$ is a positive constant such that $M_h(f_k)=m$. Since the number of vertices $x$ with $|x|_1 = j$ asymptotes to  $C(d)j^{d-1}$ as $j \ra\infty$, it follows from Lemma \ref{riem} that 
\[
A_k^2 \sim C(d, m)h^{-d}k^{-(d+2)} \  \text{ as } k\ra\infty. 
\]
Note that if $x$ and $y$ are neighboring points, then $f_k(x)\ne f_k(y)$ if and only if one of them has $\ell^1$ norm $j$ and the other has $\ell^1$ norm $j+1$, for some $0\le j< k$. And in that case,
\[
|f_k(x)-f_k(y)| = A_k. 
\]
There are $\sim C(d)j^{d-1}$ such pairs as $j\ra\infty$. Thus, again by Lemma \ref{riem}, 
\begin{align*}
G(f_k) &\sim C(d)h^{d-2}A_k^2 \sum_{j=0}^{k-1} j^{d-1} \sim C(d,m)h^{d-2}A_k^2 k^d\sim C(d,m)(hk)^{-2}. 
\end{align*}
Again by Lemma \ref{riem}, 
\begin{align*}
h^d\sum_{x\in \zz^d} |f_k(x)|^{p+1}  &= h^d\sum_{j=0}^{k-1} \sum_{x\; : \; |x|_1 = j} |f_k(x)|^{p+1}\\
&\sim C(d)h^dA_k^{p+1}\sum_{j=0}^{k-1}j^{d-1} (k-j)^{p+1}\\
&\sim C(p,d,m)h^{d-d(p+1)/2}k^{-(d+2)(p+1)/2}k^{p+d+1}\\
&= C(p,d,m) (hk)^{-d(p-1)/2}. 
\end{align*}
Thus, for all $k$,
\begin{align*}
H(f_k) &\le C_1(p,d,m)  (hk)^{-2}-C_2(p,d,m) (hk)^{-d(p-1)/2}.
\end{align*}
If $p< 1+4/d$, then  $d(p-1)/2 < 2$. It is now easy to see from the above bound that if $h \le A$ for some constant $A$, then 
\[
E_{\min}(m,h) \le \inf_k H(f_k) \le - C(p,d,m, A)<0,
\]
which concludes the proof. 
\end{proof}

\begin{lmm}\label{emax}
For any $m > 0$, $E_{\max}(m,h) = 2dm/h^2$.
\end{lmm}
\begin{proof}
Fix $n$. Let $\Gamma$, $R$ and $\Lambda$ be the matrices from Section~\ref{diagonal}. Let 
\[
\lambda_{\max} := \max_{y\in V_n} \lambda_y = \max_{y\in V_n} \sum_{i=1}^d 4\sin^2(\pi y_i/n). 
\]
Clearly, $\lambda_{\max}\le 4d$ always, and $\lambda_{\max} \ra 4d$ as $n\ra\infty$. Note that for any $f\in \cc^{V_n}$ with $M_{h,n}(f)=m$, 
\begin{align*}
H_{h,n}(f) &\le \frac{h^{d-2}}{2}\sum_{x,y\in V_n \atop x\sim y} |f(x)-f(y)|^2\\
&= \frac{h^{d-2}}{2} (f, \Gamma f) \\
&\le \frac{h^{d-2}}{2} \lambda_{\max} \sum_{x\in V_n} |f(x)|^2 = \frac{\lambda_{\max}m}{2h^2}. 
\end{align*}
Thus, $E_{\max}(m,h,n) \le 2dm/h^2$. 

Next, let $f := \sqrt{mh^{-d}}\rho_{[n/2]}$, where $(\rho_y)_{y\in V_n}$ are the eigenfunctions defined in Section \ref{diagonal} and $[n/2]$ is the element of $V_n$ whose components are all equal to the integer part of $n/2$. Note that $|f(x)| = \sqrt{mh^{-d}}n^{-d/2}$  for each $x$. Therefore, $M_{h,n}(f) = m$ and 
\begin{align*}
H_{h,n}(f) &= \frac{h^{d-2}}{2}\sum_{x,y\in V_n \atop x\sim y } |f(x)-f(y)|^2 - \frac{h^d}{p+1} \sum_{x\in V_n} |f(x)|^{p+1}\\
&= \frac{h^{d-2}}{2} (f, \Gamma f) - \frac{h^d}{p+1} (mh^{-d})^{(p+1)/2}n^{-d(p-1)/2}\\
&= \frac{h^{d-2} \lambda_{[n/2]}}{2}\sum_{x\in V_n} |f(x)|^2  - \frac{h^d}{p+1} (mh^{-d})^{(p+1)/2}n^{-d(p-1)/2}.
\end{align*}
Since $\lambda_{[n/2]}\ra 4d$ as $n\ra\infty$ and the second term goes to zero, this shows that
\[
\liminf_{n\ra\infty} E_{\max}(m,h,n) \ge \frac{2dm}{h^2}. 
\]
By Lemma \ref{emfacts4}, this completes the proof. 
\end{proof}

\section{Large deviations for the gradient}\label{largedev}
Recall the function $\Psi_d$ defined in the statement of Theorem \ref{discretelimit}. The following proposition summarizes some important properties of this function. 
\begin{prop}\label{psiprops}
The function $\Psi_d$ has the following properties: it is continuous in $(0, 4d)$, it is strictly decreasing in $(0, 2d]$ and strictly increasing in $[2d,4d)$, $\Psi_d(2d)=0$,  and 
\[
\lim_{\alpha \downarrow 0} \Psi_d(\alpha) =\infty = \lim_{\alpha \uparrow 4d} \Psi_d(\alpha). 
\]
In particular, $\Psi_d$ is a continuous function from $\rr$ into $[0,\infty]$. 
\end{prop}
\begin{proof}
It is obvious from the definition that $\Psi$ is strictly decreasing in $(0, 2d)$ and by symmetry, strictly increasing in $(2d, 4d)$. To extend the monotonicity up to the point $2d$, we have to show that $\Psi(\alpha) > 0$ for all $\alpha < 2d$. For each $\alpha\le2d$ and $0< \gamma < 1$ define 
\[
K_\alpha(\gamma) := \int_{[0,1]^d}\log \biggl(1-\gamma + \frac{4\gamma}{\alpha} \sum_{i=1}^d \sin^2(\pi x_i) \biggr) dx_1\cdots dx_d,
\] 
so that
\begin{equation}\label{psik}
\Psi_d(\alpha) = \sup_{0< \gamma < 1} K_\alpha(\gamma). 
\end{equation}
By dominated convergence, $K_\alpha$ is continuous on $[0, 1)$ and differentiable in $(0,1)$, and has a right derivative at $0$. Moreover, $K_\alpha(0)=0$. A simple computation shows that $K'_\alpha(0)=2d/\alpha -1 > 0$ if $\alpha < 2d$. Thus if $\alpha < 2d$, then $K_\alpha$ is strictly increasing at $0$ and therefore by \eqref{psik}, $\Psi_d(\alpha)$ is positive. By symmetry, $\Psi_d(\alpha)$ is positive for $\alpha \in (2d, 4d)$ also. 

Next, let $\alpha_n$ be a sequence converging to $\alpha \in (0, 2d)$. By dominated convergence, $K_{\alpha_n}\ra K_\alpha$ pointwise in $[0,1)$. It is easily seen that $K_{\alpha_n}$ is a strictly concave function for every $n$. By concavity and pointwise convergence, it follows that $\sup_{0< \gamma < 1} K_{\alpha_n}(\gamma)$ converges to $\sup_{0< \gamma< 1} K_{\alpha}(\gamma)$, proving that $\Psi_d$ is continuous at $\alpha$. To show continuity at $2d$, follow the same argument and simply note that $K_{2d}$ attains its maximum   at $0$, since $K_{2d}$ is concave in $[0, 1)$ and $K'_{2d}(0)=0$. 

Lastly, note that for any fixed $\gamma \in (0,1)$, $K_\alpha(\gamma) \ra\infty$ as $\alpha \ra 0$. This proves that $\lim_{\alpha \downarrow0}\Psi_d(\alpha)=\infty$. By symmetry, $\lim_{\alpha \uparrow 4d}\Psi_d(\alpha)=\infty$.
\end{proof}
Next, let $\xi = (\xi(x))_{x\in V_n}$ be a random function chosen uniformly from the unit sphere
\[
\biggl\{u\in \cc^{V_n}: \sum_{x\in V_n}|u(x)|^2 =1 \biggr\}.
\]
The following theorem is a large deviation result for the gradient of $\xi$ that is of fundamental importance in the rest of the manuscript. 
\begin{thm}\label{psifacts}
For each $\alpha \in (0,2d)$, 
\begin{align*}
\lim_{n\ra\infty} \frac{1}{n^d}\log \pp\biggl(\sum_{x, y\in V_n \atop x\sim y} |\xi(x)-\xi(y)|^2 \le \alpha\biggr) 
&= -\Psi_d(\alpha). 
\end{align*}
Moreover, for any $\delta > 0$, the same limit holds for 
\[
\frac{1}{n^d}\log \pp\biggl(\sum_{x, y\in V_n \atop x\sim y} |\xi(x)-\xi(y)|^2 \le \alpha, \ \max_{x\in V_n} |\xi(x)|^2 \le n^{-d(1-\delta)}\biggr). 
\]
The same conclusions hold if $\alpha \in (2d, \infty)$ and the `$\le \alpha$' is replaced by `$\ge\alpha$' in both expressions.  
\end{thm}
\begin{proof}[Proof of Theorem \ref{psifacts} in the case $0< \alpha < 2d$] 
Fix $\alpha \in (0, 2d)$ and a positive integer $n$. Let $\phi$ be the Gaussian random function defined in Section \ref{gaussian}. Let $\Gamma$, $R$ and $\Lambda$ be the matrices defined in Section \ref{diagonal}. Let 
\begin{equation}\label{taudef}
\tau := R^* \phi,
\end{equation}
where $R^*$ is the adjoint of $R$. 
Since $R$ is a unitary matrix,  $\tau$ has the same distribution as $\phi$. Moreover, 
\begin{equation}\label{r1}
\sum_{y\in V_n} |\tau(y)|^2 = \sum_{x\in V_n} |\phi(x)|^2
\end{equation}
and 
\begin{equation}\label{r2}
\sum_{y\in V_n} \lambda_y |\tau(y)|^2 = (\phi, \Gamma\phi) = \sum_{x, y\in V_n \atop x\sim y}|\phi(x)-\phi(y)|^2. 
\end{equation}
Let 
\[
\xi(x) := \frac{\phi(x)}{\bigl(\sum_{y\in V_n} |\phi(y)|^2\bigr)^{1/2}}.
\]
Then $\xi$ is uniformly distributed on the unit sphere of $\cc^{V_n}$. Moreover, by \eqref{r1} and \eqref{r2},
\begin{align*}
\sum_{x, y\in V_n \atop x\sim y} |\xi(x) - \xi(y)|^2 &= \frac{\sum_{x,y\in V_n \atop x\sim y} |\phi(x)-\phi(y)|^2}{\sum_{x\in V_n} |\phi_x|^2}\\
&= \frac{\sum_{y\in V_n} \lambda_y |\tau(y)|^2}{\sum_{y\in V_n} |\tau(y)|^2}. 
\end{align*}
Let $\eta(y) := n^dh^d |\tau(y)|^2$. Then $\eta(y)$ is an exponential random variable with mean~$1$ and the $\eta(y)$'s are independent. For any $\alpha \in (0, 2d)$, 
\begin{align*}
\pp\biggl(\sum_{x, y\in V_n \atop x\sim y} |\xi(x) - \xi(y)|^2 \le \alpha\biggr) &= \pp\biggl(\sum_{y\in V_n} (\lambda_y - \alpha) |\tau(y)|^2 \le 0\biggr)\\
&=  \pp\biggl(\sum_{y\in V_n} (\lambda_y - \alpha) \eta(y) \le 0\biggr).
\end{align*}
Thus, for any $\theta \in [0, 1/\alpha)$, 
\begin{align*}
\pp\biggl(\sum_{x, y\in V_n \atop x\sim y} |\xi(x) - \xi(y)|^2 \le \alpha\biggr) &\le \ee(e^{-\theta\sum_{y\in V_n} (\lambda_y - \alpha) \eta(y)})\\
&= \prod_{y\in V_n} \frac{1}{1+\theta(\lambda_y - \alpha)}. 
\end{align*}
Note that we need the restriction that $\theta < 1/\alpha$ since $\lambda_0=0$. Now,
\begin{align}
&-\frac{1}{n^d} \log \prod_{y\in V_n} \frac{1}{1+\theta(\lambda_y - \alpha)}\nonumber \\
&= \frac{1}{n^d}\sum_{y_1,\ldots, y_d=0}^{n-1} \log \biggl(1-\theta\alpha + 4\theta\sum_{i=1}^d \sin^2(\pi y_i/n)\biggr) =: I_n(\theta). \label{indef}
\end{align}
The sequence of functions $I_n$ converges pointwise to the function $I$ on $[0,1/\alpha)$, where
\[
I(\theta) = \int_{[0,1]^d}\log \biggl(1-\theta \alpha + 4\theta \sum_{i=1}^d \sin^2(\pi x_i) \biggr) dx_1\cdots dx_d. 
\]
This shows that for each $\theta\in [0,1/\alpha)$, 
\begin{equation}\label{upperfirst}
\limsup_{n\ra\infty} \frac{1}{n^d} \log \pp\biggl(\sum_{x,y\in V_n\atop x\sim y} |\xi(x) - \xi(y)|^2 \le \alpha\biggr) \le -I(\theta).
\end{equation}
Now, $I(\theta) = K_\alpha(\theta \alpha)$, where $K_\alpha$ is defined in the proof of Proposition~\ref{psiprops}. By \eqref{psik} and \eqref{upperfirst}, this shows that 
\begin{equation}\label{uppermain1}
\begin{split}
&\limsup_{n\ra\infty} \frac{1}{n^d}\log \pp\biggl(\sum_{x, y\in V_n, \atop x\sim y} |\xi(x)-\xi(y)|^2 \le \alpha\biggr) \\
&\le -\sup_{\theta \in (0,1/\alpha)} I(\theta) = -\sup_{\theta\in (0,1/\alpha)} K_\alpha(\theta \alpha) =  -\Psi_d(\alpha). 
\end{split}
\end{equation}
This proves the upper bound in the case $0< \alpha < 2d$. Next, we establish the matching lower bound.

As noted in Proposition \ref{psiprops}, $K_\alpha$ is a concave function and is strictly increasing at $0$ if $\alpha < 2d$. Thus, the maximum of $I$ in $[0,1/\alpha)$ must be achieved either inside $(0,1/\alpha)$, or as $\theta \ra1/\alpha$. The first case holds if $I'(1/\alpha) < 0$, and the second happens if $I'(1/\alpha) \ge 0$. The proof of the lower bound is different in the two cases. 

\vskip.1in
{\bf Case 1:} $I'(1/\alpha) < 0$. 

In this case, there is a unique $\theta^*\in (0,1/\alpha)$ where $I$ is maximum.   Fix any $\ep > 0$, and let $\theta'$ be a point so close to $\theta^*$ from the right  that $I'(\theta') \in (-\ep, 0)$. Fixing $n$, let $\eta' = (\eta'(y))_{y\in V_n}$ be a collection of independent random variables, where $\eta'(y)$ is an Exponential random variable with mean $1/(1+\theta'(\lambda_y- \alpha))$.  Assume that $\eta'$ is defined on the same probability space as all other variables, and is independent of everything else. 

Recall the definition \eqref{taudef} of $\tau$. Let $\sigma(y) := \tau(y)/|\tau(y)|$. Since $\tau(y)$ is a complex Gaussian random variable with mean zero, $\sigma(y)$ and $|\tau(y)|$ are independent random variables. Note that 
\[
\tau(y) = \sigma(y) |\tau(y)|= \sigma(y) \sqrt{\frac{\eta(y)}{n^dh^d}}.
\] 
Define a random vector $\xi' = (\xi'(x))_{x\in V_n}$ as follows: Take the variables $\sigma(y)$ and $\eta'(y)$  defined above, and let 
\[
\tau'(y) := \sigma(y) \sqrt{\frac{\eta'(y)}{n^dh^d}}.
\]
Let $\phi' := R \tau'$, and let 
\[
\xi'(x) := \frac{\phi'(x)}{\bigl(\sum_{y\in V_n} |\phi'(y)|^2\bigr)^{1/2}} =  \frac{\phi'(x)}{\bigl(\sum_{y\in V_n} |\tau'(y)|^2\bigr)^{1/2}}.
\]
Note that $\tau'$, just like $\tau$, is a collection of independent complex Gaussian random variables with $\ee(\tau'(y))=\ee({\tau'(y)}^2)=0$ for each $y$, but unlike $\tau$, 
\begin{equation}\label{taum}
\ee|\tau'(y)|^2 = \frac{1}{n^dh^d(1+\theta'(\lambda_y-\alpha))}. 
\end{equation}
Consequently, $\phi'$ is a complex Gaussian random vector satisfying
\begin{equation}\label{phim}
\ee|\phi'(x)|^2 \le \frac{1}{n^dh^d(1-\theta'\alpha)} \ \text{ for all } x,
\end{equation}
but its components are not independent. 

Note that the map that takes $(\sigma, \eta)$ to $\xi$ is the same map as the one that takes $(\sigma, \eta')$ to $\xi'$. Therefore a simple change-of-measure computation shows that for any measurable set $A\subseteq \cc^{V_n}$, 
\begin{align*}
\pp(\xi\in A) &= \ee(\rho(\eta')1_{\{\xi'\in A\}}), 
\end{align*}
where
\[
\rho(\eta') = e^{\sum_{y\in V_n} \theta'(\lambda_y -\alpha)\eta'(y)} \prod_{y\in V_n} \frac{1}{1+\theta'(\lambda_y-\alpha)}
\]
Thus, if we fix $\delta> 0$ and let $E$ be the event
\begin{align*}
E &:= \biggl\{-4n^d\ep\le \sum_{y\in V_n} (\lambda_y - \alpha) \eta'(y) \le 0, \\
&\qquad \max_{x\in V_n}|\phi'(x)|^2 \le n^{-d(1-\delta)}\sum_{x\in V_n} |\tau'(x)|^2 \biggr\},
\end{align*}
then 
\begin{align}
&\pp\biggl(\sum_{x, y\in V_n \atop x\sim y} |\xi(x)-\xi(y)|^2 \le \alpha, \ \max_{x\in V_n} |\xi(x)|^2 \le n^{-d(1-\delta)}\biggr)\nonumber \\
&= \ee\biggl(\rho(\eta') 1\biggl\{\sum_{x, y\in V_n\atop x\sim y} |\xi'(x)-\xi'(y)|^2 \le \alpha, \ \max_{x\in V_n} |\xi'(x)|^2 \le n^{-d(1-\delta)}\biggr\}\biggr)\nonumber \\
&= \ee\biggl(\rho(\eta')1\biggl\{\sum_{y\in V_n} (\lambda_y - \alpha) \eta'(y) \le 0, \nonumber \\
&\qquad \qquad \max_{x\in V_n}|\phi'(x)|^2 \le n^{-d(1-\delta)}\sum_{x\in V_n} |\tau'(x)|^2 \biggr\} \biggr)\nonumber \\
&\ge \ee(\rho(\eta')1_E)\ge e^{-4n^d\theta'\ep} \pp(E) \prod_{y\in V_n} \frac{1}{1+\theta'(\lambda_y-\alpha)}. \label{xixip}
\end{align}
Now, if $I_n$ is the function defined in \eqref{indef}, a  simple computation gives 
\[
I_n' (\theta') = \frac{1}{n^d}\sum_{y\in V_n} \frac{\lambda_y-\alpha}{1+\theta'(\lambda_y - \alpha)}.
\]
Thus, 
\[
\ee\biggl(\frac{1}{n^d}\sum_{y\in V_n} (\lambda_y - \alpha)\eta'(y)\biggr)=I_n'(\theta'). 
\]
But by independence, 
\[
\var\biggl(\frac{1}{n^d}\sum_{y\in V_n} (\lambda_y - \alpha)\eta'(y)\biggr)\le \frac{C(\alpha, \theta')}{n^d}.
\]
In particular, the variance tends to zero as $n\ra\infty$. Since $I_n'(\theta')\ra I'(\theta)\in (-\ep, 0)$, this shows that as $n\ra\infty$, 
\begin{equation}\label{lambdaeta}
\pp\biggl(-4n^d\ep \le \sum_{y\in V_n} (\lambda_y - \alpha)\eta'(y)\le 0 \biggr) \ra 1. 
\end{equation}
By the Gaussian nature of $\phi'$ and \eqref{phim}, for any $\delta' > 0$,
\[
\lim_{n\ra\infty} \pp\bigl(n^dh^d\max_{x\in V_n}|\phi'(x)|^2 \le n^{\delta'd}) =1,
\]
and similarly by \eqref{taum} and the independence of the coordinates of $\tau'$,
\begin{align*}
\frac{1}{n^d}\sum_{y\in V_n} n^dh^d|\tau'(y)|^2 &\ra \int_{[0,1]^d} \frac{1}{1-\theta'\alpha + 4\theta' \sum_{i=1}^d\sin^2(\pi x_i)} dx_1\cdots dx_d > 0  \\
&\text{ in probability as } n\ra\infty.  
\end{align*}
Combining the last two displays, and choosing $\delta'\in (0,\delta)$, we get 
\begin{align}\label{phitau}
\lim_{n\ra\infty} \pp\biggl(\max_{x\in V_n}|\phi'(x)|^2 \le n^{-d(1-\delta)} \sum_{x\in V_n} |\tau'(x)|^2\biggr) = 1. 
\end{align}
Therefore, by \eqref{xixip}, \eqref{lambdaeta} and \eqref{phitau}, $\pp(E)\ra 1$ as $n \ra\infty$, and hence 
\begin{align*}
&\liminf_{n\ra\infty} \frac{1}{n^d} \log \pp\biggl(\sum_{x,y\in V_n \atop x\sim y} |\xi(x)-\xi(y)|^2 \le \alpha, \ \max_{x\in V_n} |\xi(x)|^2 \le n^{-d(1-\delta)}\biggr)  \\
&\ge -4\theta'\ep + \lim_{n\ra\infty} \frac{1}{n^d} \sum_{y\in V_n} \log(1+\theta'(\lambda_y -\alpha)) \\
&= -4\theta'\ep - I(\theta'). 
\end{align*}
Since $\ep$ is arbitrary and $\theta' \ra\theta^*$ as $\ep\ra 0$, this shows that when $\alpha\in (0, 2d)$ and $I'(\alpha)< 0$,
\begin{equation}\label{lowermain1}
\begin{split}
&\liminf_{n\ra\infty} \frac{1}{n^d}\log \pp\biggl(\sum_{x, y\in V_n \atop x\sim y} |\xi(x)-\xi(y)|^2 \le \alpha, \ \max_{x\in V_n} |\xi(x)|^2 \le n^{-d(1-\delta)}\biggr) \\
&\ge -I(\theta^*) =  -\Psi_d(\alpha). 
\end{split}
\end{equation}

\vskip.1in 
{\bf Case 2:} $I'(1/\alpha) \ge 0$.  

Take any $\ep >0$ and let $\theta'$ solve $1-\theta'\alpha= \ep$.  Fix $n$ and define $\eta'$, $\tau'$, $\xi'$ and $\delta$ as in the previous case. Then as before, we arrive at the inequality \eqref{xixip}, with $E$ being the same event. It suffices, as before, to show that 
\begin{equation}\label{pebd}
\liminf_{n\ra\infty} \frac{1}{n^{d}}\log \pp(E) \ge 0.
\end{equation}
Let 
\[
X_n := \frac{1}{n^d} \sum_{y\in V_n\backslash \{0\}} (\lambda_y - \alpha) \eta'(y), 
\]
and 
\[
Y_n := \frac{(\lambda_0-\alpha)\eta'(0)}{n^d} = -\frac{\alpha \eta'(0)}{n^d}. 
\]
As before, it is easy to argue that $\var(X_n)\ra 0$ as $n\ra\infty$. Again, 
\begin{align*}
\ee(X_n) &=  \frac{1}{n^d} \sum_{y\in V_n\backslash \{0\}} \frac{\lambda_y - \alpha}{1+\theta'(\lambda_y -\alpha)}\\
&\ra I'(\theta') \text{ as } n\ra\infty. 
\end{align*}
Consequently, $X_n \ra I'(\theta')$ in probability as $n\ra\infty$. 

Next let $\eta''(0)$ be an independent copy of $\eta'(0)$. Let $\eta''$ be the vector in $\cc^{V_n}$ whose $0$th component is $\eta''(0)$ and $\eta''(y) = \eta'(y)$ for every $y \ne 0$. Let $\tau''$ and $\phi''$  be obtained from $(\sigma, \eta'' )$ the same way as $\tau'$ and $\phi'$  were obtained from $(\sigma, \eta')$. Then note that $\phi''$ is independent of $\eta'(0)$. Since the elements of the matrix $R$ are bounded in absolute value by $n^{-d/2}$, and from definition we have 
\begin{enumerate}
\item[(a)] $\phi''=R \tau''$, $\phi' = R\tau'$,  
\item[(b)] $\tau''(0) = \sigma(0) \sqrt{\eta''(0)/(nh)^d}$, $\tau'(0)= \sigma(0)\sqrt{ \eta'(0)/(nh)^d}$, and 
\item[(c)] $\tau''(y) = \tau'(y)$ for all $y \ne 0$,
\end{enumerate}
therefore for all $x\in V_n$,
\begin{align}\label{phiphi}
|\phi''(x) - \phi'(x)| \le n^{-d/2}(nh)^{-d/2}|(\eta''(0))^{1/2}-(\eta'(0))^{1/2}|. 
\end{align}
Fix $\delta' \in (0,\delta)$ and some $\ep'$ so small that $-d(1-\delta') + \ep' < -d(1-\delta)$, and define four events:
\begin{align*}
E_1 &:= \{|X_n - I'(\theta')| \le \ep\},\\
E_2 &:= \{-I'(\theta') - 3\ep\le Y_n \le -I'(\theta')-\ep\} \\
&= \biggl\{\frac{n^d(I'(\theta')+\ep)}{\alpha} \le \eta'(0) \le\frac{ n^d(I'(\theta') + 3\ep)}{\alpha}\biggr\}, \\
E_3 &:= \biggl\{(nh)^d\max_{x\in V_n}|\phi''(x)|^2 \le n^{\delta'd}, \ \frac{1}{n^d}\sum_{x\in V_n \backslash \{0\}}(nh)^d|\tau''(x)|^2\ge n^{-\ep'}\biggr\},\\
E_4 &:= \{|\eta''(0)|\le n^d\}. 
\end{align*} 
Then by \eqref{phiphi}, $E_2\cap E_3 \cap E_4$ implies that for each $x$,
\begin{align*}
|\phi'(x)|^2 &\le 2|\phi''(x)|^2 + 2|\phi'(x)- \phi''(x)|^2 \\
&\le (nh)^{-d}(2n^{\delta'd} + C(\alpha, \ep)).
\end{align*}
Since $\tau''(x) =\tau'(x)$ for all $x\ne 0$, this shows that for all $n\ge C(\alpha, \ep, \delta', \ep')$, $E_2\cap E_3\cap E_4$ implies 
\[
\max_{x\in V_n}|\phi'(x)|^2 \le n^{-d(1-\delta')+\ep'} \sum_{x\in V_n\backslash\{0\}} |\tau'(x)|^2 \le n^{-d(1-\delta)} \sum_{x\in V_n} |\tau'(x)|^2. 
\]
Again, $E_1\cap E_2$ implies $-4\ep \le X_n + Y_n \le 0$, 
which is the same as 
\[
-4n^d\ep \le \sum_{y\in V_n} (\lambda_y - \alpha) \eta'(y) \le 0. 
\]
Thus, for all $n\ge C(\alpha, \ep, \delta', \ep', \delta)$, $E_1\cap E_2 \cap E_3\cap E_4$ implies $E$. So, to show~\eqref{pebd}, it suffices  to show that 
\begin{equation}\label{pebd2}
\liminf_{n\ra\infty} \frac{1}{n^d} \log \pp(E_1\cap E_2 \cap E_3\cap E_4) \ge 0. 
\end{equation}
Since $\eta'(0)$ is an Exponential random variable with mean $1/\ep$, it is easy to see that 
\begin{align*}
\pp(E_2)&\ge \frac{2n^d\ep^2}{\alpha} \exp\biggl(-\frac{\ep n^d(I'(\theta') + 3\ep)}{\alpha}\biggr). 
\end{align*}
As argued above, $X_n \ra I'(\theta')$ in probability and therefore $\pp(E_1) \ra 1$ as $n \ra\infty$. The probability of $E_4$ tends to $1$ trivially, and $\pp(E_3)\ra1$ by the same logic that led to \eqref{phitau}. Thus, $\pp(E_1\cap E_3\cap E_4) \ra 1$. Lastly, observe that the events $E_1$, $E_3$ and $E_4$ are jointly independent of $E_2$. Combining all of these observations, we get
\begin{align*}
&\liminf_{n\ra\infty} \frac{1}{n^d} \log \pp(E_1\cap E_2\cap E_3\cap E_4) \\
&= \liminf_{n\ra\infty} \frac{1}{n^d} (\log \pp(E_1\cap E_3\cap E_4) + \log \pp(E_2)) \\
&\ge -\frac{\ep(I'(\theta') +3\ep)}{\alpha}. 
\end{align*}
Since $\ep$ is arbitrary and $0\le I'(\theta) \le I'(0)=2d-\alpha$ for all $\theta$ (by concavity and the assumption that $I'(1/\alpha)\ge 0$), this proves \eqref{pebd2} and hence \eqref{pebd}, leading to the proof of \eqref{lowermain1} when $I'(1/\alpha) \ge 0$. Combining \eqref{uppermain1} and \eqref{lowermain1}, the proof of Theorem \ref{psifacts} for $\alpha\in (0,2d)$ is complete.  
\end{proof}

\begin{proof}[Proof of Theorem \ref{psifacts} in the case $2d< \alpha < \infty$]
The proof  is very similar to the previous case, with a few important modifications. Let all notation be as before. Note that for any $\alpha \in (2d, 4d)$, 
\begin{align*}
\pp\biggl(\sum_{x,y \in V_n \atop x\sim y} |\xi(x) - \xi(y)|^2 \ge \alpha\biggr) &= \pp\biggl(\sum_{y\in V_n} (\lambda_y - \alpha) |\tau(y)|^2 \ge 0\biggr)\\
&=  \pp\biggl(\sum_{y\in V_n} (\lambda_y - \alpha) \eta(y) \ge 0\biggr).
\end{align*}
Thus, for any $\theta \in [0, 1/(4d-\alpha))$, 
\begin{align*}
\pp\biggl(\sum_{x,y\in V_n \atop x\sim y} |\xi(x) - \xi(y)|^2 \ge \alpha\biggr) &\le \ee(e^{\theta\sum_{y\in V_n} (\lambda_y - \alpha) \eta(y)})\\
&= \prod_{y\in V_n} \frac{1}{1-\theta(\lambda_y - \alpha)}. 
\end{align*}
The condition $\theta < 1/(4d-\alpha)$ ensures that the right-hand side makes sense, since $\lambda_y$ is uniformly bounded above by $4d$. Now,
\begin{align*}
&-\frac{1}{n^d} \log \prod_{y\in V_n} \frac{1}{1-\theta(\lambda_y - \alpha)}\\
&= \frac{1}{n^d}\sum_{y_1,\ldots, y_d=0}^{n-1} \log \biggl(1+\theta\alpha - 4\theta\sum_{i=1}^d \sin^2(\pi y_i/n)\biggr) =: J_n(\theta), 
\end{align*}
and the sequence of functions $J_n$ converges pointwise to the function $J$ on the interval $[0,1/(4d-\alpha))$, where
\[
J(\theta) = \int_{[0,1]^d}\log \biggl(1+\theta \alpha - 4\theta \sum_{i=1}^d \sin^2(\pi x_i) \biggr) dx_1\cdots dx_d. 
\]
This shows that for each $\theta\in [0,1/(4d-\alpha))$, 
\[
\limsup_{n\ra\infty} \frac{1}{n^d} \log \pp\biggl(\sum_{x,y\in V_n \atop x\sim y} |\xi(x) - \xi(y)|^2 \ge \alpha\biggr) \le -J(\theta),
\]
and therefore
\begin{align*}
\limsup_{n\ra\infty} \frac{1}{n^d}\log \pp\biggl(\sum_{x, y\in V_n, \atop x\sim y} |\xi(x)-\xi(y)|^2 \ge \alpha\biggr) &\le -\sup_{\theta \in (0,1/(4d-\alpha))} J(\theta). 
\end{align*}
Now, putting $\gamma=(4d-\alpha)\theta$ gives
\begin{align*}
&\sup_{\theta \in (0,1/(4d-\alpha))} J(\theta) = \sup_{\gamma\in (0,1)} J(\gamma/(4d-\alpha))\\
&= \sup_{\gamma\in (0,1)} \int_{[0,1]^d} \log\biggl(1+\frac{\gamma\alpha}{4d-\alpha} - \frac{4\gamma}{4d-\alpha}\sum_{i=1}^d \sin^2(\pi x_i)\biggr) dx_1\cdots dx_d\\
&= \sup_{\gamma\in (0,1)} \int_{[0,1]^d} \log\biggl(1-\gamma + \frac{4\gamma}{4d-\alpha}\sum_{i=1}^d \cos^2(\pi x_i)\biggr) dx_1\cdots dx_d.
\end{align*}
The above expression is the same as $\Psi_d(4d-\alpha)$, except that we have $\cos$ instead of $\sin$. However, this does not matter, since we can break up the hypercube $[0,1]^d$ into a union of smaller hypercubes like $[a_1,a_1+1/2]\times\cdots \times[a_d, a_d+1/2]$ where each $a_i\in \{0,1/2\}$, and then, within each hypercube, replace $\cos$ by $\sin$ by a change of variable $y_i=1/2-x_i$ when $a_i=0$ and $y_i=3/2-x_i$ when $a_i=1/2$. Thus, 
\begin{equation}\label{uppermain2}
\begin{split}
&\limsup_{n\ra\infty} \frac{1}{n^d}\log \pp\biggl(\sum_{x, y\in V_n, \atop x\sim y} |\xi(x)-\xi(y)|^2 \ge \alpha\biggr) \\
&\le  -\sup_{\theta \in (0, 1/(4d-\alpha))} J(\theta) =  -\Psi_d(4d-\alpha)=-\Psi_d(\alpha). 
\end{split}
\end{equation}
Next, we turn our attention to the lower tail. The function $J$ is continuous in  the interval $[0,1/(4d-\alpha))$ and differentiable in $(0,1/(4d-\alpha))$. Moreover, $J$ is differentiable from the right at $0$ and differentiable from the left at $1/(4d-\alpha)$, and an easy computation gives $J'(0) = \alpha-2d$. Since $\alpha > 2d$ and $J$ is continuous at $0$, this shows that $J$ must be strictly increasing in a neighborhood of $0$. It is easy to check that $J$ is a concave function. Since $J$ is increasing at $0$, its maximum in $[0,1/(4d-\alpha))$ must be achieved either inside $(0,1/(4d-\alpha))$, or as $\theta \ra1/(4d-\alpha)$. The first case holds if $J'(1/(4d-\alpha)) < 0$, and the second happens if $J'(1/(4d-\alpha)) \ge 0$. Just as before, the proof of the lower bound is different in the two cases. 

\vskip.1in
{\bf Case 1:} $J'(1/(4d-\alpha)) < 0$. 

In this case, there is a unique $\theta^*\in (0,1/(4d-\alpha))$ where $J$ is maximum. Fix any $\ep > 0$, and let $\theta'$ be a point so close to $\theta^*$ from the right  that $J'(\theta') \in (-\ep, 0)$. Given $n$, let $(\eta'(y))_{y\in V_n}$ be a collection of independent random variables, where $\eta'(y)$ has the Exponential distribution with mean $1/(1-\theta'(\lambda_y- \alpha))$.  (Note the minus sign in front of $\theta'$, which was plus in the case $\alpha < 2d$.) As before, assume that $\eta'$ is defined on the same probability space as all other variables, and is independent of everything else. 

Given $\eta'$, define $\tau'$, $\phi'$ and $\xi'$ as before. Then all the properties of these vectors are same as before, except that now
\begin{equation*}\label{taum2}
\ee|\tau'(y)|^2 = \frac{1}{(nh)^d(1-\theta'(\lambda_y-\alpha))} 
\end{equation*}
and 
\begin{equation*}\label{phim2}
\ee|\phi'(x)|^2 \le \frac{1}{(nh)^d(1-\theta'(4d-\alpha))} \ \text{ for all } x\in V_n.
\end{equation*}
Defining
\[
\rho(\eta') = e^{-\sum_{y\in V_n} \theta'(\lambda_y -\alpha)\eta'(y)} \prod_{y\in V_n} \frac{1}{1-\theta'(\lambda_y-\alpha)}.
\]
Fix $\delta > 0$ and let $E$ be the event
\begin{align*}
E &:= \biggl\{0\le \sum_{y\in V_n} (\lambda_y - \alpha) \eta'(y) \le 4n^d\ep, \\
&\qquad \qquad \max_{x\in V_n}|\phi'(x)|^2 \le n^{-d(1-\delta)}\sum_{x\in V_n} |\tau'(x)|^2 \biggr\}.
\end{align*}
Then  as before we arrive at the inequality 
\begin{align}
&\pp\biggl(\sum_{x, y\in V_n \atop x\sim y} |\xi(x)-\xi(y)|^2 \ge \alpha, \ \max_{x\in V_n} |\xi(x)|^2 \le n^{-d(1-\delta)}\biggr)\nonumber \\
&\ge e^{-4\theta'n^d\ep} \pp(E) \prod_{y\in V_n} \frac{1}{1-\theta'(\lambda_y-\alpha)}. \label{xixip2}
\end{align}
Exactly as before, we can now argue that 
\[
\frac{1}{n^d}\sum_{y\in V_n} (\lambda_y - \alpha)\eta'(y)\ra -J'(\theta') \ \text{ in probability as } n \ra\infty.
\]
Since $J'(\theta)\in (-\ep, 0)$, this shows that as $n\ra\infty$, 
\begin{equation}\label{lambdaeta2}
\pp\biggl(0\le \sum_{y\in V_n} (\lambda_y - \alpha)\eta'(y)\le 4n^d\ep \biggr) \ra 1. 
\end{equation}
The inequality \eqref{phitau} continues to be valid, and therefore, by \eqref{xixip2}, \eqref{lambdaeta2} and~\eqref{phitau},
\begin{align*}
&\liminf_{n\ra\infty} \frac{1}{n^d} \log \pp\biggl(\sum_{x,y\in V_n\atop x\sim y} |\xi(x)-\xi(y)|^2 \ge \alpha, \ \max_{x\in V_n} |\xi(x)|^2 \le n^{-d(1-\delta)}\biggr) \\
&\ge -4\ep + \lim_{n\ra\infty} \frac{1}{n^d} \sum_{y\in V_n} \log(1-\theta'(\lambda_y -\alpha)) \\
&= -4\ep - J(\theta'). 
\end{align*}
Since $\ep$ is arbitrary and $\theta' \ra\theta^*$ as $\ep\ra 0$, this shows (as in \eqref{uppermain2}) that when $\alpha \in (2d, 4d)$ and  $J'(1/(4d-\alpha)) < 0$,  
\begin{equation}\label{lowermain2}
\begin{split}
&\liminf_{n\ra\infty} \frac{1}{n^d} \log \pp\biggl(\sum_{x,y\in V_n\atop x\sim y} |\xi(x)-\xi(y)|^2 \ge \alpha, \ \max_{x\in V_n} |\xi(x)|^2 \le n^{-d(1-\delta)}\biggr) \\
&\qquad \ge -\Psi_d(\alpha). 
\end{split}
\end{equation} 

\vskip.1in 
{\bf Case 2:} $J'(1/(4d-\alpha)) \ge 0$.  

Take any $\ep >0$ and let $\theta'$ solve $1-\theta'(4d-\alpha)= \ep$.  Fix $n$ and define $\eta'$, $\tau'$ and $\xi'$ as in the previous case. Let $[n/2]$ is the vector in $\cc^{V_n}$ whose components are all equal to the integer part of $n/2$. Then as before, we arrive at the inequality \eqref{xixip2}, with $E$ being the same event. 
Let 
\[
X_n := \frac{1}{n^d} \sum_{y\in V_n\backslash \{[n/2]\}} (\lambda_y - \alpha) \eta'(y), 
\]
and 
\[
Y_n := \frac{(\lambda_{[n/2]}-\alpha)\eta'([n/2])}{n^d}. 
\]
Note that $\lambda_{[n/2]} \ra 4d$ as $n \ra\infty$. 

As before, it is easy to argue that $\var(X_n)\ra 0$ as $n\ra\infty$. Again, 
\begin{align*}
\ee(X_n) &=  \frac{1}{n^d} \sum_{y\in V_n\backslash \{[n/2]\}} \frac{\lambda_y - \alpha}{1-\theta'(\lambda_y -\alpha)}\\
&\ra -J'(\theta') \text{ as } n\ra\infty. 
\end{align*}
Consequently, $X_n \ra -J'(\theta')$ in probability as $n\ra\infty$. 

Next let $\eta''([n/2])$ be an independent copy of $\eta'([n/2])$. Let $\eta''$ be the vector in $\cc^{V_n}$ whose $[n/2]$th component is $\eta''([n/2])$ and $\eta''(y) = \eta'(y)$ for every $y \ne [n/2]$. Let $\tau''$ and $\phi''$  be obtained from $(\sigma, \eta'' )$ the same way as $\tau'$ and $\phi'$  were obtained from $(\sigma, \eta')$. Then note that $\phi''$ is independent of $\eta'(0)$. Exactly as we proved \eqref{phiphi}, it follows that for all $x\in V_n$,
\begin{align*}
|\phi''(x) - \phi'(x)| \le n^{-d/2}(nh)^{-d/2}|(\eta''([n/2]))^{1/2}-(\eta'([n/2]))^{1/2}|. 
\end{align*}
Fix $\delta' \in (0,\delta)$ and some $\ep'$ so small that $-d(1-\delta') + \ep' < -d(1-\delta)$, and define four events:
\begin{align*}
E_1 &:= \{|X_n + J'(\theta')| \le \ep\},\\
E_2 &:= \{J'(\theta') + \ep\le Y_n \le J'(\theta')+3\ep\} \\
&= \biggl\{\frac{n^d(J'(\theta')+\ep)}{\lambda_{[n/2]} - \alpha} \le \eta'([n/2]) \le\frac{ n^d(J'(\theta') + 3\ep)}{\lambda_{[n/2]} - \alpha}\biggr\}, \\
E_3 &:= \biggl\{(nh)^d\max_{x\in V_n}|\phi''(x)|^2 \le n^{\delta'd}, \ \frac{1}{n^d}\sum_{x\in V_n\backslash\{[n/2]\}}(nh)^d|\tau''(x)|^2\ge n^{-\ep'}\biggr\},\\
E_4 &:= \{\eta''([n/2])\le n^d\}. 
\end{align*} 
Then $E_2\cap E_3 \cap E_4$ implies that for each $x$,
\begin{align*}
|\phi'(x)|^2 &\le 2|\phi''(x)|^2 + 2|\phi'(x)- \phi''(x)|^2 \\
&\le (nh)^d(2n^{\delta'd} + C(\alpha, \ep)).
\end{align*}
Since $\tau''(x) =\tau'(x)$ for all $x\ne [n/2]$, this shows that for $n \ge C(\alpha, \ep, \delta', \ep')$, $E_2\cap E_3\cap E_4$ implies 
\[
\max_{x\in V_n}|\phi'(x)|^2 \le n^{-d(1-\delta)} \sum_{x\in V_n \backslash\{0\}} |\tau'(x)|^2 \le n^{-d(1-\delta)} \sum_{x\in V_n} |\tau'(x)|^2. 
\]
Again, $E_1\cap E_2$ implies $0 \le X_n + Y_n \le 4\ep$, 
which is the same as 
\[
0 \le \sum_{y\in V_n} (\lambda_y - \alpha) \eta'(y) \le 4n^d\ep. 
\]
Thus, for $n \ge C(\alpha, \ep, \delta', \ep', \delta)$, $E_1\cap E_2 \cap E_3\cap E_4$ implies $E$. So it suffices to find a lower bound for the probability of $E_1\cap E_2 \cap E_3\cap E_4$.

Since $\eta'([n/2])$ is an Exponential random variable with mean 
\[
\frac{1}{1-\theta'(\lambda_{[n/2]} - \alpha)},
\]
and $\lambda_{[n/2]} \ra 4d$ as $n\ra\infty$ and by definition of $\theta'$, $1-\theta'(4d-\alpha) = \ep$, it is easy to see that 
\begin{align*}
\liminf_{n\ra\infty} \frac{1}{n^d}\log\pp(E_2)&\ge -\frac{\ep (J'(\theta') + 3\ep)}{4d-\alpha}. 
\end{align*}
The proof is now completed exactly as for the lower tail in the case $\alpha \in (0,2d)$. 
\end{proof}

The following theorem is the main result of this section. 

\begin{thm}\label{psifacts2}
For each $\ep > 0$ and $\alpha > 0$, 
\begin{equation*}\label{third}
\lim_{n\ra\infty} \frac{1}{n^d}\log \pp\biggl(\biggl|\sum_{x, y\in V_n \atop x\sim y} |\xi(x)-\xi(y)|^2 - \alpha\biggr| \le \ep\biggr) 
= -\Psi_{d, \ep}(\alpha),
\end{equation*}
where 
\[
\Psi_{d,\ep}(\alpha) = 
\begin{cases}
\Psi_d(\alpha+\ep) &\text{ if } \alpha \le 2d-\ep,\\
\Psi_d(\alpha-\ep) &\text{ if } \alpha \ge 2d+\ep,\\
0 &\text{ if } 2d-\ep< \alpha < 2d+\ep.
\end{cases}
\]
In particular, $\Psi_{d,\ep}$ converges uniformly to $\Psi_d$ on compact subsets of $(0, 4d)$ as $\ep \ra0$. As in Theorem \ref{psifacts}, the same limit holds if we include the additional requirement that $\max_{x\in V_n}|\xi(x)|^2 \le n^{-d(1-\delta)}$.
\end{thm}
\begin{proof}
Obvious from Theorem \ref{psifacts} and Proposition \ref{psiprops}. 
\end{proof}

\section{The variational problem}\label{variational}
For each $E\ge 0$, $m \ge 0$ and $h > 0$ define 
\[
\Theta(E, m,h) := \log m - \Psi_d\biggl(\frac{2h^2E}{m}\biggr). 
\]
When $m = 0$, the right-hand side is interpreted as $-\infty$. With this definition, it is easy to verify that for fixed $h$, $\Theta$ is a continuous function from $[0,\infty)\times [0,\infty)$ into $[-\infty, \infty)$. 

Given $m_0 > 0$ and $E_{\min}(m,h)< E_0 < E_{\max}(m,h)$ let $\mm(E_0, m_0,h)$ denote the set of all $(E, m)$ that maximize $\Theta(E,m,h)$ in the set  
\begin{equation}\label{mr2}
\begin{split}
\mr(E_0, m_0,h) &= \{(E, m): 0\le m\le m_0, \\
&\qquad \max\{E^-(m,h), 0\}\le E\le E^+(m,h)\},
\end{split}
\end{equation}
where 
\begin{equation}\label{eplus}
\begin{split}
E^-(m,h) &:= E_0-E_{\max}(m_0-m,h), \\ 
E^+(m,h) &:= E_0 - E_{\min}(m_0-m,h).
\end{split}
\end{equation}
Define 
\[
\widehat{\Theta}(E_0, m_0,h) := \max_{(E,m)\in \mr(E_0, m_0,h)} \Theta(E,m,h).
\]
The following lemma lists some important properties of the sets $\mr$ and $\mm$. 
\begin{lmm}\label{rlmm}
Suppose that $E_{\min}(m_0,h) \le E_0 < dm_0/h^2$. Then the set $\mr(E_0, m_0, h)$ is a non-empty compact subset of $\rr^2$, and so is $\mm(E_0, m_0,h)$. Moreover, any $(E, m) \in \mm(E_0, m_0,h)$ satisfies $m \in (0, m_0]$ and  $E = E^+(m,h)$, where $E^+$ is defined in~\eqref{eplus} above. 
\end{lmm}
\begin{proof}
From Lemma \ref{emfacts3} and Lemma \ref{emfacts4} it follows that $\mr(E_0, m_0,h)$ is simply a region enclosed between the graphs of two continuous functions on a closed interval and therefore, is a compact subset of $\rr^2$. It is clearly non-empty. By continuity of $\Theta$, this shows that $\mm(E_0,m_0,h)$ is also compact and non-empty. 

It is obvious that $m > 0$, since $\Theta(E, m,h)=-\infty$ if $m = 0$. 
Let $E^-$ and $E^+$ be as in \eqref{eplus}. Let $E^*(m,h) := dm/h^2$. Lemma \ref{emax} implies that for any $m\in [0, m_0]$, 
\begin{align*}
E^-(m,h) &= E_0 - \frac{2d(m_0-m)}{h^2}\\
&< \frac{dm_0}{h^2} - \frac{2d(m_0-m)}{h^2} \\
&= \frac{dm}{h^2} - \frac{d(m_0-m)}{h^2}\le E^*(m,h). 
\end{align*}
Moreover, $E^*(m,h)$ is non-negative, and hence 
\[
E^*(m,h) \ge \max\{E^-(m,h),0\}.
\] 
Now fix any $m\in [0, m_0]$. If $E^+(m,h) < 0$, then there exists no $E$ such that $(E,m)\in \mr(E_0, m_0, h)$, and hence no $E$ such that $(E,m)\in \mr(E_0, m_0, h)$. So assume that $E^+(m,h) \ge 0$. By Proposition \ref{psiprops}, $\Theta(E, m,h)$ increases strictly  as $E$ increases from $0$ to $E^*(m,h)$, and then starts decreasing as $E$ increases further. Therefore, if we impose the restriction that $(E,m)\in \mr(E_0, m_0,h)$, then for fixed $m$, $\Theta(E,m,h)$ is maximized at $\min\{E^+(m,h), E^*(m,h)\}$. 

Suppose that $(E,m)\in \mm(E_0, m_0,h)$ is such that $E = E^*(m,h)< E^+(m,h)$. This is clearly not true if $m=m_0$, since $E^*(m_0,h) = dm_0/h^2 > E_0=E^+(m_0,h)$. We claim that this is impossible even if $m < m_0$. Indeed, if this is true for some $m < m_0$, then since $E_{\min}$ is a continuous function by Lemma \ref{emfacts4}, we can choose a slightly larger $m'> m$ such that $E^*(m',h) < E^+(m',h)$. But then $E^*(m',h) \in [E^-(m',h), E^+(m',h)]$, and 
\begin{align*}
\Theta(E^*(m',h), m',h) &= \log m' - \Psi_d\biggl(\frac{2h^2 E^*(m',h)}{m'}\biggr) \\
&= \log m' > \log m = \Theta(E, m,h), 
\end{align*}
showing that $(E,m)$ cannot belong to $\mm(E_0, m_0,h)$. 
\end{proof}

\section{Upper bound}\label{upperbound}
Fix $h>0$, and some $E_0\in \rr$ and $m_0 > 0$. 
The numbers $p$, $d$, $h$, $E_0$ and $m_0$ will be fixed throughout this section and will be called the `fixed parameters'. Any constant that depends only on the fixed parameters will be denoted simply by $C$, instead of $C(p,d,h,E_0, m_0)$. If the constant depends on additional parameters $a,b,\ldots$, then it will be denoted by $C(a,b,\ldots)$.  

Recall the random function $\phi$ defined in Section \ref{gaussian} and the objects $\mm$, $\mr$, $\Theta$ and  $\widehat{\Theta}$ defined in Section \ref{variational}.

Take any positive integer $n$. For any $\delta > 0$, we will call a function $f\in \cc^{V_n}$ a $\delta$-soliton if there exists a $g\in \cc^{V_n}$ such that
\begin{enumerate}
\item[(a)] $\|f-g\|_\infty \le \delta$, and 
\item[(b)] there exists $(E^*,m^*)\in \mm(E_0, m_0,h)$ such that 
\begin{align*}
&|(E_0-E^*)-H_{h,n}(g)|\le \delta  \ \text{ and } \\ 
&|(m_0-m^*)-M_{h,n}(g)|\le \delta.
\end{align*} 
\end{enumerate}
\begin{thm}\label{upper}
For arbitrary $\ep, \delta \in (0,1)$,  let $B = B(\ep, \delta, n)$ be the event 
\[
\{|H_{h,n}(\phi)-E_0|\le \ep, \ |M_{h,n}(\phi)-m_0|\le \ep, \ \phi \textup{ is not a $\delta$-soliton}\}. 
\]
Then for any fixed $\delta\in (0,1)$,
\[
\limsup_{\ep\ra0}\limsup_{n\ra\infty}\frac{\log \pp(B(\ep, \delta, n))}{n^d} < 1-m_0+\widehat{\Theta}(E_0, m_0,h).
\]
\end{thm}
The strict inequality is the main point of the above theorem. Let us now embark on the proof of Theorem \ref{upper}. We begin with two simple technical lemmas.
\begin{lmm}\label{gamma}
Let $z_1,\ldots, z_k$ be standard complex Gaussian random variables, where $k\ge 2$. Let $S:= \sum_{i=1}^k |z_i|^2$. Then for any $x \ge 2$ and $0 < y \le x/2$,
\begin{equation}\label{gamma1}
\pp(|S-x| \le y) =  \exp(k + k \log (x/k)  - x + R(x,y,k)),
\end{equation}
where 
\[
|R(x,y,k)|\le C \log k + C\log x + \frac{Cky}{x} + y,
\]
where $C$ is a universal constant. Moreover,
\begin{equation}\label{gamma2}
\pp(S\le x) \le \exp(k + k\log (x/k)). 
\end{equation}
\end{lmm}
\begin{proof}
The random variable  $S$ has a Gamma density with parameters $k$ and~$1$. Explicitly, the density function of $S$ is:
\[
\rho(t) = \frac{t^{k-1} e^{-t}}{(k-1)!}. 
\]
Let $T(x,k) := k + k \log (x/k)  - x$. Note that for any $t\in [x-y, x+y]$, 
\begin{align*}
&|\log \rho(t) - T(x,k)|\\
&\le |\log (k-1)! - (k\log k - k) | + |\log t| + k|\log t - \log x| + |x-t|\\
&\le C\log k + C\log x + \frac{Cky}{x} + y. 
\end{align*}
(The assumption that $y\le x/2$ was used to bound second and third terms.) Further, note that 
\begin{align*}
\pp(|S-x|\le y) &= \int_{x-y}^{x+y} \rho(t) dt\\
&= e^{T(x,k)}\int_{x-y}^{x+y} e^{\log \rho(t)-T(x,k)} dt. 
\end{align*}
Using the bound from the previous display finishes the proof of the first part of the lemma. To prove the second part, note that 
\begin{align*}
\pp(S\le x) &= \int_0^x \frac{t^{k-1}e^{-t}}{(k-1)!} dt \\
&\le \int_0^x \frac{t^{k-1}}{(k-1)!} dt = \frac{x^k}{k!}\le \frac{x^k}{k^k e^{-k}}.
\end{align*}
This completes the proof. 
\end{proof}

\begin{lmm}\label{pside}
Let $\Psi_{d,\ep}$ be the function defined in Theorem \ref{psifacts2}. Then for any $\alpha \in [0,\infty)$ and any $L > 0$, 
\[
\Psi_{d,\ep}(\alpha) \ge \min\{\Psi_d(\alpha),\; L\} - a(\ep,L,d),
\]
where $a(\ep,L,d)$ is a quantity that depends only on $\ep$, $L$ and $d$ (and not on $\alpha$), such that for any fixed $L>0$, $\lim_{\ep \ra0} a(\ep,L,d)=0$. 
\end{lmm}
\begin{proof}
In this proof, $a(\ep, L,d)$ will denote any constant with the properties described above.

By the properties of $\Psi_d$ listed in Proposition \ref{psiprops}, there exist $0 < c_1 < c_2 < 4d$ such that $\Psi_d(\alpha) > L$ whenever $\alpha \not \in [c_1, c_2]$. Fix $c_1'$ and $c_2'$ such that $0< c_1'< c_1 < c_2 < c_2'<4d$. Note that $c_1$, $c_2$, $c_1'$ and $c_2'$ can be chosen depending only on $L$ and $d$. By the uniform convergence of $\Psi_{d,\ep}$ on compact sets, if $\alpha \in [c_1', c_2']$, 
\[
\Psi_{d,\ep}(\alpha) \ge \Psi_d(\alpha)-a(\ep, L, d). 
\] 
On the other hand, by the definition of $\Psi_{d,\ep}$, it is easy to see that if $\alpha \not\in [c_1', c_2']$, for all sufficiently small $\ep$ (depending only on $L$ and $d$), $\Psi_{d,\ep}(\alpha) > L$. In particular, if $\alpha \not\in [c_1', c_2']$, then 
\[
\Psi_{d,\ep}(\alpha) \ge L - a(\ep, L,d). 
\]
The proof is completed by combining the two cases. 
\end{proof}

For a subset $U$ of $V_n$, recall that $U^c$ denotes the set $V_n\backslash U$, and $\partial U$ denotes the set of all vertices in $U$ that are adjacent to some vertex in $U^c$. Recall also the definitions of  $M_{h,n}(f,U)$, $H_{h,n}(f,U)$, $N_{h,n}(f,U)$ and $G_{h,n}(f,U)$ from Section \ref{possible}. 
The following lemma shows that the nonlinear component of any $f$ must come from a small region. 
\begin{lmm}\label{specialset}
Take any $f\in \cc^{V_n}$. Let $m:= M_{h,n}(f)$. Then, given any $\ep \in (0,1)$,  there exists a non-empty subset $U$ of $V_n$ such that 
\begin{enumerate}
\item[(i)] $|U|\le \frac{4^d m}{h^d \ep^{d+2}}$.
\item[(ii)] $|f(x)|\le \ep$ for all $x\in U^c$. 
\item[(iii)] $N_{h,n}(f, U^c) \le \ep^{p-1}m$.
\item [(iv)] $M_{h,n}(f, \partial U \cup \partial U^c) \le 2\ep m$.  
\end{enumerate}
\end{lmm}
\begin{proof}
Let $U_1$ be the subset of $V$ on which $f$ is bigger than $\ep$. If this set is empty, let $U_1$ be any singleton subset of $V_n$. Then note that $f\le \ep$ outside $U_1$, and 
\[
h^d\sum_{x\in U_1^c} |f|^{p+1}\le \ep^{p-1} h^d\sum_{x\in U_1^c} |f|^2 \le \ep^{p-1} m. 
\]
Also note that 
\begin{align*}
|U_1| \le \frac{h^d\sum_{x\in U_1} |f(x)|^2}{h^d \ep^2}\le \frac{m}{h^d \ep^2}. 
\end{align*}
Define $U_2,U_3,\ldots$ as follows: For each $i\ge 2$, let $U_i$ be the set of vertices that are either in $U_{i-1}$ or adjacent to a vertex in $U_{i-1}$. Note that $U_{i-1}\subseteq U_i$ for each~$i$. Let $W_i := \partial U_i$. Then $W_1, W_2,\ldots$ are disjoint sets, and $W_{i+1} = \partial U_i^c$ for each $i\ge 1$. Thus, for any $k$, 
\begin{equation*}\label{ikhd}
\begin{split}
\min_{1\le i\le k} h^d\sum_{x\in \partial U_i\cup \partial U_{i}^c} |f(x)|^2&= \min_{1\le i\le k} h^d\sum_{x\in W_i\cup W_{i+1}} |f(x)|^2 \\
&\le \frac{2h^d}{k}\sum_{i=1}^{k+1}\sum_{x\in W_i} |f(x)|^2 \le \frac{2m}{k}. 
\end{split}
\end{equation*}
Since each element of $U_k$ is within $\ell^1$ distance $k-1$ from some element of $U_1$, and the $\ell^1$ ball of radius $k-1$ around any vertex has $\le (2k-1)^d$ points, 
\begin{align*}
\max_{1\le i\le k}|U_i|= |U_{k}|\le (2k-1)^d|U_1|. 
\end{align*}
Choose an integer $k$ such that $1/\ep \le k\le 2/\ep$. The proof is completed by choosing an $i$ between $1$ and $k$ that  minimizes $h^d\sum_{x\in \partial U_i\cup \partial U_{i}^c} |f_x|^2$, and defining $U$ to be the set  $U_i$. 
\end{proof}

Given any $f\in \cc^{V_n}$, recall that $\MU(f, \ep)$ denote the set of all non-empty subsets of $V_n$ that satisfy conditions (i) through (iv) of Lemma \ref{specialset}. Then Lemma \ref{specialset} says that $\MU(f, \ep)$ is non-empty for any $\ep > 0$. 

\begin{lmm}\label{mainupperlemma}
For any $\ep \in (0,1)$, $E\ge 0$ and $m \ge 0$,  let $K = K(\ep, n,  E, m)$ be the event 
\begin{align*}
&\{|M_{h,n}(\phi)-m_0|\le \ep, \ |H_{h,n}(\phi)-E_0|\le \ep,  \textup{ and for some $U\in \MU(\phi, \ep)$, } \\
&\qquad \qquad |M_{h,n}(\phi, U^c)-m|\le \ep, \ |G_{h,n}(\phi, U^c) -E|\le \ep\}. 
\end{align*}
Then for any $L > 0$,
\begin{align*}
& \limsup_{n\ra\infty}\frac{\log \pp(K(\ep, n, E,m))}{n^d}\\
&\qquad \le  \max\{1-m_0 + \Theta(E,m,h), \; C-L\} + a(\ep,L), 
\end{align*}
where $C$ depends only on the fixed parameters, and $a(\ep,L)$ is a quantity that depends only on $\ep$, $L$ and the fixed parameters, such that for any fixed $L>0$, $\lim_{\ep \ra0} a(\ep,L)=0$. In particular, $a(\ep,L)$ does not depend on $E$ and $m$. 
\end{lmm}
\begin{proof}
Fix $L >0$, $E\ge 0$, $m\ge 0$ and $\ep\in (0,1)$. Throughout this proof, $a(\ep, L)$ will denote any constant with the properties outlined in the statement of the lemma, and $o(1)$ will denote any constant that depends only on $\ep$, $E$, $m$, $L$, $n$ and the fixed parameters, that goes to zero as $n \ra\infty$ while keeping the other parameters fixed. 

Choose a positive integer $n$ and a set $U\subseteq V_n$. For notational simplicity, define  
\begin{equation}\label{mg}
\begin{split}
M &:= M_{h,n}(\phi, U), \\
H &:= H_{h,n}(\phi, U), \\
M' &:= M_{h,n}(\phi, U^c),\\
G' &:= G_{h,n}(\phi, U^c).
\end{split}
\end{equation}
Let $\phi'$ be an independent copy of $\phi$, and define 
\[
\tau(x) := 
\begin{cases}
\phi'(x) &\text{ if } x\in U,\\
\phi(x) &\text{ if } x\in U^c. 
\end{cases}
\]
Note that $\tau$ has the same distribution as $\phi$, and is independent of $(\phi(x))_{x\in U}$.

Let $K_0=K_0(\ep, n, U, E, m)$ be the event
\begin{align*}
& \{|M_{h,n}(\phi)-m_0|\le \ep, \ |H_{h,n}(\phi)-E_0|\le \ep, \ |M'-m|\le \ep, \\
&\qquad \qquad \ |G'-E|\le \ep, \  U\in \MU(\phi,\ep)\}. 
\end{align*}
Note that it is possible for $K_0$ to happen only if $m\le M' + \ep \le M_{h,n}(\phi) +\ep \le m_0+2\ep\le C$ and 
$E\le G' + \ep \le CM' + \ep\le C$. 
Therefore we will assume these upper bounds on $E$ and $m$  in what follows. We will also assume that  
\begin{equation}\label{ucond}
|U|\le \frac{4^d(m_0+\ep)}{h^d\ep^{d+2}},
\end{equation}
since without this condition, the event $K_0$ is impossible. 
Let $K_1$ be the event
\[
K_1 := \biggl\{M_{h,n}(\tau, U) \le \frac{2|U|}{n^d}\biggr\}. 
\]
Since $M+M' = M_{h,n}(\phi)$, therefore if $K_0$ happens, then 
\begin{align}\label{mmm}
|M-(m_0-m)| &\le |M+M'-m_0| + |M'-m|\le 2\ep. 
\end{align}
Again, if $K_0\cap K_1$ happens, then
\begin{align}
|M_{h,n}(\tau) -m| &\le |M'-m| + M_{h,n}(\tau, U)\nonumber \\
&\le \ep + \frac{2|U|}{n^d}\nonumber \\
&\le \ep + C\ep^{-(d+2)}n^{-d} =: \ep_1, \label{mmm2}
\end{align}
and 
\begin{align}
|G_{h,n}(\tau)-E| &\le |G'-E| + G_{h,n}(\tau, U) + \frac{h^{d-2}}{2}\sum_{x\in U, \ y\in U^c\atop x\sim y} |\tau_x-\tau_y|^2\nonumber \\
&\le \ep  + C \sum_{x\in U} |\tau_x|^2 + C \sum_{x\in \partial U^c} |\phi_x|^2\nonumber \\
&\le C \ep+  C\ep^{-(d+2)}n^{-d}=: \ep_2. \label{mmm3}
\end{align}
Let 
\begin{align*}
p_1 &:= \pp(|M-(m_0-m)|\le 2\ep),\\
p_2 &:= \pp( |M_{h,n}(\tau) - m|\le \ep_1, \ |G_{h,n}(\tau) -E|\le \ep_2). 
\end{align*}
Note that $K_0$ and $K_1$ are independent events, and there is a positive universal constant $C_0$ such that $\pp(K_1)\ge 1/C_0$. 
Thus, by \eqref{mmm}, \eqref{mmm2}, \eqref{mmm3} and the independence of $M$ and $\tau$, 
\begin{align}\label{pk0}
\pp(K_0) \le C_0 \pp(K_0)\pp(K_1)&= C_0 \pp(K_0\cap K_1) \le C_0 p_1 p_2. 
\end{align}
Define 
\[
\xi(x) := \frac{\tau(x)}{\bigl(\sum_{y\in V_n} |\tau(y)|^2\bigr)^{1/2}}.
\]
Then $\xi$ is uniformly distributed on the unit sphere of $\cc^{V_n}$. Note that 
\begin{align*}
2h^{2-d} G_{h,n}(\xi) = \sum_{x,y\in V_n \atop x\sim y} |\xi(x)-\xi(y)|^2 &= \frac{h^d}{M_{h,n}(\tau)} \sum_{x,y\in V_n \atop x\sim y} |\tau(x)-\tau(y)|^2 \\
&= \frac{2h^2 G_{h,n}(\tau)}{M_{h,n}(\tau)}. 
\end{align*}
Thus, if $|M_{h,n}(\tau)-m|\le \ep_1$ and $|G_{h,n}(\tau)-E|\le \ep_2$, then (since $m\le C$ and $E\le C$, as observed before), 
\begin{align*}
\biggl|2h^{2-d} G_{h,n}(\xi) - \frac{2h^2 E}{m}\biggr|&= 2h^2\biggl|\frac{G_{h,n}(\tau)}{M_{h,n}(\tau)} - \frac{E}{m}\biggr|\\
&\le \frac{C\ep_1+C\ep_2}{(m-\ep_1)m} =: \ep_3,
\end{align*}
provided that $m > \ep_1$. If $m \le \ep_1$, define $\ep_3 = \infty$. Now, it is a simple probabilistic fact that $\xi$ and $M_{h,n}(\tau)$ are independent. Hence
$p_2 \le p_3 p_4$, 
where 
\begin{align*}
p_3 &:= \pp(|M_{h,n}(\tau)-m|\le \ep_1), \\
p_4 &:= \pp\biggl( \biggl|2h^{2-d} G_{h,n}(\xi) - \frac{2h^2 E}{m}\biggr|\le \ep_3\biggr). 
\end{align*}
Thus, from \eqref{pk0} we have
\begin{equation}\label{pk02}
\pp(K_0)\le C_0 p_1 p_3 p_4.
\end{equation} 
Our next task is to get upper bounds for $p_1$, $p_3$ and $p_4$. To bound $p_1$, we consider two cases. First, if $m_0-m > 2\ep$, then we apply \eqref{gamma1} from Lemma~\ref{gamma} (with $k\le C(\ep)$, $x= n^d(m_0-m)$ and $y=n^d\ep$) to get 
\[
\frac{\log p_1}{n^d} = -(m_0-m) + \ep + o(1), 
\]
where recall that the notation $o(1)$ stands for a quantity depending only on $\ep$, $E$, $m$, $L$, $n$ and the fixed parameters, that goes to zero as $n \ra\infty$ with all else fixed. In particular, $o(1)$ does not depend on our choice of $U$. 

Next, if $m_0-m\le 2\ep$, apply \eqref{gamma2} from Lemma \ref{gamma} (with $k\le C(\ep)$ and $x=3n^d\ep$) to get
\[
\frac{\log p_1}{n^d} = o(1). 
\]
Combining the two cases, we get 
\begin{equation}\label{p1approx}
\frac{\log p_1}{n^d} \le -(m_0-m) + 3\ep + o(1). 
\end{equation}
We deal with $p_3$ similarly. If $m \ge 2\ep^{1/4}$, we apply \eqref{gamma1} with $k=n^d$, $x=n^d m$ and $y =  n^d\ep_1$ to get 
\begin{equation}\label{p3approx1}
\frac{\log p_3}{n^d} = 1+ \log m - m + C\ep^{3/4} + o(1). 
\end{equation}
When $m < 2\ep^{1/4}$, we apply \eqref{gamma2} with $k=n^d$ and  $x=6n^d \ep $ to get
\begin{equation}\label{p3approx2}
\frac{\log p_3}{n^d} \le 1+ C\log \ep + o(1). 
\end{equation}
Again, if $m \ge 2\ep^{1/4}$, then $\ep_3 \le C\ep^{1/2} + o(1)$. Therefore by Theorem~\ref{psifacts2} we have that if $m \ge 2\ep^{1/4}$, then 
\begin{align*}
\frac{\log p_4}{n^d} &\le  - \Psi_{d, C\ep^{1/2}}\biggl(\frac{2h^2E}{m}\biggr) + o(1).
\end{align*}
By Lemma \ref{pside}, this gives 
\begin{align}\label{p4approx2}
\frac{\log p_4}{n^d} &\le \max\biggl\{ - \Psi_{d}\biggl(\frac{2h^2E}{m}\biggr), \; -L\biggr\} + a(\ep, L) + o(1),
\end{align}
where recall that $a(\ep, L)$ stands for a quantity that depends only on $\ep$, $L$ and the fixed parameters, that goes to zero as $\ep \ra 0$ for any fixed $L$. In particular, $a(\ep, L)$ does not depend on $E$, $m$, $n$ or $U$. 

Combining \eqref{pk02}, \eqref{p1approx}, \eqref{p3approx1} and \eqref{p4approx2} and the observation that $m\le C$, we see that when $m \ge 2\ep^{1/4}$, we have
\begin{align*}
\frac{\log \pp(K_0)}{n^d} &\le \frac{\log p_1 + \log p_3 + \log p_4}{n^d} \\
&\le 1-m_0 + \log m \\
&\qquad + \max\biggl\{ - \Psi_{d}\biggl(\frac{2h^2E}{m}\biggr), \; -L\biggr\} + a(\ep, L) + o(1)\\
&\le \max\{1-m_0 + \Theta(E, m,h), \; C-L\} + a(\ep, L) + o(1).
\end{align*}
On the other hand, since 
\[
C\log \ep \le -L + a(\ep, L), 
\]
and $\log p_4 \le 0$ and $m\le C$, it follows from \eqref{pk02}, \eqref{p1approx} and \eqref{p3approx2} that when $m < 2\ep^{1/4}$,
\begin{align*}
\frac{\log \pp(K_0)}{n^d} &\le \frac{\log p_1 + \log p_3}{n^d}\\
&\le C-L + a(\ep,L) + o(1). 
\end{align*}
Combining the last two displays, we see that for all $\ep\in (0,1)$, $n\ge 1$, $E\ge 0$, $m\ge 0$ and $U$ satisfying \eqref{ucond}, we have
\begin{align*}
&\frac{\log \pp(K_0(\ep, n, U, E, m))}{n^d} \\
&\le \max\{1-m_0 + \Theta(E, m,h), \; C-L\} + a(\ep, L) + o(1). 
\end{align*}
Now note that $K$ can be written simply as
\[
K=\bigcup_{U\subseteq V_n\atop |U|\le \frac{4^d (m_0+2\ep)}{h^d \ep^{d+2}}}  K_0(\ep, n, U, E, m). 
\]
Since there are at most $e^{C(\ep)\log n}$ terms in the above union, this completes the proof of the lemma.
\end{proof}
\begin{lmm}\label{abclmm}
Fix $n$ and let $K=K(\ep, n, E, m)$ be the event defined in Lemma~\ref{mainupperlemma}. If $K$ happens, then there exists a function $\eta$ on $V_n$ such that
\begin{enumerate}
\item[(a)] $\|\phi-\eta\|_\infty\le\ep $,  
\item[(b)] $|(E_0-E)-H_{h,n}(\eta)|\le 2\ep$, and 
\item[(c)] $|(m_0-m)-M_{h,n}(\eta)|\le C\ep + C\ep^{p-1}$. 
\end{enumerate}
\end{lmm}
\begin{proof}
Suppose that $K$ has happened. Choose any $U\in \MU(\phi, \ep)$ satisfying the conditions of $K$ and let $M$, $H$, $M'$ and $G'$ be as in \eqref{mg}. Let $\eta$ be the (random) function
\begin{equation}\label{etadef}
\eta(x) := 
\begin{cases}
\phi(x) &\text{ if } x\in U,\\
0 &\text{ if } x\in U^c. 
\end{cases}
\end{equation}
Then
\begin{align*}
\|\phi-\eta\|_\infty &= \max_{x\in U^c} |\phi(x)| \le \ep. 
\end{align*}
Next, note that
\begin{equation}\label{mmmm}
\begin{split}
|M_{h,n}(\eta) - (m_0-m)|&= |M-(m_0-m)| \\
&= |M_{h,n}(\phi)-M' - (m_0-m)|\le 2\ep, 
\end{split}
\end{equation}
and 
\begin{equation}\label{hhhh}
\begin{split}
|H_{h,n}(\eta)-(E_0-E)| &\le |H_{h,n}(\phi)-E_0| + |G'-E| \\
&\qquad + |H_{h,n}(\phi)- G'- H_{h,n}(\eta)| \\
&\le 2\ep + C \sum_{x\in \partial U\cup \partial U^c} |\phi(x)|^2 + C\sum_{x\in U^c} |\phi(x)|^{p+1}\\
&\le C\ep + C\ep^{p-1}. 
\end{split}
\end{equation}
This completes the proof. 
\end{proof}
\begin{lmm}\label{kcharlmm}
Let $K(\ep,n,E,m)$ be as in Lemma \ref{mainupperlemma}. Given any $m_1 > 0$ and $E_1 > 0$, there exists $C_0 = C_0(E_1, m_1)$ and $C_1$ such that if $m\ge m_1$, $E\ge E_1$, $\ep\le C_0$ and $(E,m)$ is at $\ell^\infty$ distance greater than $C_1\ep + C_1\ep^{p-1}$ from the set $\mr(E_0, m_0,h)$, then for all $n \ge C_2(E_1, m_1, \ep)$, the event $K(\ep, n, E,m)$ is impossible. 
\end{lmm}
\begin{proof}
If $K(\ep, n, E, m)$ happens, then by Lemma \ref{abclmm},  there exists a function $\eta\in \cc^{V_n}$ such that 
\begin{align}
&|(m_0-m)-M_{h,n}(\eta)|\le 2\ep, \label{kmineq}\\
&|(E_0-E)-H_{h,n}(\eta)|\le C\ep + C\ep^{p-1}.\label{khineq}
\end{align}
Let $E' := E_0-H_{h,n}(\eta)$ and $m' := m_0 - M_{h,n}(\eta)$. Then the above inequalities may be rewritten as 
\begin{equation}\label{kmhineq}
|m-m'| \le 2\ep, \ \ |E- E'|\le C\ep + C\ep^{p-1}. 
\end{equation}
Now, if $m\ge m_1$ and if $\ep$ is small enough depending only on $m_1$ and the fixed parameters, then  by \eqref{kmineq}, 
\begin{equation}\label{mzero}
m'= m_0-M_{h,n}(\eta) \ge m-2\ep \ge m_1-2\ep \ge 0. 
\end{equation}
But by definition, $m'\le m_0$. Therefore,  $m' \in [0, m_0]$. Next, note that 
\[
E_{\min}(M_{h,n}(\eta), h,n) \le H_{h,n}(\eta) \le E_{\max}(M_{h,n}(\eta), h, n),
\]
which is the same as 
\[
E_0- E_{\max}(m_0-m', h,n) \le E' \le E_0 - E_{\min}( m_0-m', h,n).
\]
Again, if $E\ge E_1$, then for sufficiently small $\ep$ (depending only on $m_1$ and the fixed parameters), \eqref{khineq} implies that 
\begin{equation}\label{hzero}
E' = E_0-H_{h,n}(\eta) \ge E_1-C\ep-C\ep^{p-1}\ge 0.
\end{equation}
Combining the last two displays, we see that if $E\ge E_1$, $m\ge m_1$ and $\ep$ is sufficiently small, then 
\begin{equation}\label{eee}
\begin{split}
&\max\{E_0- E_{\max}(m_0-m', h,n), 0\} \\
&\le E' \le E_0 - E_{\min}(m_0-m', h,n). 
\end{split}
\end{equation}
By Lemma \ref{emfacts4} and the characterization \eqref{mr2} of the set $\mr(E_0, m_0, h)$ from Section \ref{variational}, the equations \eqref{mzero}, \eqref{hzero} and \eqref{eee} show that if $\ep \le C_0(E_1, m_1)$ and $n \ge C_2(E_1, m_1, \ep)$, then the event $K(\ep, n, E,m)$ implies that the point $(E', m')$ is within distance $\ep$ from the set $\mr(E_0, m_0, h)$. By \eqref{kmhineq}, the proof is done. 
\end{proof}
\begin{lmm}\label{qlmm}
For  any $\ep\in (0,1)$ and any closed set $A\subseteq [0,\infty)^2$, let $F = F(\ep, n,  A)$ be the event 
\begin{align*}
&\{|M_{h,n}(\phi)-m_0|\le \ep, \ |H_{h,n}(\phi)-E_0|\le \ep,  \textup{ and for some $U\in \MU(\phi, \ep)$, } \\
&\qquad \qquad (G_{h,n}(\phi, U^c), M_{h,n}(\phi, U^c))\in A\}. 
\end{align*}
Then for any such set $A$,  
\[
\limsup_{\ep \ra 0} \limsup_{n\ra\infty} \frac{\log \pp(F(\ep, n, A))}{n^d} \le 1-m_0 + \widehat{\Theta}(E_0, m_0,h, A),  
\]
where
\[
\widehat{\Theta}(E_0, m_0,h, A) :=  \sup_{(E, m)\in A\cap \mr(E_0, m_0,h)}\Theta(E,m,h).
\]
\end{lmm}
\begin{proof}
If $F(\ep, n, A)$ happens, then $M_{h,n}(\phi, U^c)$ and $G_{h,n}(\phi, U^c)$ are bounded by constants depending only on the fixed parameters. Also, $\mr(E_0, m_0,h)$ is a bounded set. Hence we can assume without loss of generality that $A$ is contained in some bounded region determined by  the fixed parameters. 

Then $A$ is compact, and therefore there is a minimal collection $\mc(\ep)$ of points in $A$ such that the union of $\ell^\infty$-balls of radius $\ep$ around these points covers $A$, and the size of this collection is bounded by a constant  depending only on $\ep$ and the fixed parameters. Then we have
\begin{align*}\label{diff0n}
F(\ep, n, A) \subseteq \bigcup_{(E,m)\in \mc(\ep)} K(\ep, n, E, m),
\end{align*}
and therefore 
\begin{equation}\label{probq}
\frac{\log\pp(F(\ep,n,A))}{n^d} \le \frac{\log |\mc(\ep)|}{n^d} + \max_{(E,m)\in \mc(\ep)} \frac{\log \pp(K(\ep, n, E, m))}{n^d}. 
\end{equation}
Fix $L > 0$. Fix $E_1 > 0$ and $m_1 > 0$ so small that whenever $E< E_1$ or $m< m_1$, 
\begin{equation}\label{mthetal}
1-m_0 + \Theta(E, m, h) \le -L. 
\end{equation}
(It is easy to see that this is possible, by first choosing $m_1$ so small that $1-m_0+\log m_1 \le -L$, and then choosing $E_1$ depending on $m_1$.) Let $C_0=C_0(E_1, m_1)$ be the constant from Lemma \ref{kcharlmm}, and assume that $\ep \le C_0$. Let $C_1$ be the second constant from Lemma \ref{kcharlmm}. 
Let $\mc'(\ep)$ be the set of points $(E, m)\in \mc(\ep)$ that are at $\ell^\infty$ distance $\le C_1\ep + C_1 \ep^{p-1}$ from the set $\mr(E_0, m_0, h)$, and satisfy $E\ge E_1$ and $m \ge m_1$. Let $\mc''(\ep)$ be the set of points in $(E, m)\in \mc(\ep)$ such that $E< E_1$ or $m < m_1$. Let $\mc'''(\ep)$ be the set of all remaining points in $\mc(\ep)$. Then by Lemma \ref{kcharlmm}, for each $(E, m)\in \mc'''(\ep)$,
\[
\lim_{n\ra\infty}\frac{\log\pp(K(\ep, n, E, m))}{n^d} = -\infty. 
\]
For any $(E,m)\in \mc''(\ep)$, Lemma \ref{mainupperlemma} and the inequality \eqref{mthetal} imply that
\begin{align*}
&\limsup_{n \ra\infty} \frac{\log\pp(K(\ep,n,E,m))}{n^d} \\
&\le \max\{1-m_0+ \max_{(E,m)\in \mc''(\ep)} \Theta(E, m, h), \; C-L\} + a(\ep,L)\\
&\le C-L +  a(\ep,L),
\end{align*}
where $a(\ep, L)$ is a quantity depending only on $\ep$, $L$ and the fixed parameters, such that for each fixed $L > 0$, $\lim_{\ep\ra0} a(\ep, L)=0$. 

Finally, again by Lemma \ref{mainupperlemma}, for any $(E, m)\in \mc'(\ep)$, 
\begin{align*}
&\limsup_{n \ra\infty} \frac{\log\pp(K(\ep,n,E,m))}{n^d} \\
&\le \max\{1-m_0+ \max_{(E,m)\in \mc'(\ep)} \Theta(E, m, h), \; C-L\} + a(\ep,L).
\end{align*}
Combining the last three displays with \eqref{probq}, we get
\begin{align*}
&\limsup_{n\ra\infty} \frac{\log \pp(F(\ep, n, A))}{n^d}\\ &\le \max\{1-m_0+ \max_{(E,m)\in \mc'(\ep)} \Theta(E, m, h), \; C-L\} + a(\ep,L).
\end{align*} 
Now, by the definition of $\mc'(\ep)$ and the continuity of $\Theta$, it follows easily that 
\begin{align*}
\limsup_{\ep\ra0}\max_{(E,m)\in \mc'(\ep)} \Theta(E, m, h) \le \widehat{\Theta}(E, m, h, A).
\end{align*}
Therefore,
\begin{align*}
&\limsup_{\ep \ra0}\limsup_{n\ra\infty} \frac{\log \pp(F(\ep, n, A))}{n^d}\\ &\le \max\{1-m_0+ \widehat{\Theta}(E, m, h, A), \; C-L\}.
\end{align*} 
Since $L$ was arbitrary, this completes the proof. 
\end{proof}

\begin{lmm}\label{last}
Take any $\ep, \delta\in (0,1)$. Let $A_\delta$ be the set
\begin{align*}
&\{(E, m)\in [0,\infty)^2: \max\{|E-E^*|, |m-m^*|\} \ge \delta/2 \\
&\qquad \textup{ for all } (E^*, m^*)\in \mm(E_0,m_0,h)\}. 
\end{align*}
If $\ep$ is sufficiently small depending only on $\delta$ and the fixed parameters, then for any $n$, $B(\ep, \delta, n)\implies F(\ep, n, A_\delta)$. 
\end{lmm}
\begin{proof}
Suppose that $B(\ep, \delta, n)$ has happened. Recall that  $\MU(\phi, \ep)$ is always non-empty by Lemma \ref{specialset}; choose any $U\in \MU(\phi, \ep)$. Let $E := G_{h,n}(\phi, U^c)$ and $m := M_{h,n}(\phi, U)$. Suppose that $|E-E^*|< \delta/2$ and $|m-m^*|< \delta/2$ for some $(E^*, m^*) \in \mm(E_0, m_0,h)$. We prove that this is impossible  by arriving at a contradiction. 

In the situation described above, the event $K(\ep, n, E, m)$ has happened. Thus, there is a function $\eta$ satisfying the conditions (a), (b) and (c) of  Lemma \ref{abclmm}. Consequently, we have
\begin{enumerate}
\item[(a1)] $\|\phi-\eta\|_\infty\le \ep$, 
\item[(b1)] $|(E_0-E^*)-H_{h,n}(\eta)|\le 2\ep + \delta/2$, and
\item[(c1)] $|(m_0-m^*)-M_{h,n}(\eta)|\le C\ep + C\ep^{p-1} + \delta/2$. 
\end{enumerate}
If $\ep$ is sufficiently small depending only on $\delta$ and the fixed parameters, then this shows that $\phi$ is a $\delta$-soliton, giving the desired contradiction. 
\end{proof}

\begin{proof}[Proof of Theorem \ref{upper}]
Fix $\delta\in (0,1)$. Choose $\ep \in (0,1)$ small enough to satisfy the criterion of Lemma \ref{last}. Then for all $n$, 
\[
\pp(B(\ep, \delta, n)) \le \pp(F(\ep, n, A_\delta)). 
\]
By the continuity of $\Theta$ and the fact that $A_\delta$ is a closed set not intersecting the region where $\Theta$ attains its maximum in $\mr(E_0, m_0, h)$, it follows that
\[
\widehat{\Theta}(E_0, m_0, h, A_\delta) < \widehat{\Theta}(E_0, m_0, h). 
\]
Lemma \ref{qlmm} now completes the proof. 
\end{proof}

\section{Exponential decay of solitons}\label{expodecay}
Fix $n$. Suppose that $q\in \cc^{V_n}$ is a ground state soliton, that is, it minimizes energy among all functions with a given mass. Then by standard Euler-Lagrange theory, $q$ satisfies 
\[
-\omega q = -\Delta q - |q|^{p-1} q, 
\]
for some real number $\omega$. This can be rewritten as 
\begin{equation}\label{solitoneq}
(2dh^{-2}+\omega) q(x) - h^{-2}\sum_{y\; : \; y\sim x} q(y) = |q(x)|^{p-1} q(x) \ \text{ for all } x\in V_n. 
\end{equation}
Let $m := M_{h,n}(q)$. The following theorem shows that $q$ must be exponentially decaying outside a small set.
\begin{thm}\label{expodecaythm}
There exists a subset $U$ of $V_n$, whose size can be bounded by a number depending only on $m$, $h$, $p$ and $d$, such that for all $x\in V_n$,
\[
|q(x)|\le A e^{-bD_U(x)},
\]
where $D_U(x)$ is the $\ell^1$ distance of $x$ from $U$, that is, the minimum of $|y-x|_1$ over all $y\in U$, and $A$, $b$ are positive constants depending only on $m$, $h$, $p$ and~$d$. Here $y-x$ means the difference of $y$ and $x$ modulo $n$ in each coordinate. 
\end{thm}
The proof of Theorem \ref{expodecaythm} follows closely the outline of the proof of exponential decay in the continuum case, as given in \cite[Proposition B.7]{tao06}.  The main difference is that in the discrete case, we have to deal with discrete Green's functions. The proof is divided into several lemmas. 
\begin{lmm}\label{omega}
If $p < 1+4/d$, then for any $m > 0$, $h >0$ and $n\ge C(p,d,h,m)$, 
\[
\omega > -\frac{E_{\min}(m,h)}{m}> 0.
\] 
\end{lmm}
\begin{proof}
Multiplying both sides of \eqref{solitoneq} by $\overline{q(x)}$ and summing over $x\in V_n$, we get
\begin{align*}
\omega \sum_{x\in V_n} |q(x)|^2 &= -2dh^{-2}\sum_{x\in V_n} |q(x)|^2 + h^{-2}\sum_{x,y\in V_n\atop x\sim y} \bigl(\overline{q(x)} q(y) + q(x) \overline{q(y)}\bigr)\\
&\qquad + \sum_{x\in V_n} |q(x)|^{p+1}\\
&= -h^{-2}\sum_{x,y\in V_n\atop x\sim y } |q(x)-q(y)|^2 + \sum_{x\in V_n} |q(x)|^{p+1}\\
&\ge -h^{-2}\sum_{x,y\in V_n\atop x\sim y } |q(x)-q(y)|^2 + \frac{2}{p+1}\sum_{x\in V_n} |q(x)|^{p+1} \\
&= - 2h^{-d}E_{\min}(m,h,n). 
\end{align*}
Note that by Lemma \ref{emfacts4}, $E_{\min}(m,h,n) \ra E_{\min}(m,h)$ as $n \ra\infty$ and by Lemma \ref{emin0}, $E_{\min}(m,h) < 0$. This completes the proof. 
\end{proof}

\begin{lmm}\label{invlap}
Let $r := 2d/(2d+\omega h^2)$. Let $p(x,y, k)$ be the probability that a simple symmetric random walk on the torus $V_n$ starting at $x$ at step $0$ is at $y$ at step $k$. Then the soliton $q$ satisfies for all $x\in V_n$ the identity 
\[
q(x) = \frac{h^2}{2d} \sum_{y\in V_n} \sum_{k=0}^\infty r^{k+1}p(x,y, k)|q(y)|^{p-1} q(y). 
\]
\end{lmm}
\begin{proof}
Given $q$ satisfying \eqref{solitoneq}, let $f$ be the function on the right-hand side in the above display. (Note that the series converges because $r < 1$, by Lemma \ref{omega}.) Our goal is to show that $q=f$. First, note that 
\begin{align*}
\sum_{z\; : \; z\sim x} f(z) &= \frac{h^2}{2d} \sum_{y\in V_n} \sum_{k=0}^\infty r^{k+1}|q(y)|^{p-1} q(y)\biggl(\sum_{z\; : \; z\sim x} p(z,y, k)\biggr).
\end{align*}
By the translation invariance of the torus, it is easy to see that
\begin{align*}
\sum_{z\; : \; z\sim x} p(z,y, k) = \sum_{w \; : \; w\sim y} p(x,w,k). 
\end{align*}
Again, note that the random walk can be at $y$ at time $k+1$ if and only if it was at some neighbor of $y$ at time $k$ and moved to $y$ at the $(k+1)$th step. Therefore,
\[
p(x, y, k+1) = \frac{1}{2d} \sum_{w\; : \; w\sim y} p(x, w, k). 
\]
Combining the last three displays, we get
\begin{align*}
\sum_{z\; : \; z\sim x} f(z) &= h^2\sum_{y\in V_n} \sum_{k=0}^\infty r^{k+1}p(x,y, k+1)|q(y)|^{p-1} q(y)\\
&= \frac{h^2}{r}  \sum_{y\in V_n} \sum_{k=0}^\infty r^{k+1}p(x,y,k)|q(y)|^{p-1} q(y) \\
&\qquad - h^2\sum_{y\in V_n} p(x,y,0) |q(y)|^{p-1} q(y)\\
&= (2d+\omega h^2) f(x) - h^2|q(x)|^{p-1} q(x). 
\end{align*}
Comparing with \eqref{solitoneq}, this shows that for all $x\in V_n$,
\[
(2d+\omega h^2) (f(x)-q(x)) - \sum_{y\; : \; y\sim x} (f(y)-q(y)) = 0. 
\]
In other words, $(\omega I -\Delta) (f-q)=0$, where  $I$ is the identity matrix in $\cc^{V_n\times V_n}$. Since $\Delta$ is a negative semidefinite operator (Lemma \ref{diaglmm}) and $\omega > 0$ by Lemma \ref{omega}, $\omega I -\Delta$ is non-singular. This shows that $q=f$ and completes the proof. 
\end{proof}
For any $\delta > 0$, let 
\[
U_\delta := \{x\in V_n: |q(x)| > \delta\}. 
\]
For $x\in V_n$, let $D_\delta(x)$ denote the $\ell^1$ distance of $x$ from $U_\delta$, that is, the minimum of $|x-y|_1$ over all $y\in U_\delta$. Here, as usual the difference $x-y$ is computed modulo $n$ in each coordinate. 
\begin{lmm}\label{qbound}
For each $x\in V_n$ and $\delta > 0$, 
\[
|q(x)|\le C_0r^{D_\delta(x)} + \omega^{-1}\delta^{p-1} \sum_{y\in V_n} r^{|y-x|_1} |q(y)|,
\]
where $C_0 =\omega^{-1} m^{p+1} h^{-d(p-1)}\delta^{-2}$ and $r= 2d/(2d+\omega h^2)$. 
\end{lmm}
\begin{proof}
A random walk starting at $x$ at time $0$ cannot reach $y$ before time $|y-x|_1$. Thus, $p(x,y,k) = 0$ for all $k< |y-x|_1$. By Lemma \ref{invlap}, this gives 
\begin{align*}
|q(x)|&\le \frac{h^2}{2d}\sum_{y\in V_n} |q(y)|^p \biggl(\sum_{k=|x-y|_1}^\infty r^{k+1}\biggr)\\
&= \omega^{-1}\sum_{y\in V_n} r^{|x-y|_1} |q(y)|^p.  
\end{align*}
Now, if $y\not \in U_\delta$, then $|q(y)|^p \le \delta^{p-1} |q(y)|$. On the other hand, if $y\in U_\delta$, then  $|y-x|_1 \ge D_\delta(x)$. But again, $|q(x)|^2\le h^{-d} m$ for all $x$ and 
\[
|U_\delta| \le \frac{h^d \sum_{x\in U_\delta} |q(x)|^2}{h^d \delta^2}\le \frac{m}{h^d\delta^2}. 
\]
Thus,
\begin{align*}
\sum_{y\in V_n} r^{|x-y|_1} |q(y)|^p &\le \frac{m}{h^d \delta^2}r^{D_\delta(x)} (h^{-d} m)^p  + \delta^{p-1}\sum_{y\not \in U_\delta} r^{|y-x|_1} |q(y)|. 
\end{align*}
This completes the proof. 
\end{proof}
\begin{proof}[Proof of Theorem \ref{expodecaythm}] 
Define 
\[
B(x) := C_0\max_{y\in V_n} r^{\frac{1}{2}|y-x|_1+D_\delta(y)}, 
\]
where $C_0$ is the constant from Lemma \ref{qbound}. 
Note that $B(x)$ is never zero, $B(x) \ge C_0r^{D_\delta(x)}$ for all $x$, and for all $x,y$,
\begin{align*}
B(y) &= C_0\max_{z\in V_n} r^{\frac{1}{2}|z-y|_1+D_\delta(z)}\\
&\le C_0\max_{z\in V_n} r^{\frac{1}{2}|z-x|_1 - \frac{1}{2}|y-x|_1+D_\delta(z)}\\
&= r^{-\frac{1}{2}|y-x|_1} B(x). 
\end{align*}
Let $K$ be the smallest number such that $|q(x)|\le K B(x)$ for all $x\in V_n$. Since $B$ is never zero on $V_n$ and $V_n$ is a finite set,  $K$ must be finite. By Lemma \ref{qbound}, Lemma \ref{omega}, and the above observations,
\begin{align*}
|q(x)| &\le C_0r^{D_\delta (x)} + K\omega^{-1}\delta^{p-1} \sum_{y\in V_n} r^{|y-x|_1} B(y)\\
&\le B(x) + K B(x)\omega^{-1}\delta^{p-1} \sum_{y\in V_n} r^{\frac{1}{2}|y-x|_1}\\
&\le B(x) + K B(x) \delta^{p-1} C(p,d,h,m). 
\end{align*}
If $\delta$ is chosen so small that $\delta^{p-1}C(p,d,h,m) \le 1/2$, then the above inequality implies that
\[
K \le 1 + \frac{K}{2}. 
\]
In other words, $K\le 2$. Thus, with such a  choice of $\delta$, 
\[
|q(x)|\le 2C_0 \max_{y\in V_n} r^{\frac{1}{2}|y-x|_1 + D_\delta(y)}
\]
for all $x\in V_n$. To complete the proof, note that for any $y$,
\begin{align*}
\frac{1}{2}|y-x|_1 + D_\delta(y)  &\ge \frac{1}{2}|y-x|_1 +  \frac{D_\delta(y) }{2}\\
&\ge  \frac{1}{2}|y-x|_1 + \frac{D_\delta(x) - |y-x|_1}{2} = \frac{D_\delta(x)}{2}. 
\end{align*}
Thus, $|q(x)|\le 2C_0 r^{D_\delta(x)/2}$ for all $x\in V_n$. 
\end{proof}

\section{Lower bound}\label{lowerbound}
Fix $h>0$, and some $E_0\in \rr$ and $m_0 > 0$ such that 
\begin{equation}\label{focusingcase}
E_{\min}(m_0,h)<  E_0 < \frac{dm_0}{h^2}.
\end{equation}
Let $(E^*, m^*)$ be a point in $\mm(E_0, m_0,h)$. By Lemma \ref{rlmm}, $m^*$ is strictly positive. The numbers $p$, $d$, $h$, $E_0$, $m_0$, $E^*$ and $m^*$ will be fixed throughout this section and will be called the `fixed parameters'. Any constant that depends only on the fixed parameters will be denoted simply by $C$, instead of $C(p,d,h,E_0, m_0, E^*, m^*)$. If the constant depends on additional parameters $a,b,\ldots$, then it will be denoted by $C(a,b,\ldots)$. 

Recall the random function $\phi$ defined in Section \ref{gaussian} and the objects $\mm$, $\mr$, $\Theta$ and  $\widehat{\Theta}$ defined in Section \ref{variational}. 
\begin{thm}\label{lower}
Assume the condition \eqref{focusingcase}. For arbitrary $\ep \in (0,1)$,  let $B_0 = B_0(\ep, n)$ be the event 
\[
\{|H_{h,n}(\phi)-E_0|\le \ep, \ |M_{h,n}(\phi)-m_0|\le \ep\}. 
\]
Then 
\[
\liminf_{\ep\ra0}\liminf_{n\ra\infty}\frac{\log \pp(B_0(\ep,  n))}{n^d} \ge 1-m_0+\widehat{\Theta}(E_0, m_0,h).
\]
\end{thm}
Fix $n$ and let $f^*$ be an element of $\cc^{V_n}$ such that $M_{h,n}(f^*)=m_0-m^*$ and $H_{h,n}(f^*)=E_{\min}(m_0-m^*,h,n)$. Then $f^*$ is a ground state soliton for the DNLS on $V_n$. Fix $\ep > 0$. Let $U$ be a subset of $V_n$ such that 
\begin{align}\label{mfuu}
M_{h,n}(f^*, U^c\cup \partial U) \le \ep^2,
\end{align}
where, as before, $U^c = V_n \backslash U$. 
By Theorem \ref{expodecaythm}, there exists a $U$ satisfying the above property such that $|U|\le C(\ep)$.

Let $z$ be a point chosen uniformly at random from $V_n$. Let $\phi'$ be an independent copy of $\phi$ and define
\[
\gamma(x) :=
\begin{cases}
\phi'(x) &\text{ if } x\in z+ U, \\
\phi(x) &\text{ otherwise.}
\end{cases}
\]
Here $z+U$ is the translate of $U$ by $z$ on the torus $V_n$; that is, the addition is modulo $n$ in each coordinate. 

Let $U'$ be the set of all points that are either in $U$ or adjacent to some point in $U$. In other words, $U' = U \cup \partial U^c$. Define the following events:
\begin{align*}
A_1 &:= \{|M_{h,n}(\phi)-m^*|\le \ep^2\},\\
A_2 &:= \biggl\{\biggl|\frac{2h^2G_{h,n}(\phi)}{M_{h,n}(\phi)}-\frac{2h^2E^*}{m^*}\biggr|\le \ep^2, \\
&\qquad \qquad \max_{x\in V_n} |\phi(x)|^2 \le n^{-d/2}\sum_{x\in V_n} |\phi(x)|^2\biggr\},\\
A_3 &:= \{ M_{h,n}(\phi, z+ U') \le 4m^*(2d+1)n^{-d}|U|\},\\
A_4 &:= \{|\phi'(x) - f^*(x-z)| \le n^{-2d} \text{ for all } x\in z+U\}. 
\end{align*}
Finally, let $A = A(\ep, n) := A_1 \cap A_2 \cap A_3 \cap A_4$.
\begin{lmm}\label{lowerlmm1}
Let $A$ and $\gamma$ be defined as above. We claim that if $A$ happens, then 
\begin{align}\label{mgamma}
|M_{h,n}(\gamma)-m_0|&\le 2\ep^2 + C(\ep) n^{-d}
\end{align}
and 
\begin{align}\label{hgamma}
|H_{h,n}(\gamma)-E_0| &\le C\ep^2 + b(\ep,n),
\end{align}
where $b(\ep,n)$ is a number depending only on the fixed parameters, $\ep$ and $n$ such that $\lim_{n\ra\infty} b(\ep,n)=0$. 
\end{lmm}
\begin{proof}
Suppose that $A$ has happened. 
To prove \eqref{mgamma}, note that
\begin{align}\label{mgamma1}
M_{h,n}(\gamma) &= M_{h,n}(\phi, z+U^c)+ M_{h,n}(\phi',z+ U).  
\end{align}
By $A_1$ and $A_3$,
\begin{align}
&|M_{h,n}(\phi, z+U^c)- m^*| \nonumber \\
&\le |M_{h,n}(\phi)- m^*| + M_{h,n}(\phi, z+U) \nonumber \\
&\le |M_{h,n}(\phi)- m^*| + M_{h,n}(\phi, z+U')\nonumber \\
&\le \ep^2 + C(\ep)n^{-d}.\label{mgamma2}
\end{align}
Now if $A_4$ holds, then by the inequality 
\[
|a^r - b^r| \le r |a-b|\max\{a^{r-1}, b^{r-1}\} \le r|a-b|(a+|a-b|)^{r-1}
\]
that holds for any $a,b>0$ and $r>1$, we see that 
for any $x\in U$ and any~$r>1$, 
\begin{align}
\bigl||f^*(x)|^r - |\phi'(z+x)|^r \bigr| &\le C(r) n^{-2d}. \label{a2b2}
\end{align}
Thus by $A_4$, \eqref{mfuu}, and the fact that $M_{h,n}(f^*)=m_0-m^*$, we have 
\begin{align}
&| M_{h,n}(\phi',z+ U) - (m_0-m^*)| \nonumber \\
&\le |M_{h,n}(f^*, U) - (m_0-m^*)| + C|U|n^{-2d}\nonumber \\
&\le M_{h,n}(f^*, U^c) + C(\ep) n^{-2d} \nonumber \\
&\le \ep^2 + C(\ep) n^{-2d}.\label{mgamma3}
\end{align}
Combining \eqref{mgamma1}, \eqref{mgamma2} and \eqref{mgamma3} gives
\[
|M_{h,n}(\gamma) - m_0| \le 2\ep^2 + C(\ep)n^{-d}. 
\]
This proves \eqref{mgamma}. 
Next, note that 
\begin{equation}\label{disp0}
\begin{split}
H_{h,n}(\gamma) &= H_{h,n}(\phi, z+U^c) + H_{h,n}(\phi',z+ U) \\
&\qquad + \frac{h^{d-2}}{2}\sum_{x\in z+U, \; y\in z+U^c\atop x\sim y} |\phi'(x)-\phi(y)|^2 
\end{split}
\end{equation}
By \eqref{mfuu}, \eqref{a2b2} and $A_4$, 
\begin{align*}
\sum_{x\in z+\partial U} |\phi'(x)|^2 &\le \sum_{x\in \partial U} |f^*(x)|^2 + C(\ep) n^{-2d} \\
&\le C\ep^2 + C(\ep) n^{-2d}. 
\end{align*}
By $A_3$,
\begin{align*}
\sum_{x\in z+\partial U^c} |\phi(x)|^2 \le C(\ep) n^{-d}. 
\end{align*}
The last two displays imply that 
\begin{align}\label{disp1}
\frac{h^{d-2}}{2}\sum_{x\in z+U, \; y\in z+U^c\atop x\sim y} |\phi'(x)-\phi(y)|^2 \le C\ep^2 + C(\ep) n^{-d}. 
\end{align}
By $A_4$ and \eqref{a2b2}, 
\begin{align*}
|H_{h,n}(\phi',z+ U) - H_{h,n}(f^*, U)| &\le C(\ep) n^{-2d}. 
\end{align*}
Again from \eqref{mfuu}, it follows easily that
\[
|H_{h,n}(f^*, U)-H_{h,n}(f^*)|\le C\ep^2. 
\]
From the last two displays, we have
\begin{align}\label{disp2}
|H_{h,n}(\phi', z+U) - H_{h,n}(f^*)| \le C\ep^2 + C(\ep) n^{-2d}.
\end{align}
Next, note that 
\begin{align*}
H_{h,n}(\phi, z+U^c) &= G_{h,n}(\phi, z+U^c) - N_{h,n}(\phi, z+U^c). 
\end{align*}
By $A_3$,
\begin{align*}
|G_{h,n}(\phi, z+U^c)-G_{h,n}(\phi)|\le C M_{h,n}(\phi, z+U')\le C(\ep) n^{-d}. 
\end{align*}
Again by $A_1$ and $A_2$ and the fact that $m^*> 0$, it follows that if $\ep$ is sufficiently small (depending only on the fixed parameters), then
\begin{align*}
|G_{h,n}(\phi)-E^*|&\le C\ep^2. 
\end{align*}
Lastly, note that by $A_2$, 
\begin{align*}
\sum_{x\in z+U^c} |\phi(x)|^{p+1} &\le (\max_{x\in V_n} |\phi(x)|^{p-1}) \sum_{x\in V_n} |\phi(x)|^2\\
&\le Cn^{-d(p-1)/2}. 
\end{align*}
The last four displays combine to give 
\begin{align}\label{disp3}
|H_{h,n}(\phi, z+U^c)- E^*|&\le C\ep^2 + C(\ep) n^{-d} + Cn^{-d(p-1)/2}.
\end{align}
Combining \eqref{disp0}, \eqref{disp1}, \eqref{disp2} and \eqref{disp3} we get
\[
|H_{h,n}(\gamma)-(H_{h,n}(f^*)+E^*)|\le C\ep^2 + C(\ep) n^{-d} + Cn^{-d(p-1)/2}. 
\]
By Lemma \ref{emfacts4}, $H_{h,n}(f^*)\ra E_{\min}(m_0-m^*, h)$ as $n \ra\infty$. On the other hand by Lemma \ref{rlmm}, $E_{\min}(m_0-m^*, h) = E_0-E^*$. This completes the proof. 
\end{proof}

\begin{lmm}\label{lowerlmm2}
Let $A = A(\ep, n)$ be the event defined immediately before the statement of Lemma~\ref{lowerlmm1}. Then
\[
\lim_{\ep \ra 0}\lim_{n\ra\infty}\frac{\log \pp(A(\ep, n))}{n^d} = 1-m_0 + \widehat{\Theta}(E_0, m_0,h).
\]
\end{lmm}
(Note that the definition of $A(\ep, n)$ involves our choice of $f^*$. The above result holds for any  sequence of choices of $f^*$ as $n\ra\infty$.)
\begin{proof}
Let $A_1$, $A_2$, $A_3$ and $A_4$ be as in Lemma \ref{lowerlmm1}. Write 
\[
\pp(A) = \pp(A_1) \; \pp(A_2\mid A_1)\; \pp(A_3 \mid A_1\cap A_2) \pp(A_4\mid A_1\cap A_2 \cap A_3).
\]
For the first term, simply note that by Lemma \ref{gamma},
\begin{align*}
\pp(A_1) &= \exp(n^d + n^d \log m^* - n^dm^* + n^d o(\ep, n)), 
\end{align*}
where $o(\ep, n)$ is a term depending only on $\ep$, $n$ and the fixed parameters  such that 
\[
\limsup_{\ep \ra0}\limsup_{n \ra\infty}|o(\ep, n)| = 0. 
\]
Next,  define 
\[
\xi_x := \frac{\phi_x}{\bigl(\sum_{y\in V} |\phi_y|^2\bigr)^{1/2}}. 
\]
Then, as in the proof of Lemma \ref{mainupperlemma}, $\xi$ is uniformly distributed on the unit sphere of $\cc^{V_n}$, is independent of $M_{h,n}(\phi)$, satisfies  
\[
\frac{2h^2 G_{h,n}(\phi)}{M_{h,n}(\phi)} = \sum_{x,y\in V_n\atop x\sim y} |\xi(x)-\xi(y)|^2 
\]
and 
\[
\frac{\max_{x\in V_n} |\phi(x)|^2}{\sum_{x\in V_n} |\phi(x)|^2} = \max_{x\in V_n} |\xi(x)|^2. 
\]
Consequently by Theorem \ref{psifacts2} and Proposition \ref{psiprops},
\begin{align*}
\pp(A_2 \mid A_1) = \pp(A_2) &= \exp\biggl(-n^d \Psi_d\biggl(\frac{2h^2 E^*}{m^*}\biggr) + n^do(\ep, n)\biggr). 
\end{align*}
Let $z$ be as in the proof of Lemma \ref{lowerlmm1}.  Note that since  $z$ is uniformly distributed on $V_n$ and is independent of $\phi$, and therefore by Markov's inequality, 
\begin{align*}
\pp(A_3^c \mid \phi) &\le \frac{\ee(M_{h,n}(\phi, z+U')\mid \phi)}{4m^*(2d+1)n^{-d}|U|}\\
&= \frac{|U'|M_{h,n}(\phi)}{n^d}\frac{n^d}{4m^*(2d+1)|U|} \le \frac{M_{h,n}(\phi)}{4m^*}. 
\end{align*}
Thus, if $A_1$ happens and $\ep$ is sufficiently small (depending only on the fixed parameters), then $\pp(A_3 \mid \phi) \ge 1/2$. Consequently,
\[
1\ge \pp(A_3 \mid A_1 \cap A_2) \ge 1/2. 
\]
Lastly, note that since the coordinates of $\phi'$ are i.i.d.\ complex Gaussian with probability density function $(nh)^d \pi^{-1}\exp(-(nh)^d|x|^2)$, and are independent of $z$ and~$\phi$, and $|U|\le C(\ep)$; therefore by \eqref{mfuu}, 
\begin{align*}
\pp(A_4 \mid \phi, z) &= \pp(A_4)= \exp\biggl(- (nh)^d \sum_{x\in U}|f^*(x)|^2 + n^do(\ep, n) \biggr)\\
&= \exp(- n^d(m_0-m^*) + n^d o(\ep, n)). 
\end{align*}
Consequently, the same is true  for $\pp(A_4 \mid A_1 \cap A_2 \cap A_3)$. Combining the above estimates finishes the proof.
\end{proof}
\begin{proof}[Proof of Theorem \ref{lower}]
By Lemma \ref{lowerlmm1}, if $\ep$ is sufficiently small (depending only on the fixed parameters) and $n$ sufficiently large (depending only on $\ep$ and the fixed parameters), then the event $A(\ep, n)$ implies the event $B_0(\ep, n)$, but with $\phi$ replaced by $\gamma$. However, since $\gamma$ has the same distribution as $\phi$, by Lemma~\ref{lowerlmm2} this completes the proof of Theorem \ref{lower}. 
\end{proof}

\section{The radiating case}\label{radiatingsec}
Fix $h>0$, and some $E_0\in \rr$ and $m_0 > 0$ such that 
\[
\frac{dm_0}{h^2}\le E_0< \frac{2dm_0}{h^2}.
\]
As in Section \ref{upperbound}, the numbers $p$, $d$, $h$, $E_0$ and $m_0$ will be fixed throughout this section and will be called the `fixed parameters'. Any constant that depends only on the fixed parameters will be denoted simply by $C$, instead of $C(p,d,h,E_0, m_0)$. If the constant depends on additional parameters $a,b,\ldots$, then it will be denoted by $C(a,b,\ldots)$.  

Recall the random function $\phi$ defined in Section \ref{gaussian}. 
\begin{thm}\label{radiating}
Fix $\ep \in (0,1)$ and $\delta \in (0,1)$. Let $A_0$ be the event 
\[
\{|M_{h,n}(\phi)-m_0|\le \ep, \ |H_{h,n}(\phi)-E_0| \le \ep\}. 
\]
Then for any fixed $\delta \in (0,1)$, 
\begin{align*}
\limsup_{\ep\ra0} \limsup_{n\ra\infty} \frac{\log\pp(\max_{x\in V_n}|\phi(x)|> \delta \mid A_0)}{n^d} < 0. 
\end{align*}
\end{thm}
The rest of this section is devoted to the proof of this theorem. Define 
\[
\xi_x := \frac{\phi(x)}{\bigl(\sum_{y\in V_n} |\phi(y)|^2\bigr)^{1/2}}. 
\]
Then $\xi$ is uniformly distributed on the unit sphere of $\cc^{V_n}$. Let 
\[
\alpha_0 := \frac{2h^2E_0}{m_0},
\]
so that $\alpha_0\in [2d, 4d)$. 
Define three events:
\begin{align*}
A_1 &:= \{|M_{h,n}(\phi)-m_0|\le \ep^2\}. \\
A_2 &:= \biggl\{\biggl|\sum_{x,y\in V_n \atop x\sim y} |\xi(x)-\xi(y)|^2- \alpha_0\biggr|\le \ep^2\biggr\}.\\
A_3 &:= \{\max_{x\in V_n}|\xi(x)|^2\le n^{-d/2}\}. 
\end{align*}
\begin{lmm}\label{a0a1a2a3}
If $\ep < C$ and $n > C(\ep)$, then $A_1\cap A_2 \cap A_3 \implies A_0$. 
\end{lmm}
\begin{proof}
Note that 
\begin{align*}
\biggl|\sum_{x,y\in V_n \atop x\sim y} |\xi(x)-\xi(y)|^2- \alpha_0\biggr| &= 2h^2\biggl|\frac{G_{h,n}(\phi)}{M_{h,n}(\phi)}- \frac{E_0}{m_0}\biggr|. 
\end{align*}
This shows that if  $\ep< C$, then $A_1\cap A_2$ implies that  
\begin{equation}\label{ghne}
|G_{h,n}(\phi)-E_0| \le \ep^{3/2}. 
\end{equation}
Now, $A_1 \cap A_3$ implies 
\[
\max_{x\in V_n}|\phi(x)|^2 \le n^{-d/2} \sum_{x\in V_n} |\phi(x)|^2 \le Cn^{-d/2}
\]
and hence 
\begin{align*}
N_{h,n}(\phi) &\le C\bigl(\max_{x\in V_n}|\phi(x)|^2\bigr)^{(p-1)/2} \sum_{x\in V_n}|\phi(x)|^2\\
&\le Cn^{-d(p-1)/4}. 
\end{align*}
Therefore if $n> C(\ep)$, then $N_{h,n}(\phi) \le \ep^{3/2}$. Combining this with \eqref{ghne} completes the proof. 
\end{proof}
\begin{lmm}\label{radlower}
\begin{align*}
\liminf_{\ep\ra0} \liminf_{n\ra\infty} \frac{\log\pp(A_0)}{n^d}&\ge 1+\log m_0 - m_0 - \Psi_d(\alpha_0). 
\end{align*}
\end{lmm}
\begin{proof}
The random variable $M_{h,n}(\phi)$ and the random vector $\xi$ are independent. Therefore by Lemma \ref{a0a1a2a3}, for $\ep < C$ and $n > C(\ep)$, 
\begin{equation}\label{a01}
\pp(A_0)\ge \pp(A_1\cap A_2\cap A_3) = \pp(A_1)\pp(A_2\cap A_3). 
\end{equation}
By Theorem \ref{psifacts2},
\begin{equation}\label{a02}
\lim_{\ep\ra0} \lim_{n\ra\infty}\frac{\log \pp(A_2\cap A_3)}{n^d} = -\Psi_d(\alpha_0). 
\end{equation}
By Lemma \ref{gamma}, 
\begin{equation}\label{a03}
\lim_{\ep\ra0} \lim_{n\ra\infty}\frac{\log \pp(A_1)}{n^d} = 1+\log m_0-m_0.
\end{equation}
Combining \eqref{a01}, \eqref{a02} and \eqref{a03} proves the lemma.
\end{proof}
\begin{proof}[Proof of Theorem \ref{radiating}]
Define an event $A_4$ as
\[
A_4 := \{\max_{x\in V_n} |\phi(x)| \ge \delta\}. 
\]
We have to show that 
\[
\limsup_{\ep\ra0} \limsup_{n\ra\infty} \frac{\log\pp(A_0\cap A_4) - \log \pp(A_0)}{n^d} < 0. 
\]
If  $|H_{h,n}(\phi)- E_0|\le \ep$  and $\max_{x\in V_n} |\phi(x)| \ge \delta$, then 
\begin{align*}
G_{h,n}(\phi) &= H_{h,n}(\phi) + \frac{h^d}{p+1}\sum_{x\in V_n} |\phi(x)|^{p+1}\\
&\ge E_0 -\ep + C\delta^{p+1}. 
\end{align*}
Therefore, $A_0\cap A_4$ implies
\begin{align*}
\sum_{x,y\in V_n\atop x\sim y} |\xi(x)-\xi(y)|^2 &= \frac{2h^2G_{h,n}(\phi)}{M_{h,n}(\phi)}\\
&\ge \frac{2h^2(E_0-\ep+C \delta^{p+1})}{m_0+\ep}. 
\end{align*}
Thus, there is a constant $C_0$ depending only on the fixed parameters such that if $\ep < C(\delta)$, then $A_0\cap A_4$ implies the event $A_1\cap A_5$, where 
\[
A_5 := \biggl\{\sum_{x,y\in V_n\atop x\sim y} |\xi(x)-\xi(y)|^2 \ge \alpha_0 + C_0\delta^{p+1}\biggr\}. 
\]
Moreover, by the independence of $M_{h,n}(\phi)$ and $\xi$, the events $A_1$ and $A_5$ are independent. Therefore, if $\ep < C(\delta)$, then 
\begin{align*}
\pp(A_0\cap A_4) &\le \pp(A_1\cap A_5) = \pp(A_1)\pp(A_5). 
\end{align*}
Theorem \ref{psifacts} shows that for any $\ep < C(\delta)$,  
\[
\lim_{n\ra\infty} \frac{\log \pp(A_5)}{n^d} = - \Psi_d(\alpha_0 + C_0\delta^{p+1}),
\]
and we already know the limit of $n^{-d}\log \pp(A_1)$ from \eqref{a03}. Since $\Psi_d$ is a strictly increasing function in the interval $[2d, 4d)$ by Proposition \ref{psiprops}, a combination of the above inequality and Lemma \ref{radlower} proves the theorem.
\end{proof}

\section{Discrete concentration-compactness}\label{conccompsec}
Fix $h > 0$ and $m > 0$. Let $f_n$ be a sequence of functions on $\zz^d$ such that $M_h(f_n)=m$ for all $n$ and $H_h(f_n) \ra E_{\min}(m,h)$ as $n \ra\infty$. The following theorem is the main result of this section. 
\begin{thm}\label{conccomp}
There is a subsequence $n_k$ of natural numbers and a sequence of points $y_{k}\in \zz^d$ such that the sequence of functions $g_k$ defined as 
\[
g_k(x) := f_{n_k}(y_{k} + x)
\] 
converges to a limit function $g$ in the $L^q$ norm for every $q\in [2,\infty]$. The function $g$ has mass $m$ and  energy $E_{\min}(m,h)$. 
\end{thm}
The proof of Theorem \ref{conccomp} uses the well-known  concentration-compactness argument (see \cite[Section 1.4]{raphael08} and references therein). The main new challenge  in the discrete case is that the properties of $E_{\min}(m,h)$ are not as well-understood as those of $E_{\min}(m)$. 

For each $x\in \zz^d$ and each positive integer $R$, let $B(x,R)$ denote the $\ell^1$ ball of radius $R$ around $x$ in $\zz^d$. For any two positive integers $n$ and $R$, define the `concentration function'
\[
\rho_n(R) := \sup_{x\in \zz^d} \sum_{y\in B(x, R)} h^d |f_n(x)|^2. 
\]
Note that $\rho_n(R)$ is a non-decreasing function of $R$ for each $n$. Let
\[
\mu:= \lim_{R\ra\infty} \liminf_{n\ra\infty} \rho_n(R). 
\]
Clearly, there are sequences $R_k$ and $n_k$ increasing to infinity such that 
\[
\lim_{k\ra\infty} \rho_{n_k}(R_k)=\mu. 
\]
It is easy to fix $R_k$ such that $R_k$ is even for each $k$. 
Passing to a subsequence if necessary (using a diagonal argument), we may assume that 
\[
\rho(R) := \lim_{k\ra\infty} \rho_{n_k}(R)
\]
exists for each positive integer $R$. 
\begin{lmm}\label{cc1}
$\mu = \lim_{k\ra\infty} \rho_{n_k}(R_k) = \lim_{k\ra\infty} \rho_{n_k}(R_k/2) = \lim_{R\ra\infty} \rho(R)$. 
\end{lmm}
\begin{proof}
First observe that from the monotonicity of $\rho_{n_k}$, 
\[
\limsup_{k\ra\infty} \rho_{n_k}(R_k/2) \le \limsup_{k\ra\infty} \rho_{n_k}(R_k) = \mu. 
\]
On the other hand, for each $R$, 
\[
\rho(R) = \liminf_{k\ra\infty} \rho_{n_k}(R) \ge \liminf_{n\ra\infty} \rho_n(R),
\]
and thus
\[
\lim_{R\ra\infty} \rho(R) \ge \mu. 
\]
Lastly, for any $R$, we have $R_k/2\ge R$ for all $k$ large enough, and thus
\[
\rho_{n_k}(R_k/2)\ge \rho_{n_k}(R).
\]
Let $k\ra\infty$ gives that for all $R$,
\[
\liminf_{k\ra\infty} \rho_{n_k}(R_k/2) \ge \liminf_{k\ra\infty} \rho_{n_k}(R) = \rho(R). 
\]
Letting $R\ra\infty$ completes the proof.
\end{proof}
Since the function $\rho_n$ is bounded between $0$ and $m$ for every $n$, therefore $0\le \mu\le m$. 
\begin{lmm}\label{cc2}
$\mu> 0$. 
\end{lmm}
\begin{proof}
Suppose that $\mu=0$. Let all notation be as in Lemma \ref{cc1}. Then from Lemma \ref{cc1}, $\lim_{R\ra\infty} \rho(R)=0$. But $\rho$ is a non-decreasing, non-negative function. Therefore $\rho(R)=0$ for every $R$. In particular $\rho(1)=0$ and hence $\lim_{k\ra\infty} \rho_{n_k}(1)=0$. This implies that 
\[
\lim_{k\ra\infty} \sup_{x\in \zz^d}|f_{n_k}(x)| = 0, 
\]
and therefore 
\begin{align*}
\lim_{k\ra\infty}(H_h(f_{n_k}) - G_h(f_{n_k})) &= \lim_{k\ra\infty} \frac{h^d}{p+1}\sum_{x\in \zz^d} |f_{n_k}(x)|^{p+1}\\
&\le\lim_{k\ra\infty} \bigl(\sup_{x\in \zz^d}|f_{n_k}(x)|\bigr)^{p-1}\frac{M_h(f_{n_k})}{p+1}=0. 
\end{align*}
In particular, 
\[
\liminf_{k\ra\infty} H_h(f_{n_k}) \ge 0,
\]
contradicting $\lim_{k\ra\infty} H_h(f_{n_k}) = E_{\min}(m,h) < 0$ (Lemma \ref{emin0}). 
\end{proof}

\begin{lmm}\label{cc3}
There is a sequence of points $y_n$ in $\zz^d$ such that 
\[
\liminf_{n\ra\infty} |f_n(y_n)| > 0.
\] 
\end{lmm}
\begin{proof}
Let all notation be as in Lemma \ref{cc1}. By Lemma \ref{cc2}, 
\[
0 < \mu = \lim_{R\ra\infty} \liminf_{n\ra \infty} \rho_n(R). 
\]
Thus, there is some $R$ such that 
\[
\liminf_{n\ra\infty} \rho_n(R) > 0. 
\]
Fix such an $R$. Let $x_n(R)$ be a point such that 
\[
\rho_n(R) = \sum_{z\in B(x_n(R), R)} h^d |f_n(z)|^2. 
\] 
(The existence of such a point follows easily from the assumption that $M_h(f^n)=m<\infty$.) Since $R$ is fixed, this shows that 
\[
\liminf_{n\ra\infty} \frac{\sum_{z\in B(x_n(R), R)} |f_n(z)|^2}{|B(x_n(R), R)|} > 0. 
\]
But there exists a point $y_n\in B(x_n(R), R)$ such that 
\[
|f_n(y_n)|^2 \ge \frac{\sum_{z\in B(x_n(R), R)} |f_n(z)|^2}{|B(x_n(R), R)|}. 
\]
This completes the proof. 
\end{proof}
\begin{lmm}\label{cc4}
If $a$ and $b$ are positive real numbers and $\alpha > 1$, then 
\[
(a+b)^\alpha \ge a^\alpha + b^{\alpha}. 
\]
Moreover, for each $0 < c_1 < c_2$, there is a positive constant $ C= C(c_1,c_2, \alpha)$ such that whenever $a,b\in [c_1, c_2]$, 
\[
(a+b)^\alpha \ge a^\alpha + b^{\alpha} - C.
\]
\end{lmm}
\begin{proof}
The first inequality is a simple consequence of 
\[
\sup_{0< x <1} (x^\alpha+(1-x)^\alpha) =1,
\]
by putting $x=a/(a+b)$. The second assertion follows similarly.
\end{proof}
\begin{lmm}\label{cc5}
For any positive $m, m'$, 
\[
E_{\min}(m,h)+E_{\min}(m',h) > E_{\min}(m+m',h).
\] 
\end{lmm}
\begin{proof}
Let $f_n$ and $g_n$ be sequences of functions such that $M_h(f_n)=m$ and $M_h(g_n)=m'$ for all $n$, and 
\[
\lim_{n\ra\infty} H_h(f_n)=E_{\min}(m,h), \ \ \lim_{n\ra\infty} H_h(g_n) = E_{\min}(m',h). 
\]
Since for any function $f$, $M_h(|f|)= M(f)$ and $H_h(|f|)\le H_h(f)$ by the triangle inequality, we may assume that the functions $f_n$ and $g_n$ are non-negative real valued. By Lemma \ref{cc3}, there exist sequences of points $y_n$ and $z_n$ and a positive real number $\ep$ such that 
\[
\liminf_{n\ra\infty} f_n(y_n) > \ep, \ \ \liminf_{n\ra\infty} g_n(z_n) > \ep. 
\]
Define functions $v_n$ as 
\[
v_n(x) := f_n(x+y_n) + \I g_n(x+z_n), 
\] 
where $\I = \sqrt{-1}$. 
Since $f_n$ and $g_n$ are  real valued, it is clear that $M_h(v_n) = m+ m'$  and $G_h(v_n) = G_h(f_n) + G_h(g_n)$ for each $n$. Now note that since $p > 1$, Lemma~\ref{cc4} implies that for all $x$,  
\begin{align*}
|v_n(x)|^{p+1} &= |f_n(x+y_n) +\I g_n(x+z_n)|^{p+1} \\
&= ((f_n(x+y_n))^2 +(g_n(x+z_n))^2)^{(p+1)/2}\\
&\ge |f_n(x+y_n)|^{p+1} + |g_n(x+z_n)|^{p+1}. 
\end{align*}
If $n$ is large enough, then $f_n(y_n) >\ep$ and $g_n(z_n) > \ep$. Since $M_h(f_n)=m$ and $M_h(g_n)=m'$ for each $n$, therefore the sequences $f_n(y_n)$ and $g_n(z_n)$ are also uniformly bounded above. Therefore by Lemma \ref{cc4}, there is a positive constant $C$ such that for all $n$,
\begin{align*}
|v_n(0)|^{p+1} &= ((f_n(y_n))^2 +(g_n(z_n))^2)^{(p+1)/2}\\
&\ge |f_n(y_n)|^{p+1} +|g_n(z_n)|^{p+1} - C.
\end{align*}
Combining all of the above observations, it follows that 
\[
\limsup_{n\ra\infty} H_h(v_n) < \lim_{n\ra\infty} (H_h(f_n) + H_h(g_n)). 
\]
Since $M_h(v_n)= m+m'$ for all $n$, this completes the proof of the lemma. 
\end{proof}
\begin{lmm}\label{cc6}
$\mu=m$. 
\end{lmm}
\begin{proof}
By Lemma \ref{cc2} we know that $\mu > 0$. We also know from definition that $0\le \mu\le m$. So we only have to eliminate the case $0 < \mu < m$. Suppose that this is true. Let $x_n(R)$ be as in the proof of Lemma \ref{cc3}.  Then we can write
\[
f_{n_k} = u_k + v_k + w_k
\]
with
\begin{align*}
u_k(x) &= f_{n_k}(x) 1_{\{|x-x_{n_k}(R_k/2)|\le R_k/2\}},\\
v_k(x) &= f_{n_k}(x) 1_{\{|x-x_{n_k}(R_k/2)|> R_k\}},\\
w_k(x) &= f_{n_k}(x) 1_{\{R_k/2< |x-x_{n_k}(R_k/2)|\le R_k\}}.
\end{align*}
Then note that by Lemma \ref{cc1},
\begin{align*}
M_h(w_k) &= \sum_{x\in B(x_{n_k}(R_k/2), R_k)} h^d |u_{n_k}(x)|^2 - \sum_{x\in B(x_{n_k}(R_k/2), R_k/2)} h^d |u_{n_k}(x)|^2\\
&\le \rho_{n_k}(R_k) - \sum_{x\in B(x_{n_k}(R_k/2), R_k/2)} h^d |u_{n_k}(x)|^2\\
&= \rho_{n_k}(R_k)-\rho_{n_k}(R_k/2) \ra 0 \ \text{ as } k\ra\infty. 
\end{align*}
By Lemma \ref{cc1} it also follows that
\[
M_h(u_k) = \sum_{x\in B(x_{n_k}(R_k/2), R_k/2)} h^d |f_{n_k}(x)|^2 = \rho_{n_k}(R_k/2) \ra \mu \ \text{ as } k\ra\infty. 
\]
Thus, $\lim_{k\ra\infty} M_h(v_k) = m-\mu$. From similar arguments using Lemma \ref{cc1} it follows that  
\begin{equation}\label{hfnk}
\lim_{k\ra\infty} (H_h(f_{n_k}) - (H_h(u_k)+ H_h(v_k))) = 0. 
\end{equation}
By the continuity of $E_{\min}$ (Lemma \ref{emfacts4}), 
\[
\liminf_{k\ra\infty} H_h(u_k) \ge E_{\min}(\mu,h), \ \ \liminf_{k\ra\infty} H_h(v_k) \ge E_{\min}(m-\mu,h). 
\] 
Therefore by \eqref{hfnk}, 
\[
\liminf_{k\ra\infty} H_h(f_{n_k}) \ge E_{\min}(\mu,h)+E_{\min}(m-\mu,h). 
\]
By the initial assumption that $H_h(f_n) \ra E_{\min}(m,h)$, this shows that 
\[
E_{\min}(m,h)\ge  E_{\min}(\mu,h)+E_{\min}(m-\mu,h),
\]
contradicting Lemma \ref{cc5}. 
\end{proof}
\begin{proof}[Proof of Theorem \ref{conccomp}]
By Lemma \ref{cc6}, we know that 
\[
\lim_{R\ra\infty} \liminf_{n\ra\infty} \rho_n(R) = m.  
\] 
Choose $R_0$ so large that 
\[
\liminf_{n\ra\infty} \rho_n(R_0) > m/2. 
\]
Let $x_n(R)$ be as in the proof of Lemma \ref{cc3}. Then for all sufficiently large $n$,
\[
\sum_{x\in B(x_n(R_0), R_0)} h^d |f_n(x)|^2  > m/2. 
\]
Fix $\ep \in (0, m/2)$. Let $R_\ep$ be so large that 
\[
\liminf_{n\ra\infty} \rho_n(R_\ep) > m-\ep. 
\]
Then for all $n$ sufficiently large,
\[
\sum_{x\in B(x_n(R_\ep), R_\ep)} h^d |f_n(x)|^2  > m-\ep. 
\]
Since $m-\ep + m/2 > m$, and $M_h(f_n)=m$ for all $n$, therefore for sufficiently large $n$, the balls $B(x_n(R_0), R_0)$ and $B(x_n(R_\ep), R_\ep)$ cannot be disjoint. In particular, 
\[
|x_n(R_0)-x_n(R_\ep)|_1 \le R_0+R_\ep,
\]
and therefore, 
\[
B(x_n(R_\ep), R_\ep) \subseteq B(x_n(R_0), R_0 + 2R_\ep). 
\]
Thus, if we define 
\[
v_n(x) := f_n(x+ x_n(R_0)),
\]
then for every $\ep\in (0,m/2)$, there is a sufficiently large integer $S_\ep$ such that for all sufficiently large $n$, 
\[
\sum_{x \; : \; |x|_1 > S_\ep}h^d |v_n(x)|^2 \le \ep. 
\]
Thus, the sequence $v_n$ is compact in $L^2(\zz^d)$ and therefore has a convergence subsequence $v_{n_k}$, which we call $g_k$. Let $g$ denote the limit of $g_k$ in $L^2$. Since the convergence is in $L^2$, it follows automatically that $M_h(g) =m$. Again, since the $L^\infty$ norm of a function on $\zz^d$ is bounded above by its $L^2$ norm, it follows that $\|g_k-g\|_\infty$ also goes to zero. Therefore, since for any $q\in (2,\infty)$,
\[
\sum_{x\in \zz^d}|g_k(x)-g(x)|^q \le (\sup_{x\in\zz^d} |g_k(x)- g(x)|)^{q-2} \sum_{x\in \zz^d}|g_k(x)-g(x)|^2
\]
it follows that $g_k$ converges to $g$ in $L^q$ for any $q\in (2,\infty)$. This implies, in particular, that $H_h(g_k) \ra H_h(g)$. 
\end{proof}

\section{Harmonic analysis on the lattice}\label{harmonic}
In this section $p$ will not denote the nonlinearity parameter in the NLS. Instead, it will typically play the role of the $p$ in the $L^p$ norm.

We define the $L^p$ norm for functions on $\zz^d$ at grid size $h$ as follows:
\[
\|f\|_{p,h} := \biggl(h^d\sum_{x\in \zz^d} |f_x|^p\biggr)^{1/p} = h^{d/p}\|f\|_p. 
\]
It may seem strange to define a new norm by multiplying the standard $L^p$ norm by a constant;  the purpose of the definition is to ensure that the constants in the discrete analogs of classical inequalities (that we develop below) do not depend on the grid size~$h$. Note also that $\|f\|_{p,h} = \|\tf\|_{p}$, where $\tf$ is the continuum image of $f$ at grid size $h$. 

Similar inequalities were developed by Ladyzhenskaya \cite{ladyzhenskaya} in the context of the ``finite-difference method''. However, since I could not find in \cite{ladyzhenskaya} exactly what I needed  (in particular, some delicate analyses of discrete Green's functions and a discrete version of the Hardy-Littlewood-Sobolev inequality of fractional integration), I decided to go ahead with my own derivations.

The grid size $h$ will be fixed throughout this section. We will assume that $h\in (0,1)$, to avoid complications arising out of large values of $h$, in which we are not interested since we eventually want to send $h$ to zero.   
\subsection{Convolutions}
We define the convolution of two functions $f$ and $g$ on $\zz^d$ at grid size $h$ as 
\begin{align*}
(f*g)(x) := h^d\sum_{y\in \zz^d} f(y) g(x-y). 
\end{align*}
Although the notation does not explicit include the grid size, it will be understood from the context.

\subsection{Young's inequality}
Below we state and prove the discrete analog of Young's inequality for convolutions, at grid size $h$. The proof is exactly the same as in the continuous case; the important thing is that the constant does not depend on the grid size. 
\begin{prop}\label{young}
Let $f,g$ be complex-valued functions on $\zz^d$. Let 
$1\le p,q,r\le \infty$ satisfy 
\[
\frac{1}{p}+\frac{1}{q} = \frac{1}{r} + 1.
\] 
Then for any $h > 0$, 
\[
\|f*g\|_{r,h} \le \|f\|_{p,h} \|g\|_{q,h}. 
\]
\end{prop}
\begin{proof}
Let $\alpha = (r-p)/r$ and $\beta=(r-q)/r$, so that $\alpha$ and $\beta$ both belong to the interval $[0,1]$. Let $p_1 = p/\alpha$ and $p_2= q/\beta$, so that $p_1,p_2\in [1,\infty]$. Note that 
\[
\frac{1}{p_1}+\frac{1}{p_2}+\frac{1}{r}=1. 
\]
Let $u= f*g$. By H\"older's inequality,
\begin{align*}
&|u(x)| = h^d \biggl|\sum_{y\in \zz^d} f(x-y) g(y)\biggr|\\
&\le h^d \sum_{y\in \zz^d} |f(x-y)|^{1-\alpha} |g(y)|^{1-\beta} |f(x-y)|^\alpha |g(y)|^\beta\\
&\le h^d \biggl(\sum_{y\in \zz^d} |f(x-y)|^{(1-\alpha)r} |g(y)|^{(1-\beta)r}\biggr)^{1/r} \\
&\qquad \qquad \times \biggl(\sum_{y\in \zz^d} |f(x-y)|^{\alpha p_1}\biggr)^{1/p_1} \biggl(\sum_{y\in \zz^d} |g(y)|^{\beta p_2}\biggr)^{1/p_2}\\
&= \biggl(h^d\sum_{y\in \zz^d} |f(x-y)|^{(1-\alpha)r} |g(y)|^{(1-\beta)r}\biggr)^{1/r}\|f\|_{\alpha p_1,h}^\alpha\|g\|_{\beta p_2, h}^\beta. 
\end{align*}
Taking $r$th power and summing over $x$ gives 
\begin{align*}
\|u\|_r^r &\le \|f\|_{(1-\alpha)r,h}^{(1-\alpha)r}\|g\|_{(1-\beta)r,h}^{(1-\beta)r} \|f\|_{\alpha p_1,h}^{\alpha  r}\|g\|_{\beta p_2,h}^{\beta r}. 
\end{align*}
Since $(1-\alpha)r = \alpha p_1 = p$  and $(1-\beta)r = \beta p_2 = q$, this completes the proof. 
\end{proof}
\subsection{Fourier transform}
Take any $h > 0$ and let $K=1/h$. For a function $f:\zz^d \ra\cc$ with finite $L^1$ norm, we define the Fourier transform of $f$ at grid size $h$ as
\[
\hat{f}(\xi) := h^d \sum_{x\in \zz^d} f(x) e^{\I\pi h x\cdot \xi}, \ \xi \in [-K, K]^d,  
\]
where $\I = \sqrt{-1}$. 
Again, the notation does not explicitly include the grid size; it is to be understood from the context.  
It is easy to verify the inversion formula
\begin{align*}
f(x) = 2^{-d}\int_{[-K,K]^d} \hat{f}(\xi) e^{-\I \pi h x\cdot \xi} \; d\xi. 
\end{align*}
With the above definition, if $u = f*g$ (at grid size $h$), then
\[
\hat{u}(\xi) = \hat{f}(\xi) \hat{g}(\xi),
\]
provided that $f$ and $g$ are in $L^1 \cap L^2$. Another easy fact is the Plancherel identity
\begin{equation}\label{plancherel}
\|f\|_{2,h} = 2^{-d} \int_{[-K,K]^d} |\hat{f}(\xi)|^2 d\xi. 
\end{equation}
\subsection{Littlewood-Paley decomposition}
Fix $h > 0$ and let $K=1/h$, as in the preceding Subsection. Let $\gamma_0: [-K,K] \ra[0,1]$ be a smooth (i.e.\ $C^\infty$) function that is $1$ in $[-1,1]$ and $0$ outside $(-2,2)$. Define $\gamma: [-K,K]^d \ra [0,1]$ as 
\[
\gamma(\xi) := \prod_{i=1}^d \gamma_0(\xi_i).
\]
Note that $\gamma$ is simply $\hat{g}$, where $g:\zz^d \ra \cc$ is the function 
\begin{align*}
g(x) &= 2^{-d}\int_{[-K,K]^d}\gamma(\xi) e^{-\I \pi h x\cdot \xi} \; d\xi \\
&= 2^{-d}\prod_{i=1}^d\int_{-2}^2\gamma_0(\xi_i) e^{-\I \pi h x_i \xi_i} \; d\xi_i= \prod_{i=1}^d \varphi(\pi h x_i), 
\end{align*}
where $\varphi: \rr\ra\cc$ is the function
\begin{equation}\label{phiform}
\varphi(x) = \frac{1}{2}\int_{-2}^2\gamma_0(t) e^{-\I x t}\; dt.  
\end{equation}
For any $a \in (0, K]$, let $\gamma_a:[-K,K]^d\ra[0,1]$ be the function $\gamma(\xi/a)$.
A computation similar to the above shows that $\gamma_a$ is the Fourier transform of $g_a$, where
\[
g_a(x) =  a^d\prod_{i=1}^d  \varphi(a\pi h x_i). 
\]
\begin{lmm}\label{glemma}
For any $a \in (0,K]$ and any $p\in [1,\infty]$, 
\[
C_1(p,d)a^{d(p-1)/p}\le \|g_a\|_{p,h}\le C_2(p,d) a^{d(p-1)/p}, 
\]
where $C_1(p,d)$ and $C_2(p,d)$ are constants depending only on $p$ and $d$. 
\end{lmm}
\begin{proof}
Note that for any $p\in [1,\infty]$, 
\begin{align*}
\|g_a\|_{p,h}^p &= h^da^{dp} \biggl(\sum_{x\in \zz} |\varphi(a\pi h x)|^p \biggr)^d. 
\end{align*}
From the formula \eqref{phiform}, the properties of $\gamma_0$, and standard results about oscillatory integrals (see e.g.\ \cite[Chapter VIII, Proposition 1]{stein93}) it follows that $\varphi(x)$ decays faster than $|x|^{-\alpha}$ for any $\alpha$ as $|x|\ra\infty$. Moreover, $\varphi$ is a continuous function. Therefore, as $ah\ra0$, 
\[
a\pi h\sum_{x\in \zz} |\varphi(a\pi h x)|^p\ra \int_{\rr} |\varphi(u)|^p\; du\in (0,\infty).
\]
This shows that 
\begin{align*}
\frac{C_1(p)}{ah} \le \sum_{x\in \zz} |\varphi(a\pi h x)|^p\le \frac{C_2(p)}{ah},
\end{align*}
which completes the proof.
\end{proof}
For each $a$, let $P_a f$ be the function whose Fourier transform is 
\[
(\gamma_a(\xi) - \gamma_{a/2}(\xi)) \hat{f}(\xi).
\] 
In other words, $P_a f = (g_a-g_{a/2}) * f$. Now, for any nonzero $\xi \in [-K, K]^d$,
\begin{align*}
\sum_{j=0}^\infty (\gamma_{2^{-j}K}(\xi)-\gamma_{2^{-(j+1)}K}(\xi)) =  1,  
\end{align*}
which shows that for functions with suitable decay at infinity (e.g.\ functions with bounded support)
\[
f = \sum_{j=0}^\infty P_{2^{-j}K} f.
\]
This is the discrete analog (at grid size $h$) of the classical Littlewood-Paley decompositions (see e.g.\ \cite[Appendix A]{tao06}).  
\begin{lmm}\label{palemma}
For any $h > 0$, any $1\le p\le q$, any $a\in (0,K]$, and any $f:\zz^d \ra \cc$, 
\begin{align*}
\|P_a f\|_{q,h} &\le C(p,q,d) a^{\frac{d}{p}-\frac{d}{q}} \|f\|_{p,h}. 
\end{align*}
\end{lmm}
\begin{proof}
By Young's inequality (Lemma \ref{young}), for any $1\le p\le q$ and any $a\in (0,K]$,
\begin{align*}
\|P_a f\|_{q,h} &\le (\|g_a\|_{s,h} + \|g_{a/2}\|_{s,h}) \|f\|_{p,h},
\end{align*}
where 
\[
\frac{1}{s}=\frac{1}{q}+1 - \frac{1}{p}. 
\]
The proof is now easily completed using  Lemma \ref{glemma}.
\end{proof}
\subsection{Gagliardo-Nirenberg inequality}
The goal of this subsection is the prove a version of the Gagliardo-Nirenberg inequality  for the lattice at grid size $h$. (For the well-known continuum version, see e.g.\ \cite[Appendix A]{tao06}.) The result is stated as Proposition \ref{gn} below. To prepare for this, we first need some definitions and lemmas.

For $i=1,\ldots, d$, let $e_i$ denote the $i$th coordinate  vector. Let $\nabla_i$ denote the discrete derivative operator in the $i$th coordinate direction, defined as 
\begin{equation}\label{discderiv}
\nabla_i f(x) := \frac{f(x+e_i)-f(x)}{h}. 
\end{equation}
Note that $\nabla_i f$ is simply the convolution of $f$ with $\delta_i$, where $\delta_i$ is the function
\[
\delta_i(x) =
\begin{cases}
-h^{-d-1} &\text{ if } x = 0,\\
h^{-d-1} &\text{ if } x=-e_i,\\
0 &\text{ for all other } x. 
\end{cases}
\]
\begin{lmm}\label{gradf}
For any $\xi\in [-K,K]^d$, 
\[
\sum_{i=1}^d |\widehat{\nabla_i f}(\xi)|^2 \ge C|\xi|^2 |\hat{f}(\xi)|^2,
\]
where $|\xi|$ is the Euclidean norm of $\xi$ and $C$ is a positive universal constant. 
\end{lmm}
\begin{proof}
A straightforward verification shows that
\begin{align*}
\widehat{\nabla_i f}(\xi) &= \frac{e^{-\I \pi h\xi_i}-1}{h} \hat{f}(\xi). 
\end{align*}
There is a positive constant $C_0$ such that for all $\theta\in [-\pi, \pi]$, 
\[
|1-e^{-\I \theta}| \ge C_0|\theta|.
\]
Therefore, if $\xi_i \in [-K, K]$, then 
\[
|\widehat{\nabla_i f}(\xi)| \ge C|\xi_i| |\hat{f}(\xi)|. 
\]
This completes the proof of the lemma.
\end{proof}
\begin{lmm}\label{gnstep}
For any $a\in (0,K]$ and any $f\in L^2(\zz^d)$,
\[
\|P_a f\|_{2,h} \le C a^{-1} G_h(f),
\]
where $C$ is a universal constant. 
\end{lmm}
\begin{proof}
If $|\xi| \le a/2$, then $|\xi_i|\le a/2$ for each $i$, and hence 
\[
\gamma_a(\xi) = \gamma(\xi/a)= 1 = \gamma(2\xi/a)=\gamma_{a/2}(\xi).
\]
On the other hand, if $|\xi| > 2\sqrt{d}a$, then $|\xi_i|> 2a$ for some $i$, and therefore
\[
\gamma_a(\xi)=\gamma(\xi/a) = 0 = \gamma(2\xi/a)=\gamma_{a/2}(\xi).
\]
Combining these facts with Lemma \ref{gradf}, Lemma \ref{palemma} and the Plancherel identity \eqref{plancherel}, we~get 
\begin{align*}
\|P_a f\|_{2,h}^2 &= \int_{[-K,K]^d} |\widehat{P_a f}(\xi)|^2 \; d\xi \\
&= \int_{[-K,K]^d} |(\gamma_a(\xi)-\gamma_{a/2}(\xi)) \hat{f}(\xi)|^2 \; d\xi \\
&\le Ca^{-2}\int_{a/2\le |\xi| \le 2\sqrt{d}a} \sum_{i=1}^d|(\gamma_a(\xi)-\gamma_{a/2}(\xi)) \widehat{\nabla_if}(\xi)|^2\; d\xi\\
&\le Ca^{-2}\sum_{i=1}^d \|P_a \nabla_i f\|_{2,h}^2 \le Ca^{-2}\sum_{i=1}^d\|\nabla_i f\|_{2,h}^2= Ca^{-2}G_h(f). 
\end{align*}
This completes the proof. 
\end{proof}
The following proposition may be called the Gagliardo-Nirenberg inequality for the lattice with grid size $h$. The important thing, as usual, is that the constant does not depend on $h$. 
\begin{prop}\label{gn}
Take any $2<q\le \infty$ and let $\theta\in (0,1)$ solve 
\begin{equation}\label{qtheta}
\frac{1}{q}= \frac{1}{2}-\frac{\theta}{d}.
\end{equation}
Then for any $f\in L^q(\zz^d) \cap L^2(\zz^d)$, we have 
\[
\|f\|_{q,h} \le C(q,d) \|f\|_{2,h}^{1-\theta} G_h(f)^{\theta/2}. 
\]
\end{prop}
\begin{proof}
Then for any $f$, by Lemma \ref{palemma},
\begin{align}
\|f\|_{q,h} &\le \sum_{j=0}^\infty \|P_{2^{-j}K} f\|_{q,h}\nonumber\\
&\le C(q,d)\sum_{j=0}^\infty  (2^{-j} K)^{\frac{d}{2}-\frac{d}{q}} \|P_{2^{-j}K}f\|_{2,h}\nonumber \\
&= C(q,d) \sum_{j=0}^\infty (2^{-j}K)^{\theta} \|P_{2^{-j}K} f\|_{2,h}.\label{gnstep1}
\end{align} 
Again by Lemma \ref{palemma}, 
\begin{align}\label{gnstep2}
\|P_{2^{-j}K} f\|_{2,h} \le \|f\|_{2,h},
\end{align}
and by Lemma \ref{gnstep}, there is a universal constant $C_0$ such that for all $j\ge 0$, 
\begin{align}\label{gnstep3}
\|P_{2^{-j}K} f\|_{2,h} &\le C_0 2^{j}h G_h(f)^{1/2}. 
\end{align}
Let $j_0\in \zz$ be the unique integer  such that 
\begin{align*}
C_0 2^{j_0-1} h G_h(f)^{1/2}\le \|f\|_{2,h} \le C_0 2^{j_0} h G_h(f)^{1/2}.
\end{align*}
Then note that 
\begin{align}\label{gnstep4}
\sum_{j> j_0} 2^{-\theta j}K^\theta \le C(q,d)2^{-\theta j_0}K^\theta \le C(q,d) \biggl(\frac{G_h(f)^{1/2}}{\|f\|_{2,h}}\biggr)^\theta,
\end{align}
and similarly
\begin{align}
\sum_{j\le j_0} 2^{(1-\theta)j}K^\theta h &\le C(q,d)2^{(1-\theta)j_0}h^{1-\theta}\nonumber \\
&\le C(q,d) \biggl(\frac{\|f\|_{2,h}}{G_h(f)^{1/2}}\biggr)^{1-\theta}.\label{gnstep5}
\end{align}
Combining \eqref{gnstep1}, \eqref{gnstep2}, \eqref{gnstep3}, \eqref{gnstep4} and \eqref{gnstep5} we get
\begin{align*}
\|f\|_{q,h} &\le C(q,d)\sum_{j=0}^\infty (2^{-j} K)^\theta \|P_{2^{-j}K} f\|_{2,h} \\
&\le C(q,d)\sum_{j = j_0+1}^\infty(2^{- j}K)^\theta \|f\|_{2,h} \\
&\qquad + C(q,d)\sum_{j=-\infty}^{j_0}(2^{- j}K)^\theta 2^{j}hG_h(f)^{1/2}\\
&\le C(q,d)\|f\|_{2,h}^{1-\theta} G_h(f)^{\theta/2}. 
\end{align*}
This completes the proof of the discrete Gagliardo-Nirenberg inequality.
\end{proof}
\subsection{Green's function}
Let $\Delta$ be the Laplacian operator on $\cc^{\zz^d}$ with grid size $h$, i.e.
\[
\Delta f(x) = \frac{1}{h^2} \sum_{y\; : \; y\sim x} (f(y)-f(x)). 
\]
\begin{lmm}\label{greeninv}
Let $I$ be the identity operator on $\cc^{\zz^d}$ and $\omega$ be any positive real number. For any $u\in L^2(\zz^d)$, the unique solution to the equation 
\[
(\omega I - \Delta) f = u
\]
is given by $f=g*u$, where $g$ is the discrete Green's function
\begin{equation}\label{greendef}
g(x) = \frac{h^{2-d}}{2d}\sum_{k=0}^\infty r^{k+1}p(x, k),
\end{equation}
where $r= 2d/(2d+\omega h^2)$ and $p(x,k)$ is the probability that a $d$-dimensional simple symmetric random walk started at the origin is at $x$ at time $k$. 
\end{lmm}
\begin{proof}
Note that $\Delta$ is negative semidefinite, since 
\[
(f, \Delta f) = -\frac{1}{2h^2}\sum_{x,y\in \zz^d\atop x\sim y} |f_x-f_y|^2. 
\]
Thus for any positive $\omega$, $\omega I-\Delta$ is a positive definite operator. In particular, given a function $u\in L^2(\zz^d)$, there can be at most one solution of 
\[
(\omega I -\Delta) f = u. 
\]
To show that $g*u$ is a solution, one proceeds exactly as in the proof of Lemma \ref{invlap}. 
\end{proof}
The Green's function is an indispensable tool in classical harmonic analysis. While the continuum Green's functions are relatively simple objects, the discrete ones are more complicated. The purpose of this Subsection is to derive some careful estimates for the discrete Green's functions. 
\begin{lmm}\label{pxk}
Let $p(x,k)$ be as in Lemma \ref{greeninv}. Then for all $x\in \zz^d$ and $k\ge 0$,
\[
p(x,k) \le C(d)e^{-|x|^2/2k}(1+k)^{-d/2}, 
\]
where we interpret $|x|^2/2k$ as $\infty$ if $x\ne 0$ and $k=0$, and as $0$ if $x=0$ and~$k=0$.
\end{lmm}
\begin{proof}
Suppose we are given $x = (x_1,\ldots, x_d)\in \zz^d$, $k\ge 0$, and $k_1,\ldots, k_d$ summing to $k$. If we know that for each $i=1,\ldots, d$, the walk has taken a total of $k_i$ steps along the $i$th coordinate axis, then the number of  ways that the walk can be at $x$ at step $k$ is exactly
\[
\prod_{i=1}^d {k_i \choose (k_i+x_i)/2},
\]
where the combinatorial term is understood to be zero if $k_i+x_i$ is odd, or if $x_i\not \in [-k_i, k_i]$. Therefore, 
\begin{align}
p(x,k) &= (2d)^{-k}\sum_{0\le k_1,\ldots, k_d\le k\atop k_1+\cdots +k_d=k} \frac{k!}{k_1!\cdots k_d!} \prod_{i=1}^d {k_i \choose (k_i+x_i)/2} \nonumber\\
&= \sum_{0\le k_1,\ldots, k_d\le k\atop k_1+\cdots +k_d=k} \frac{k!d^{-k}}{k_1!\cdots k_d!} \prod_{i=1}^d {k_i \choose (k_i+x_i)/2} 2^{-k_i}.\label{pxkform}
\end{align}
Suppose $k$ balls are dropped independently and uniformly at random into $d$ boxes. Let $K_i$ be the number of balls falling into the $i$th box. Then for any $0\le k_1,\ldots, k_d\le k$ such that $k_1+\cdots + k_d= k$,
\[
\pp(K_1=k_1,\ldots, K_d = k_d)= \frac{k!d^{-k}}{k_1!\cdots k_d!}. 
\]
Therefore,
\begin{equation}\label{pxkexp}
p(x,k) = \ee \biggl(\prod_{i=1}^d {K_i \choose (K_i+x_i)/2} 2^{-K_i}\biggr).
\end{equation}
A simple computation using Stirling's formula (e.g.\ the matching upper and lower bounds in \cite{robbins55}) shows that there is a universal constant $C$ such that for any integers $a \ge 1$ and $b\in [-a,a]$ such that $a+b$ is even,
\begin{align*}
\log \biggl({a \choose (a+b)/2}2^{-a}\biggr) &\le C - \frac{1}{2}\log a - \frac{a+b+1}{2}\log \biggl(1+\frac{b}{a}\biggr)\\
&\qquad - \frac{a-b+1}{2}\log \biggl(1-\frac{b}{a}\biggr).
\end{align*}
It is easy to verify that for any $x\in (-1,\infty)$, $\log(1+x)\ge x-\frac{x^2}{2}$. Applying this to control the logarithms on the right-hand side of the above expression, one gets
\[
{a \choose (a+b)/2}2^{-a} \le C\frac{e^{-b^2/2a}}{\sqrt{a}}. 
\]
This holds for $a\ge 1$. To include $a=0$, one modifies the inequality slightly to get 
\begin{equation}\label{stirlingbd}
{a \choose (a+b)/2}2^{-a} \le C\frac{e^{-b^2/2a}}{\sqrt{1+a}}. 
\end{equation}
Using this bound in \eqref{pxkexp} and applying H\"older's inequality, we get
\begin{align}
p(x,k) &\le C(d)\ee\biggl(e^{-\sum_{i=1}^d x_i^2/2K_i}\prod_{i=1}^d (1+K_i)^{-1/2}\biggr)\nonumber \\
&\le C(d) e^{-|x|^2/2k}  \ee\biggl(\prod_{i=1}^d (1+K_i)^{-1/2}\biggr)\nonumber \\
&\le C(d) e^{-|x|^2/2k} \biggl(\prod_{i=1}^d \ee(1+K_i)^{-d/2}\biggr)^{1/d}.\label{pxkbd}
\end{align}
(The interpretation for $k=0$ is as in the statement of the Lemma.) 

Now, each $K_i$ has a Binomial distribution with parameters $k$ and $1/d$. Therefore by Hoeffding's inequality \cite{hoeffding63}, 
\[
\pp(K_i\le k/2d) \le e^{-C(d)k},
\]
which clearly shows that for any $r > 0$,
\begin{equation}\label{kir}
\begin{split}
\ee(1+K_i)^{-r} &\le \pp(K_i \le k/2d) + (1+k/2d)^{-r}\\
&\le C(d,r)(1+k)^{-r}. 
\end{split}
\end{equation}
Plugging this into \eqref{pxkbd} proves the lemma.
\end{proof}
\begin{lmm}\label{intlmm}
Suppose that $f:[0,\infty)\ra[0,\infty)$ is a non-increasing function and $g:[0,\infty) \ra[0,\infty)$ is a non-decreasing function. Then
\[
\sum_{k=0}^\infty f(k+1)g(k) \le \int_0^\infty f(t) g(t) dt. 
\]
More generally,
\[
\sum_{x\in \zz^d}^\infty f(|x|+\sqrt{d})g(|x|) \le \int_{\rr^d} f(|y|) g(|y|) dy. 
\]
\end{lmm}
\begin{proof}
For the first part, simply observe that when $t\in [k, k+1]$, $f(t) \ge f(k+1)$ and $g(t) \ge g(k)$, so that $f(k+1)g(k)\le \int_k^{k+1} f(t)g(t) dt$.

For the second part, let $\zz_+ = \{0,1,\ldots\}$ and $\rr_+ = [0,\infty)$. Then 
\[
\sum_{x\in \zz^d}^\infty f(|x|+\sqrt{d})g(|x|) \le 2^d \sum_{x\in \zz_+^d} f(|x|+\sqrt{d})g(|x|)
\]
and 
\[
\int_{\rr^d} f(|y|) g(|y|) dy = 2^d \int_{\rr_+^d} f(|y|) g(|y|) dy,
\]
so that it suffices to prove 
\[
\sum_{x\in \zz_+^d}^\infty f(|x|+\sqrt{d})g(|x|) \le \int_{\rr_+^d} f(|y|) g(|y|) dy. 
\]
Take any $x\in \zz_+^d$. Let $B(x)$ be the set 
\[
\{y\in \rr^d : x_i\le y_i< x_i+1, \; i=1,\ldots,d\}.
\]
The sets $B(x)$ are pairwise disjoint and all have volume $1$. Moreover, on the set $B(x)$, $|x|\le |y|\le |x|+\sqrt{d}$, and therefore 
\[
f(|x|+\sqrt{d}) g(|x|)\le  \int_{B(x)} f(|y|) g(|y|) dy. 
\]
This completes the proof of the lemma.   
\end{proof}
\begin{lmm}\label{greenbds}
Given $\omega > 0$, let $r= 2d/(2d+\omega h^2)$, as in Lemma \ref{greeninv}. Take any $\beta > 0$, $\alpha \in \rr$ and $x\in \zz^d$. Let 
\[
S := \sum_{k=0}^\infty r^{k+1} e^{-\beta |x|^2/k}(1+k)^{-\alpha/2}. 
\]
The following bounds hold:
\begin{enumerate}
\item[\textup{(a)}]  If $\alpha > 2$, then 
\[
S \le C(d,\alpha, \beta) \bigl(\sqrt{d}+|x|\bigr)^{2-\alpha} e^{-C(d,\beta, \omega)h|x|}.
\]
\item[\textup{(b)}] If $\alpha = 2$ then 
\[
S \le C(d,\alpha, \beta)  \bigl(C(\beta) + C(\beta)\bigl|\log (h(\sqrt{d}+|x|))\bigr|\bigr)e^{-C(d,\beta, \omega) h|x|}.
\]
\item[\textup{(c)}] If $\alpha < 2$, then 
\[
S \le C(d,\alpha, \beta) h^{\alpha -2}e^{-C(d,\beta, \omega) h|x|}.
\]
\end{enumerate}
In all three cases, $C(d,\beta, \omega)$ is an increasing function of $\omega$.
\end{lmm}
\begin{proof}
First, assume that $x\ne 0$. For any $y \ge 0$, 
\[
\log(1+y) = \int_0^y \frac{1}{1+z} dz \ge \frac{y}{1+y}. 
\]
Consequently, for any $K>0$ and $x\in [0,1]$, 
\begin{align*}
\frac{1}{1+Kx} &= e^{-\log (1+Kx)} \le e^{-Kx/(1+Kx)} \le e^{-Kx/(1+K)}.
\end{align*}
Since $h\in (0,1)$, the above inequality shows that 
\begin{equation}\label{rcd0}
r = \frac{2d}{2d+\omega h^2}= \frac{1}{1+\frac{\omega}{2d} h^2}\le e^{-C(d,\omega) h^2},
\end{equation}
where $C(d,\omega)$ is  an increasing function of $\omega$. 

By Lemma \ref{pxk} , Lemma \ref{intlmm} and the inequality \eqref{rcd0},
\begin{align*}
S&\le \int_0^\infty r^t e^{-\beta |x|^2/t} t^{-\alpha/2} dt\\
&\le \int_0^\infty e^{-C(d, \omega)h^2 t-\beta |x|^2/t} t^{-\alpha/2} dt.
\end{align*}
Applying the change-of-variable $u=|x|^2/t$ in the above integration, we get 
\begin{align*}
S \le |x|^{2-\alpha}\int_0^{\infty} \exp\biggl(-\frac{A}{u} - \beta u\biggr)u^{(\alpha-4)/2} du,  
\end{align*}
where $A = C(d,\omega) h^2 |x|^2$. The inequality $a^2+b^2 \ge 2ab$ shows that for all $u> 0$, 
\[
\frac{A}{2u} + \frac{\beta u}{2} \ge \frac{\sqrt{A\beta}}{2} =: B. 
\]
Consequently, 
\begin{align*}
S &\le |x|^{2-\alpha} e^{-B}\int_0^\infty \exp\biggl(-\frac{A}{2u} - \frac{\beta u}{2}\biggr) u^{(\alpha-4)/2} du. 
\end{align*}
When $\alpha > 2$, the integrand may be bounded by a constant that depends only on $\beta$, $\alpha$ and $d$ (and not $x$), by simply dropping the $A/2u$ term.

Next, consider the case $\alpha = 2$. Then
\begin{align*}
\int_{\sqrt{A}}^\infty \exp\biggl(-\frac{A}{2u} - \frac{\beta u}{2}\biggr) u^{-1} du&\le \int_{\sqrt{A}}^\infty e^{-\beta u/2} u^{-1} du\\
&\le  C(\beta) + C(\beta)|\log A|,
\end{align*}
and by the change of variable $z=A/u$, 
\begin{align*}
\int_0^{\sqrt{A}} \exp\biggl(-\frac{A}{2u} - \frac{\beta u}{2}\biggr) u^{-1} du&\le \int_{\sqrt{A}}^\infty \frac{e^{-z/2}}{z} dz \\
&\le C + C|\log A|. 
\end{align*}
Finally, consider the case $\alpha <2$. Again by the change of variable $z=A/u$, 
\begin{align*}
\int_0^\infty \exp\biggl(-\frac{A}{2u} - \frac{\beta u}{2}\biggr) u^{(\alpha-4)/2} du&\le A^{(\alpha-2)/2}\int_0^\infty e^{-z/2} z^{-\alpha/2} dz \\
&\le C(\alpha)A^{(\alpha-2)/2}. 
\end{align*}
This completes the proofs of all three cases when $x\ne 0$. When $x=0$, the steps are essentially the same, but simpler. 
\end{proof}
\begin{prop}\label{gbound}
Suppose that $d\ge 3$. Given $\omega>0$, define $g$ as in \eqref{greendef}. Then for all $x\in \zz^d$,
\[
g(x) \le C(d) h^{2-d} (\sqrt{d}+|x|)^{2-d} e^{-C(d,\omega)h|x|}, 
\]
where $C(d,\omega)$ is an increasing function of $\omega$ (when $d$ is fixed). 
\end{prop}
\begin{proof}
First, suppose that $x\ne 0$. Then by Lemma \ref{greeninv} and Lemma \ref{pxk}, 
\begin{align*}
g(x) &\le C(d)h^{2-d} \sum_{k=0}^\infty r^{k+1}e^{-|x|^2/2k} (1+k)^{-d/2}.
\end{align*}
Lemma \ref{greenbds} now completes the proof.
\end{proof}
\begin{cor}\label{gcor0}
Suppose that $d\ge 3$ and $g$ is the Green's function defined in~\eqref{greendef}. Then for any $t\in [1, d/(d-2))$, $\|g\|_{t,h}\le C(t,d, \omega)$, where $C(t,d,\omega)$ is a decreasing function of $\omega$.
\end{cor}
\begin{proof}
Let $A = 2\sqrt{d}$. By Proposition \ref{gbound},
\begin{align*}
\|g\|_{t,h}^t &\le h^d\sum_{x\in \zz^d \atop |x|\le A/h}  C(t,d)\bigl( h^{2-d} (\sqrt{d}+|x|)^{2-d}\bigr)^t \\
&\qquad + \sum_{x\in \zz^d \atop |x|> A/h}C(t,d) e^{-C_0(d,\omega) h |x|t}, 
\end{align*}
where $C_0(d,\omega)$ is an increasing function of $\omega$. By Lemma \ref{intlmm} and the assumption that $h < 1$,
\begin{align*}
h^d\sum_{x\in \zz^d \atop |x|> A/h}e^{-C_0(d,\omega) h |x|t} &= e^{C_0(d,\omega)t\sqrt{d}} h^d \sum_{x\in \zz^d} e^{-C_0(d,\omega) h(|x|+\sqrt{d}) t}1_{\{|x|> A/h\}}\\
&\le e^{C_0(d,\omega)t\sqrt{d}} h^d \int_{\rr^d} e^{-C_0(d,\omega) h|y| t}1_{\{|y|> A/h\}} dy\\
&=  e^{C_0(d,\omega)t\sqrt{d}} \int_{\rr^d} e^{-C_0(d,\omega) |z|t}1_{\{|z|> A\}} dz \\
&= e^{C_0(d,\omega)t\sqrt{d}} \int_A^\infty u^{d-1}e^{-C_0(d,\omega) ut}du.
\end{align*}
Since $A=2\sqrt{d}$, it is easy to see that  the last expression above can be bounded by a constant $C(t, d,\omega)$ that is a decreasing function of $\omega$.

Again by Lemma \ref{intlmm},
\begin{align*}
&h^d\sum_{x\in \zz^d \atop |x|\le A/h}  C(t,d)\bigl( h^{2-d} (\sqrt{d}+|x|)^{2-d}\bigr)^t\\
&\le C(t,d)h^d\int_{\rr^d} (h|y|)^{(2-d)t} 1_{\{|y|\le A/h + \sqrt{d}\}}dy\\
&= C(t,d) \int_{\rr^d} |z|^{(2-d)t} 1_{\{|z|\le A + h\sqrt{d}\}}dz\\
&= C(t,d) \int_{0}^{A+h\sqrt{d}} u^{d-1+(2-d)t}du. 
\end{align*}
Since $h < 1$ and $d-1+(2-d) t\ge -1 + \ep$ for some positive $\ep = \ep(t,d)$, this shows that the integrand is bounded by $C(t,d)$ and completes the proof of the lemma. 
\end{proof}

\subsection{Hardy-Littlewood-Sobolev inequality}
Take any $\omega > 0$ and let $g$ be the Green's function defined in \eqref{greendef}. Given a function $u\in L^2(\zz^d)$, let $f$ be the unique solution to 
\[
(\omega I - \Delta) f = u,
\]
so that by Lemma \ref{greeninv}, $f= g*u$. The following proposition is a discrete analog of the Hardy-Littlewood-Sobolev theorem of fractional integration (see \cite[Chapter V, Section 1.2]{stein70}). 
\begin{prop}\label{hls}
Let $f$ and $u$ be as above and suppose that $d\ge 3$. Let $1< p< q< \infty$ satisfy
\[
\frac{1}{q} > \frac{1}{p}-\frac{2}{d}. 
\]
Then $\|f\|_{q,h} \le C(p,q,d,\omega)\|u\|_{p,h}$, where $C(p,q,d,\omega)$ is a decreasing function of $\omega$. 
\end{prop}
This theorem is actually quite a bit simpler than the classical  Hardy-Littlewood-Sobolev theorem, which includes the endpoint case $1/q=1/p-2/d$. Including the endpoint requires a somewhat delicate argument using the Marcinkiewicz interpolation theorem, which we can afford to avoid. 

\begin{proof}
Let $s=(0,2)$ satisfy
\[
\frac{1}{q}=\frac{1}{p} - \frac{s}{d}. 
\]
Let $t=d/(d-s)$. Then by Young's inequality,
\begin{align*}
\|f\|_{q,h} &\le \|g\|_{t,h} \|u\|_{p,h}. 
\end{align*}
Since $t < d/(d-2)$, therefore by Corollary \ref{gcor0}, 
$\|g\|_{t,h}\le C(p,q,d,\omega)$, 
where $C(p,q,d,\omega)$ is a decreasing function of $\omega$. This completes the proof. 
\end{proof}

\subsection{Derivatives of the Green's function}
Recall the discrete derivative operator $\nabla_i$ defined in \eqref{discderiv}. The following proposition gives an estimate on the size of $\nabla_i g$, where $g$ is the discrete Green's function defined in \eqref{greendef}.
\begin{prop}\label{gdbound}
Given $\omega>0$, define $g$ as in \eqref{greendef}. If $d\ge 2$, then for all $x\in \zz^d$ and all $1\le i\le d$,
\[
|\nabla_i g(x)| \le C(d) h^{1-d} (\sqrt{d}+|x|)^{1-d} e^{-C(d,\omega)h|x|}, 
\]
where $C(d,\omega)$ is an increasing function of $\omega$.  When $d=1$, 
\[
|\nabla_1 g(x)|\le C \bigl(1+\bigl|\log( h(1+|x|))\bigr|\bigr)e^{-C(\omega) h|x|},
\]
where $C(\omega)$ is an increasing function of $\omega$. 
\end{prop}
To prove this proposition, we first need to introduce some notation. For each non-negative integer $k$, let 
\[
P_k := \{(k_1,\ldots, k_d) \in \zz^d: 0\le k_1,\ldots, k_d \le k, \ k_1+\cdots + k_d = k\}. 
\]
For any $k\ge 0$ and any $(k_1,\ldots, k_d)\in P_k$ and $(x_1,\ldots, x_d)\in \zz^d$, let
\[
\psi(k; k_1,\ldots, k_d; x_1, \ldots, x_d) := \frac{k!d^{-k}}{k_1!\cdots k_d!} \prod_{j=1}^d {k_j \choose (k_j+x_j)/2} 2^{-k_j}. 
\]
Here, as usual, we interpret ${a\choose b}$ as $0$ if either $a$ or $b$ is not a non-negative integer, or if $b > a$. In other words, for the above expression to be non-zero, it is necessary and sufficient that for all $j$, $x_j$ has the same parity as $k_j$ and satisfies $|x_j|\le k_j$. 
\begin{lmm}\label{psipsi}
For any $k\ge 1$, any $1\le i\le d$, any $(k_1,\ldots, k_d)\in P_{k-1}$ and any $(x_1,\ldots, x_d)\in \zz^d$, we have 
\begin{align*}
&\psi(k-1; k_1,\ldots, k_d; x_1,\ldots, x_d) \\
&= \psi(k; k_1,\ldots, k_i+1, \ldots, k_d; x_1,\ldots, x_i+1, \ldots, x_d) \frac{(k_i+x_i+2)d}{k}. 
\end{align*}
\end{lmm}
\begin{proof}
If $k_j+x_j$ is odd for some $j$, then both sides are zero. So assume that $k_j$ has the same parity as $x_j$ for each $j$. Similarly, if $|x_j|> k_j$ for some $j\ne i$, then both sides are zero. So assume that $|x_j|\le k_j$ for all $j\ne i$. 

If $|x_i|> k_i$, then the left side is zero, and there are three possibilities for the right side. First,  $x_i$ may be equal to $-k_i-2$. In that case, $k_i+x_i+2=0$ and so the right side is zero. The other possibilities are  that $x_i < -k_i-2$ or $x_i \ge k_i+1$. In both of these cases, $|x_i+1| > k_i+1$, and hence the right side is zero. Note that the case $x_i=-k_i-1$ is excluded because $x_i$ has the same parity as $k_i$. Therefore in all three cases we have equality of the two sides. So we may now safely assume that $|x_i|\le k_i$. 

At this point, we have that both sides are non-zero. Verifying the identity is now a simple algebraic exercise. 
\end{proof}
\begin{proof}[Proof of Proposition \ref{gdbound}]
Take any $x\in \zz^d$ such that $\sum_{i=1}^d x_i$ is even. Then $p(x,k) = 0$ for all odd $k$ and $p(x+e_i, k)$ is zero for all even $k$, where $e_i$ denotes the $i$th coordinate vector. Thus,  from the expression \eqref{greendef}, 
\begin{align}
\frac{g(x+e_i) - g(x)}{h} &= \frac{h^{1-d}}{2d}\biggl(\sum_{k \text{ even}} r^{k+1} p(x,k)- \sum_{k \text{ odd}} r^{k+1}p(x+e_i, k)\biggr)\nonumber \\
&= \frac{h^{1-d}}{2d}\sum_{k \text{ even}}r^{k+1}(1-r) p(x,k) \label{gderiveexp}\\
&\qquad + \frac{h^{1-d}}{2d}\sum_{k\text{ odd}} r^{k+1}(p(x, k-1)-p(x+e_i, k)). \nonumber
\end{align}
Now fix some odd $k$ and some $1\le i\le d$. 
Then by the formula \eqref{pxkform}, 
\begin{equation}\label{pxk1}
p(x,k-1) = \sum_{(k_1,\ldots, k_d)\in P_{k-1}} \psi(k-1; k_1,\ldots, k_d; x_1,\ldots x_d) 
\end{equation}
and 
\begin{equation}\label{pxk2}
p(x+e_i,k) = \sum_{(k_1,\ldots, k_d)\in P_k} \psi(k; k_1,\ldots, k_d; x_1, \ldots,x_i+1, \ldots, x_d).  
\end{equation}
Let $P_k'$ be the set of all $(k_1,\ldots, k_d) \in P_k$ with $k_i\ne 0$. It is easy to see that the map
\[
(k_1,\ldots, k_d) \mapsto (k_1,\ldots, k_{i-1}, k_i+1,k_{i+1},\ldots, k_d)
\]
is a bijection between $P_{k-1}$ and $P_k'$. 
Thus, by Lemma \ref{psipsi},
\begin{align*}
&p(x, k-1) =\sum_{(k_1,\ldots, k_d)\in P_{k-1}} \psi(k-1; k_1, \ldots, k_d; x_1, \ldots, x_d) \\
&= \hskip-.2in \sum_{(k_1,\ldots, k_d)\in P_{k-1}}\hskip-.2in\psi(k; k_1,\ldots, k_i+1, \ldots, k_d; x_1,\ldots, x_i+1, \ldots, x_d)\frac{(k_i+x_i+2)d}{k}\\
&= \hskip-.2in\sum_{(k_1,\ldots, k_d)\in P_k'}\hskip-.2in\psi(k; k_1,\ldots, k_d; x_1,\ldots,x_i+1, \ldots, x_d)\frac{(k_i+x_i+1)d}{k}. 
\end{align*}
By \eqref{pxk1} and \eqref{pxk2}, this shows that 
\begin{align*}
&|p(x, k-1)- p(x+e_i, k)| \\
&\le \hskip-.2in\sum_{(k_1,\ldots, k_d)\in P_k} \hskip-.2in\psi(k; k_1,\ldots, k_d; x_1,\ldots,x_i+1,\ldots, x_d)\biggl|\frac{(k_i+x_i+1)d}{k}1_{\{k_i\ne 0\}} - 1\biggr|.
\end{align*}
Let $K_1,\ldots, K_d$ be the random variables defined in the proof of Lemma \ref{pxk}. Then by the above inequality and \eqref{stirlingbd}, we get
\begin{align*}
&|p(x,k-1)-p(x+e_i, k)|\\
&\le \ee\biggl(2^{-k}\biggl|\frac{(K_i+x_i+1)d}{k}1_{\{K_i\ne 0\}} - 1\biggr|\prod_{j=1}^d {K_j \choose (K_j+x_j + 1_{\{j=i\}})/2}\biggr)\\
&\le C(d)e^{-|x|^2/2k} \ee\biggl(\biggl|\frac{(K_i+x_i+1)d}{k}1_{\{K_i\ne 0\}} - 1\biggr|\prod_{j=1}^d (1+K_j)^{-1/2}\biggr). 
\end{align*}
Now observe the following:
\begin{itemize}
\item By \eqref{kir}, for all $r > 0$, $\ee(1+K_i)^{-r} \le C(d,r)(1+k)^{-r}$. By H\"older's inequality, this gives 
\begin{align*}
&\ee\biggl(\biggl|\frac{(K_i+x_i+1)d}{k}1_{\{K_i\ne 0\}} - 1\biggr|\prod_{j=1}^d (1+K_j)^{-1/2}\biggr)\\
&\le (1+k)^{-d/2}\biggl(\ee\biggl|\frac{(K_i+x_i+1)d}{k}1_{\{K_i\ne 0\}} - 1\biggr|^{d+1}\biggr)^{1/(d+1)}.
\end{align*}
\item Since $|x_i|/k \le (|x|/\sqrt{k}) k^{-1/2}$,
\begin{align*}
&\biggl(\ee\biggl|\frac{(K_i+x_i+1)d}{k}1_{\{K_i\ne 0\}} - 1\biggr|^{d+1}\biggr)^{1/(d+1)}\\
&\le \frac{|x|}{\sqrt{k}} k^{-1/2} + \frac{d}{k} + \biggl(\ee\biggl|\frac{K_id}{k}1_{\{K_i\ne 0\}} - 1\biggr|^{d+1}\biggr)^{1/(d+1)}.
\end{align*}
\item Since $K_i$ is a Binomial random variable with parameters $k$ and $1/d$, it follows by Hoeffding's tail bound \cite{hoeffding63} that $\pp(K_i=0)\le e^{-C(d)k}$. Moreover, $K_i\le k$. Thus,
\begin{align*}
&\biggl(\ee\biggl|\frac{K_id}{k}1_{\{K_i\ne 0\}} - 1\biggr|^{d+1}\biggr)^{1/(d+1)}\\
&\le e^{-C(d)k} + \biggl(\ee\biggl|\frac{K_id}{k} - 1\biggr|^{d+1}\biggr)^{1/(d+1)}.
\end{align*}
\item Again by Hoeffding's bound, for any $r > 0$,
\[
\ee\biggl|\frac{K_id}{k} - 1\biggr|^r \le C(d,r) k^{-r/2}. 
\]
\end{itemize}
Combining all of the above, we see that for any odd $k$, 
\begin{equation}\label{ppxk}
|p(x, k-1) - p(x+e_i, k)|\le C_1(d) k^{-(d+1)/2}e^{-C_2(d)|x|^2/k}
\end{equation}
Using this estimate and Lemma \ref{pxk} to bound the right-hand side in \eqref{gderiveexp}, and applying Lemma \ref{greenbds}, we get that for $d\ge 2$, 
\begin{align*}
|\nabla_i g(x)| &\le C(d) h^{3-d} (\sqrt{d}+|x|)^{2-d} e^{-C(d,\omega)h|x|} \\
&\qquad + C(d)h^{1-d}(\sqrt{d}+|x|)^{1-d} e^{-C(d,\omega) h|x|},
\end{align*}
where $C(d,\omega)$ is an increasing function of $\omega$.   This proves Proposition \ref{gdbound} when $d\ge 2$.

When $d=1$, we use \eqref{ppxk}, \eqref{gderiveexp},~\eqref{greendef} and Lemma \ref{pxk} to get
\begin{align*}
|\nabla_i g(x)| &\le Ch^{2} \sum_{k=0}^\infty r^{k+1} e^{-|x|^2/2k} (1+k)^{-1/2} \\
&\qquad + C\sum_{k=0}^\infty r^{k+1} e^{-|x|^2/2k} (1+k)^{-1}. 
\end{align*}
Lemma \ref{greenbds} now completes the proof.
\end{proof}
\begin{cor}\label{greenderiv}
For any $d$ and any $1\le i\le d$, 
\[
\|\nabla_i g\|_{1,h}\le C(d, \omega),
\] 
where $C(d,\omega)$ is a decreasing function of $\omega$. 
\end{cor}
\begin{proof}
Follows easily from Proposition \ref{greenderiv} and Lemma \ref{intlmm}. 
\end{proof}

\section{Regularity of discrete solitons}\label{solitonreg}
Suppose that $1< p< 1+4/d$. As in Section \ref{harmonic}, assume that $h\in (0,1)$. Take any $m > 0$ and let $f\in L^2(\zz^d)$ be a ground state soliton of mass $m$ for the discrete NLS on $\zz^d$ under grid size~$h$. That is, $f$ minimizes energy among all functions of mass $m$. By Theorem \ref{conccomp} it is easy to see that at least one such function exists. Also, by elementary Euler-Lagrange techniques, it follows that there is an $\omega > 0$ such that $f$ satisfies the soliton equation $(\omega I - \Delta)f = |f|^{p-1} f$, 
where $I$ is the identity operator and $\Delta$ is the discrete Laplacian on $\zz^d$ at grid size~$h$ as defined in \eqref{discretelap2}. 

The purpose of this section is to prove regularity properties of discrete solitons. The main goal will be to prove that the smoothness bounds remain uniformly bounded as the grid size goes to zero. This is necessary for proving convergence to continuum solitons. The proof follows more or less the sketch of the proof of regularity for continuum solitons (see \cite[Proposition B.7]{tao06}), but using the discrete estimates from Section \ref{harmonic}. 

In this section, $C$ will denote any positive constant that depends only on $p$, $d$ and $m$. In particular, $C$ will not depend on $h$. Moreover, we impose the additional condition that $C$ is uniformly bounded as $m$ ranges over any given compact subinterval of $(0,\infty)$. We will call this the {\it uniform boundedness condition}.  If $C$ depends on additional parameters $a,b,\ldots$, then it is denoted as $C(a,b,\ldots)$. 
\begin{lmm}\label{omega2}
$\omega \ge 1/C$.
\end{lmm}
\begin{proof}
Following exactly the same steps as in the proof of Lemma \ref{omega}, we arrive at the inequality
\[
\omega \ge -\frac{E_{\min}(m,h)}{m}. 
\]
The uniform boundedness condition holds due to Lemma \ref{emin0}. 
\end{proof}
\begin{lmm}\label{regG}
$G_h(f) \le C$. 
\end{lmm}
\begin{proof}
Taking  $q = p+1$  and 
\[
\theta = \frac{d(p-1)}{2(p+1)},
\]
the discrete Gagliardo-Nirenberg inequality (Proposition \ref{gn}) implies that
\begin{align*}
\|f\|_{p+1,h}^{p+1} &\le C \|f\|_{2,h}^{(1-\theta)(p+1)} G_h(f)^{\theta(p+1)/2}\\
&\le C G_h(f)^{d(p-1)/4}.   
\end{align*}
(Note that, since $h > 0$, $f\in L^2$ implies that $f\in L^q$. Also, it is easy to verify using the condition $1<p < 1+4/d$ that $\theta \in (0,1)$. Lastly, note that the uniform boundedness condition on $C$ is clearly satisfied.) Now, $H_h(f)=E_{\min}(m,h) < 0$ by Lemma \ref{emin0}. 
Consequently, by the previous display,
\begin{align*}
G_h(f) &= H_h(f) + N_h(f)\\
&\le C \|f\|_{p+1,h}^{p+1} \le C G_h(f)^{d(p-1)/4}. 
\end{align*}
Since $1<p< 1+4/d$, or equivalently, $d(p-1)/4 \in (0,1)$, this shows that $G_h(f)\le C$. 
\end{proof}
\begin{lmm}\label{fqbd}
For all $q\in [2,\infty]$, $\|f\|_{q,h} \le C(q)$.
\end{lmm}
\begin{proof}
The case $q=2$ is already known from the assumption that $M_h(f)=m$. First, assume that $q\in (2,\infty)$. If $d\le 2$, then for any such $q$, we can find $\theta\in (0,1)$ satisfying \eqref{qtheta}, and therefore by the discrete Gagliardo-Nirenberg inequality and Lemma~\ref{regG}, it follows that $\|f\|_{q,h}\le C(q)$ for all $2 < q< \infty$. 

Next, suppose that $d\ge 3$. By the discrete Gagliardo-Nirenberg inequality (Proposition \ref{gn}) and Lemma~\ref{regG}, we have that $\|f\|_{q,h} \le C(q)$ for all $2\le q<2d/(d-2)$. 

Define a sequence $q_k$ as follows. Let $q_0 := 2d/(d-2)$. Note that the condition $p < 1+4/d$ implies that $q_0 > p + 1$.
For each $k$, let $q_{k+1}$ satisfy
\[
\frac{1}{q_{k+1}} = \frac{p}{q_k} - \frac{2}{d},
\]
unless the right hand side is nonpositive, in which let $q_{k+1}=\infty$. If $q_k=\infty$ for some $k$, then this definition implies that $q_j=\infty$ for all $j\ge k$.

Note that if $q_k$ is finite for all $k$, then for each $k$ we must have 
\begin{align*}
\frac{1}{q_k} &= \frac{p^k}{q_0} - \frac{2}{d}(1+p+\cdots + p^{k-1})\\ 
&= \frac{p^k}{2d}\biggl(d-2 - \frac{4}{p}\sum_{i=0}^{k-1} p^{-i}\biggr). 
\end{align*}
But as $k \ra\infty$, 
\begin{align*}
\frac{4}{p}\sum_{i=0}^{k-1} p^{-i} \ra \frac{4}{p-1}> d-2. 
\end{align*}
This shows that $q_k$ must become infinity at some finite $k$. 

Next, note that $q_k$ is an increasing sequence. This is easily proved by induction as follows. Suppose that $q_k \ge q_0=2d/(d-2)$. If $q_{k+1}=\infty$, there is nothing to prove. Otherwise, by the condition $p < 1+4/d$, 
\begin{align*}
\frac{1}{q_k}-\frac{1}{q_{k+1}} &= \frac{2}{d} - \frac{p-1}{q_k} \\
&\ge \frac{2}{d} -\frac{(p-1)(d-2)}{2d} > 0. 
\end{align*}
Take any $k$ such that $q_k < \infty$. Suppose we have proved that for all $q \in [2, q_k)$, 
\begin{equation}\label{qqk}
 \|f\|_{q,h}\le C(q).
\end{equation} 
Choose arbitrary $q\in [q_k, q_{k+1})$. Then $1<q_k/p < q < \infty$, and 
\[
\frac{1}{q} > \frac{1}{q_{k+1}}\ge \frac{p}{q_k} - \frac{2}{d}.
\] 
Take $q' \in [2,q_k)$ so close to $q_k$ that 
\[
\frac{1}{q} > \frac{p}{q'} - \frac{2}{d}. 
\]
Let $f_0 := |f|^{p-1} f$. Since $f = (\omega I - \Delta)^{-1} f_0$ and $\omega \ge 1/C$ by Lemma \ref{omega2}, 
it follows by the discrete Hardy-Littlewood-Sobolev inequality (Proposition~\ref{hls}) that 
\begin{align*}
\|f\|_{q} &\le C(q, q') \|f_0\|_{q'/p}\\
&= C(q,q') \|f\|_{q'}^p \le C(q,q'). 
\end{align*}
However, $q'$ can be chosen depending only on $q$, $p$ and $d$. Since the sequence $q_k$ increases to infinity, this proves by induction that $\|f\|_{q,h}\le C(q)$ for all $q\in (2,\infty)$.

(An important thing to note is that we crucially used the fact that the constant in Proposition \ref{hls} is a decreasing function of $\omega$, in conjunction with Lemma \ref{omega2}, to conclude that $C(q)$ may be chosen to depend on $\omega$ only through $m$ and not on the actual value of~$\omega$, and that moreover, $C(q)$ satisfies the uniform boundedness condition.)

Next, consider the case $q=\infty$. Take any $r > d/2$. Let $r' = r/(r-1)$, so that $r' < d/(d-2)$. Then by Corollary \ref{gcor0}, Young's inequality, Lemma \ref{omega2} and what we have already proved above, 
\begin{align*}
\|f\|_{\infty,h} &= \|g * f_0\|_{\infty,h} \le \|g\|_{r',h} \|f_0\|_{r,h}= \|g\|_{r',h}\|f\|_{pr,h}\le C. 
\end{align*}
This completes the proof. 
\end{proof}
\begin{lmm}\label{fqbd2}
For any $q\in [2, \infty]$ and any $1\le i, j\le d$, 
\[
\|\nabla_i f\|_{q,h}\le C(q) \ \ \text{ and } \ \ \|\nabla_i\nabla_j f\|_{q,h}\le C(q).
\]
\end{lmm}
\begin{proof}
Since the ground state soliton minimizes energy, and the function $|f|$ satisfies $M_h(f)= M_h(|f|)$ and $H_h(|f|)\le H_h(f)$ by the triangle inequality, therefore we must have $H_h(|f|) = H_h(f)$. Consequently, $||f(x)|-|f(y)|| = |f(x)-f(y)|$ for each neighboring pair of points $(x,y)$, which shows that $f$ must be of the form $f(x) = \alpha f_1(x)$ for some constant $\alpha\in \cc$ with $|\alpha|=1$ and some function $f_1:\zz^d \ra [0,\infty)$. Therefore, without loss of generality we will assume in this proof that $f$ is a non-negative function.

By Young's inequality, Lemma \ref{fqbd},  and the observation that $\nabla_i$ is a convolution operator, we have that for any $q\in [2,\infty]$, 
\begin{equation}\label{nabla1}
\begin{split}
\|\nabla_i f\|_{q,h} &= \|(\nabla_i g)*f^p\|_{q,h}\\
&\le \|\nabla_i g\|_{1,h} \|f^p\|_{q,h}\le C(q) \|\nabla_i g\|_{1,h}.
\end{split}
\end{equation}
Next, note that by the inequality $|a^p - b^p| \le \max\{pa^{p-1}, p b^{p-1}\} |a-b|$ for non-negative $a$ and $b$ (which follows by the mean value theorem), and fact that $\|f\|_{\infty, h}\le C$ from Lemma \ref{fqbd}, we have that for all $x$,
\[
|\nabla_i f^p(x)|\le C |\nabla_i f(x)|.
\]
In particular, by Lemma \ref{fqbd} this implies that for all $q\in [2,\infty]$, 
\begin{equation}\label{nablafp}
\|\nabla_i f^p\|_{q,h} \le C \|\nabla_i f\|_{q,h}\le C(q). 
\end{equation}
By the commutativity of convolution operators,
\begin{align*}
\nabla_i \nabla_j f &= (\nabla_i g)*( \nabla_j f^p). 
\end{align*}
Thus, by Young's inequality and \eqref{nablafp}, we see that for any $q\in [2,\infty]$, 
\begin{equation}\label{nabla2}
\begin{split}
\|\nabla_i \nabla_j f\|_{q,h} &\le \|\nabla_i g\|_{1,h} \|\nabla_j f^p\|_{q,h} \\
&\le C(q)\|\nabla_i g\|_{1,h}. 
\end{split}
\end{equation}
Applying Corollary \ref{greenderiv} (in conjunction with Lemma \ref{omega2}) to \eqref{nabla1} and \eqref{nabla2} completes the proof. The uniform boundedness condition follows from the monotonicity of the constant in Corollary \ref{greenderiv}. 
\end{proof}

\section{Continuum limit of discrete solitons}\label{solitonlim}
In this section we establish that discrete solitons converge to continuum solitons as the grid size goes to zero, provided that the nonlinearity is mass-subcritical.

Assume that $1< p< 1+4/d$. Fix $h > 0$, $m > 0$ and let $f$ be a ground state soliton of mass $m$ for the DNLS system \eqref{dnlseq2} on $\zz^d$ at grid size $h$. Let $Q_{\lambda(m)}$ be the unique ground state soliton of mass $m$ for the continuum system~\eqref{nlsequation}. Let $\tf$ be the continuum image of $f$ at grid size $h$, as defined in Section \ref{basic}. 
\begin{thm}\label{tfthm}
For any $q\in [2,\infty]$, 
\[
\tl^q(\tf, Q_{\lambda(m)}) \le C(p,d, m,q,h),
\]
where $C(p,d,m,q,h)$ satisfies, for any $0<m_0\le m_1< \infty$ and any fixed $p$, $d$ and $q$,
\begin{equation}\label{cunif}
\lim_{h\ra0}\sup_{m_0\le m\le m_1} C(p,d,m,q,h) = 0.
\end{equation}
The same bound also holds for $|H_h(f)-E_{\min}(m)|$ (with the modification that there is no $q$). 
\end{thm}
In the following, $C$ will denote any positive constant that depends only on $p$, $d$ and $m$, satisfying the uniform boundedness condition defined in Section \ref{solitonreg}: that is, $C$ remains uniformly bounded as $m$ varies over a compact subinterval of $(0,\infty)$, with $p$ and $d$ fixed. If $C$ depends on additional parameters $a,b,\ldots$, then it will be written as $C(a,b,\ldots)$. We will also adopt the convention that $o(1)$ denotes any constant that depends only on $p$, $d$, $m$ and $h$ and satisfies \eqref{cunif}. 

Let $w: \rr\ra\rr$ be the function
\[
w(t) := 
\begin{cases}
1-|t| &\text{ if } |t|\le 1,\\
0 &\text{ if } |t|> 1. 
\end{cases}
\]
Extend $w$ to $\rr^d$ as 
\[
w(t_1,\ldots, t_d) := w(t_1)w(t_2)\cdots w(t_d). 
\]
Define a function $f^c: \rr^d\ra\cc$ as 
\begin{align*}
f^c(y) &:= \sum_{x\in \zz^d} f(x)w\biggl(\frac{y-hx}{h}\biggr).
\end{align*}
\begin{lmm}\label{fcrep}
The function $f^c$ is absolutely continuous. If $x\in \zz^d$ and $y=hx+ht$ for some $t\in (0,1)^d$, then 
\[
f^c(y) = \sum_{s\in \{0,1\}^d} f(x + s) \prod_{i=1}^dt_i^{s_i} (1-t_i)^{1-s_i}. 
\]
If $\partial_i f^c$ denotes the partial derivative of $f$ in the $i$th coordinate, then for $x$ and $y$ as above,
\begin{align*}
\partial_i f^c(y) &= \sum_{s\in \{0,1\}^{d} \atop s_i=0} \nabla_i f(x+s) \prod_{1\le j\le d\atop j\ne i} t_j^{s_j}(1-t_j)^{1-s_j} 
\end{align*}
\end{lmm}
\begin{proof}
Since $w$ is absolutely continuous with bounded support, it follows easily from the definition of $f^c$ that $f^c$ is an absolutely continuous function on $\rr^d$.  Take $x$ and $y$ as in the statement of the lemma. Take any $z\in \zz^d$, and note that $w((y-hz)/h)$ is non-zero if and only if $|y_i/h -  z_i| < 1$ for $i=1,\ldots, d$. Since $y_i/h = x_i + t_i$ for some $t_i\in (0,1)$ for each $i$, therefore $w((y-hz)/h)\ne 0$ if and only if each $z_i$ is either $x_i$ or $x_i+1$. In other words, 
\begin{align*}
f^c(y) &= \sum_{s\in \{0,1\}^d} f(x + s)w\biggl(\frac{y-hx-hs}{h}\biggr)\\
&=\sum_{s\in \{0,1\}^d} f(x + s)\prod_{i=1}^dw(t_i-s_i).
\end{align*}
An easy verification shows that when $s_i\in \{0,1\}$, $w(t_i-s_i) = t_i^{s_i}(1-t_i)^{1-s_i}$. 
This completes the proof of the first identity. For the second, note that by the previous display,
\begin{align*}
\partial_i f^c(y) &= \frac{1}{h}\fpar{}{t_i} f^c(hx+ht)\\
&= \frac{1}{h}\sum_{s\in \{0,1\}^d} f(x + s)\fpar{}{t_i}\biggl(\prod_{j=1}^dw(t_j-s_j)\biggr). 
\end{align*}
Now, 
\[
\fpar{}{t_i} w(t_i-s_i) =
\begin{cases}
1 &\text{ if } s_i = 1,\\
-1 &\text{ if } s_i=0.
\end{cases}
\]
This proves the second identity. 
\end{proof}

\begin{lmm}\label{fcprops}
The function $f^c$ satisfies
\[
|M_h(f)-M(f^c)|\le Ch, \ \text{ and } \ |H_h(f)-H(f^c)|\le Ch^{2/(p+1)}.
\]
\end{lmm}
\begin{proof}
For $x=(x_1,\ldots, x_d)\in \zz^d$, let $B(x)$ be the cube
\begin{align*}
B(x) &:= \{y\in \rr^d: x_i\le y_i/h< x_i+1, \ i=1,\ldots, d\}. 
\end{align*}
Another way to represent $B(x)$ is $\{hx + ht: t\in [0,1)^d\}$. If $y$ is a point in the interior of $B(x)$, then Lemma \ref{fcrep} shows that $f^c(y)$ is a convex combination of $\{f(x+s): s\in \{0,1\}^d\}$. Therefore, for any $r\in [2,\infty)$, 
\begin{align}\label{fc1}
|f^c(y)|^r&\le \max_{s\in \{0,1\}^d}|f(x+s)|^r \le \sum_{s\in \{0,1\}^d} |f(x+s)|^r.
\end{align}
Consequently, 
\begin{align*}
\int_{B(x)} |f^c(y)|^r dy &\le h^d \sum_{s\in \{0,1\}^d} |f(x+s)|^r. 
\end{align*}
Summing over $x$ gives
\begin{equation}\label{fc2}
\begin{split}
\int_{\rr^d} |f^c(y)|^r dy &= \sum_{x\in \zz^d} \int_{B(x)} |f^c(y)|^r dy \\
&\le 2^d h^d \sum_{x\in \zz^d} |f(x)|^r = 2^d \|f\|_{r,h}. 
\end{split}
\end{equation}
By Lemma \ref{fqbd}, the bounds \eqref{fc1} and \eqref{fc2} show that for all $r\in [2, \infty]$,
\begin{equation}\label{fcbd}
\|f^c\|_r\le C(r). 
\end{equation}
Next, note that 
\begin{align}
|f^c(y)-f(x)| &\le \max_{s\in \{0,1\}^d}|f(x+s)-f(x)|\nonumber \\
&\le \sum_{s\in \{0,1\}^d} \sum_{j=1}^dh|\nabla_j f(x+s)|. \label{fc3}
\end{align}
Recall that if $y\in B(x)$, then $\tf(y) = f(x)$. Therefore by inequality \eqref{fc3} and Lemma \ref{regG}, 
\begin{align*}
\int_{\rr^d} (f^c(y)-\tf(y))^2 dy &= \sum_{x\in \zz^d} \int_{B(x)} (f^c(y)-f(x))^2 dy \\
&\le C h^2 G_h(f)\le C h^2. 
\end{align*}
It is easy to see that for all $r$,  $\|\tf\|_r = \|f\|_{r,h}$. Therefore, for any $r\in [2,\infty)$, by the $L^\infty$ bounds from \eqref{fcbd} and Lemma \ref{fqbd} and the above inequality, 
\begin{align*}
\bigl|\|f^c\|_r - \|\tf\|_r\bigr| &\le \|f^c-\tf\|_r \\
&\le \|f^c-\tf\|_\infty^{1-2/r} \|f^c-\tf\|_2^{2/r} \le C(r)h^{2/r}. 
\end{align*}
Using this bound, and again applying the $L^r$ bounds from inequality  \eqref{fcbd} and Lemma \ref{fqbd}, we get 
\begin{equation}\label{fctf}
\bigl|\|f^c\|_r^r - \|\tf\|_r^r\bigr| \le C(r)h^{2/r}. 
\end{equation}
Next, let $\partial_i f^c$ be the partial derivative of $f^c$ in the $i$th coordinate. By Lemma~\ref{fcrep}, 
\begin{align}
|\partial_i f^c(y) - \nabla_i f(x)| &\le \max_{s\in\{0,1\}^d} |\nabla_i f(x+s) - \nabla_i f(x)|\nonumber \\
&\le \sum_{s\in \{0,1\}^d} \sum_{j=1}^d h|\nabla_j \nabla_i f(x+s)|.  \label{fcd1}
\end{align}
Let $\partial_i \tf$ be the function that is identically equal to $\nabla_i f(x)$ in the interior of the box $B(x)$, and arbitrarily defined on the boundaries. Then the above inequality, together with Lemma \ref{fqbd2}, implies that
\begin{align*}
\int_{\rr^d} (\partial_if^c(y)-\partial_i\tf(y))^2 dy &= \sum_{x\in \zz^d} \int_{B(x)} (\partial_if^c(y)-\nabla_if(x))^2 dy \\
&\le C h^2 \max_{1\le j\le d} \|\nabla_j\nabla_i f\|_{2,h}^2\le C h^2. 
\end{align*}
Clearly for all $r$, $\|\partial_i \tf\|_r = \|\nabla_i f\|_{r,h}$. Therefore by the above inequality and Lemma \ref{fqbd2}, $\|\partial_i f^c\|_2$ and $\|\partial_i \tf\|_2$ are both bounded by $C$. Consequently, again applying the previous display, 
\begin{align*}
\bigl|\|\partial_i f^c\|_2^2- \|\partial_i \tf\|_2^2\bigr| &\le C \bigl|\|\partial_i f^c\|_2- \|\partial_i \tf\|_2\bigr|\\
&\le C \|\partial_i f^c - \partial_i \tf\|_2\\
&\le Ch. 
\end{align*}
Using the above bound and \eqref{fctf}, we get
\begin{align*}
|M_h(f)- M(f^c)| &= \bigl|\|\tf\|_2^2 - \|f^c\|_2^2\bigr| \le Ch,
\end{align*}
and 
\begin{align*}
|H_h(f)- H(f^c)| &\le \bigl|\|\partial_i \tf\|_2^2 - \|\partial_i f^c\|_2^2 \bigr| + \bigl|\|\tf\|_{p+1}^{p+1} - \|f^c\|_{p+1}^{p+1}\bigr|\\
&\le Ch^{2/(p+1)}. 
\end{align*}
This completes the proof of the lemma.
\end{proof}
Let $Q$ be the unique, positive radially symmetric solution of \eqref{solitondefine}. Let $Q' := Q_{\lambda(m')}$ be the  continuum ground state soliton of mass $m' := M(f^c)$. Let $\hq:\zz^d \ra\rr$ be the function $\hq(x) := Q'(hx)$. 
\begin{lmm}\label{qprops}
Recall the $o(1)$ convention introduced immediately after the statement of Theorem \ref{tfthm}. Then 
\[
|M(Q')-M_h(\hq)|\le o(1) \ \text{ and } \ |H(Q')-H_h(\hq)|\le o(1). 
\]
\end{lmm}
\begin{proof}
The function $Q$ is a Schwartz function (see the remark following the proof of Proposition B.7 in \cite[Appendix B]{tao06}). In other words, $Q$ is a $C^\infty$ function, and all its derivatives decay faster than any polynomial. This provides ample regularity of $Q$, and together with the scaling relation \eqref{solitonscaling}, this completes the proof of this lemma.
\end{proof}
\begin{lmm}\label{fcqprops}
$|H(f^c)-H(Q')| \le o(1)$. 
\end{lmm}
\begin{proof}
Let $\alpha$ be a number such that $\alpha \hq$ has the same mass as $f$. In other words,
\[
\alpha = \sqrt{\frac{M_h(f)}{M_h(\hq)}}. 
\]
By Lemma \ref{fcprops}, $|M_h(f)-M(f^c)|\le o(1)$, and by Lemma \ref{qprops}, $|M(Q')-M(\hq)|\le o(1)$. But by definition of $Q'$, $M(Q')=M(f^c)$. Thus, 
\begin{equation}\label{alphaprop}
|\alpha-1|\le o(1).
\end{equation}
Now recall that:
\begin{itemize}
\item $Q$ is a Schwartz function (see the remark in the proof of Lemma~\ref{qprops}).
\item $|M(f^c)-m|\le o(1)$ by Lemma \ref{fcprops}.
\item $Q'$ is related to $Q$ by the scaling relation \eqref{solitonscaling}.
\end{itemize}
Combining the above, it follows easily that for all $r\in [2,\infty)$ and $1\le i\le d$, $\|\hq\|_{r,h}\le C(r)$ and  $\|\nabla_i \hq\|_{r,h}\le C(r)$. Together with \eqref{alphaprop}, this implies that 
\[
|H(\alpha \hq)-H(\hq)|\le o(1). 
\]
Since $f$ is a discrete ground state soliton, therefore $H_h(\alpha\hq) \ge H_h(f)$. Combining this with the previous display and Lemmas \ref{fcprops} and \ref{qprops}, we get
\begin{align*}
H(Q') \le H(f^c) &\le H_h(f)+o(1) \\
&\le H_h(\alpha\hq)+ o(1)\\
&\le H_h(\hq)+o(1)\le H(Q')+o(1). 
\end{align*}
This completes the proof of the lemma. 
\end{proof}

\begin{lmm}\label{gni}
For any $q\in [2,\infty]$ and any absolutely continuous $v:\rr^d\ra\cc$ such that $v\in L^2(\rr^d)$ and $\nabla v$   is uniformly bounded, we have
\[
\|v\|_q \le  C(d,q)\|v\|_2^{\frac{2}{d+2}\bigl(1-\frac{2}{q}\bigr) + \frac{2}{q}} \|\nabla v\|_\infty^{\frac{d}{d+2}\bigl(1-\frac{2}{q}\bigr)}, 
\]
where $\|\nabla v\|_\infty := \max_{1\le i\le d} \|\partial_i v\|_\infty$. 
\end{lmm}
\begin{proof}
If $\|\nabla v\|_\infty=0$, then since $v\in L^2$, therefore $v$ must be zero almost everywhere. So let us assume that $\|\nabla v\|_\infty\in (0,\infty)$. 

Take any $x_0\in \rr^d$. Let $B$ be the ball of radius $r$ around $x_0$, where 
\[
r = \biggl(\frac{\|v\|_2}{\|\nabla v\|_\infty}\biggr)^{2/(d+2)}.
\] 
Note that 
\begin{align*}
\frac{\int_B |v(x)|^2 dx}{\vol(B)} \le \frac{\|v\|_2^2}{\vol(B)},
\end{align*}
which shows that there exists $y\in B$ such that 
\[
|v(y)|^2 \le \frac{\|v\|_2^2}{\vol(B)} = C(d) r^{-d}\|v\|_2^2. 
\]
Since $\|x_0-y\| \le r$, this shows that 
\begin{align*}
|v(x_0)|&\le |v(y)| + \|x_0-y\|\|\nabla v\|_\infty\\
&\le C(d)r^{-d/2}\|v\|_2+ Cr\|\nabla v\|_\infty\\
&\le C(d) \|v\|_2^{2/(d+2)} \|\nabla v\|_\infty^{d/(d+2)}. 
\end{align*}
Since this is true for every $x_0$, the right-hand side is a bound for $\|v\|_\infty$. 
To complete the proof, note that for any $q\in (2,\infty)$, $\|v\|_q^q \le \|v\|_\infty^{q-2}\|v\|_2^2$. 
\end{proof}

\begin{proof}[Proof of Theorem \ref{tfthm}]
By Lemma \ref{fcprops}, 
\begin{equation}\label{mequ}
|M(f^c)-m| =|m'-m|\le o(1).
\end{equation}
Again by Lemma~\ref{fcqprops}, $|H(f^c)-H(Q')|\le o(1)$. But $H(Q')=E_{\min}(m')$, $|m'-m|\le o(1)$ and $E_{\min}$ is a continuous function by \eqref{eminform}; consequently, 
\begin{equation}\label{hequ}
|H(f^c)-E_{\min}(m)|\le o(1).
\end{equation}
Together with Lemma \ref{fcprops}, this proves that $|H_h(f)-E_{\min}(m)|\le o(1)$. 
 
The bounds \eqref{mequ} and \eqref{hequ}, together with the orbital stability of ground state solitons (see e.g.\ \cite[Proposition 3]{raphael08}) imply that 
\[
\tl^2(f^c, Q_{\lambda(m)}) \le o(1). 
\]
By \eqref{fcd1} and Lemma \ref{fqbd2}, $\|\partial_i f^c\|_\infty\le C$ for each $i$. Therefore by Lemma \ref{gni}, $\tl^q(f^c, Q_{\lambda(m)}) \le o(1)$ for each $q\in [2,\infty]$.   
\end{proof}

\begin{cor}\label{elimit}
For every fixed $0< m_0\le m_1<\infty$, 
\[
\lim_{h\ra0}\sup_{m_0\le m\le m_1}|E_{\min}(m,h)-E_{\min}(m)|=0.
\]
\end{cor}
\begin{proof}
For every $m> 0$ and $h>0$, let $f_{m,h}$ be a discrete ground state soliton of mass $m$ for the DNLS at grid size $h$. Then $E_{\min}(m,h)=H_h(f_{m,h})$. Theorem \ref{tfthm} completes the proof. 
\end{proof}

\section{Continuum limit of the variational problem}\label{variationalcont}
Recall the objects $\Theta$, $\mm$ an $\mr$ defined in Section \ref{variational}, and the function $\Psi_d$ defined in \eqref{psiddef}.  The following theorem shows that as the grid size goes to zero, the set $\mm(E_0, m_0, h)$ converges to the single point $(0,E_0-E_{\min}(m_0))$. 
\begin{thm}\label{varlim}
Fix $m_0>0$ and $E_{\min}(m_0)< E_0<\infty$. For each $h > 0$, let $(E_h, m_h)$ be an element of $\mm(E_0, m_0,h)$. Then 
\[
\lim_{h\ra 0} m_h = 0 \ \text{ and } \ \lim_{h\ra 0} E_h = E_0-E_{\min}(m_0).
\]
\end{thm}
\begin{proof}
Recall the functions $E^+(m,h)$ and $E^-(m,h)$. 
If $h$ is sufficiently small, then $E_0< dm_0/h^2$ and by Lemma \ref{elimit}, $E_0 \ge E_{\min}(m_0, h)$. Therefore by Lemma \ref{rlmm}, when $h$ is sufficiently small, 
\[
E_h = E^+(m_h,h) = E_0-E_{\min}(m_0-m_h, h).
\] 
We will prove that $m_h \ra 0$ as $h\ra0$ by a subsequence argument. Then by the above identity and Lemma \ref{elimit}, the limit for $E_h$ will be automatically established. Let $m_h$ tend to a point $m'\in [0, m_0]$ along a sequence $h_i\ra0$. Then by Corollary~\ref{elimit}, 
\[
\lim_{i\ra\infty}E_{h_i} = E' := E_0-E_{\min}(m_0-m')\ge 0.
\]
We will show that $m'=0$ by the method of contradiction. Suppose that $m' >0$. Fix $m''\in (0, m')$ and let 
\[
E'' := E_0-E_{\min}(m_0-m'').
\]
Note that since $E_{\min}$ is strictly decreasing (by Lemma \ref{cc5} and Lemma \ref{emin0}), therefore
\begin{equation}\label{eeee}
E'' > E'\ge 0. 
\end{equation}
Also, defining 
\[
E''_i := E_0-E_{\min}(m_0-m'',h_i) = E^+(m'', h_i),
\]
we see that by Corollary \ref{elimit}, 
\begin{equation}\label{epplim}
\lim_{i\ra\infty} E''_i = E''.
\end{equation}
Fix $\gamma\in (0,1)$ and note that for all $i$, 
\begin{align*}
&\Theta(E_{h_i}, m_{h_i}, h_i) - 2\log h_i \\
&\le \log E_{h_i} -  \int_{[0,1]^d}\log \biggl(\frac{(1-\gamma)h_i^2 E_{h_i}}{m_{h_i}} + 2\gamma  \sum_{j=1}^d \sin^2(\pi x_j) \biggr) dx_1\cdots dx_d, 
\end{align*}
and therefore (since $m'>0$), 
\[
\limsup_{i\ra\infty} (\Theta(E_{h_i}, m_{h_i}, h_i) - 2\log h_i) \le  \log E' - C(\gamma),
\]
where 
\[
C(\gamma) := \int_{[0,1]^d}\log \biggl(2\gamma \sum_{j=1}^d \sin^2(\pi x_j) \biggr) dx_1\cdots dx_d. 
\]
Since this is true for all $\gamma\in (0,1)$, we can take $\gamma\ra 1$ in the above bound and get 
\begin{equation}\label{eplim}
\limsup_{i\ra\infty} (\Theta(E_{h_i}, m_{h_i}, h_i) - 2\log h_i) \le \log E' - C(1). 
\end{equation}
On the other hand, for any $i$, \eqref{epplim} gives 
\begin{align*}
&\Theta(E''_i, m'', h_i) - 2\log h_i \\
&= \log E''_i -  \sup_{0< \gamma < 1}\int_{[0,1]^d}\log \biggl(\frac{(1-\gamma)h_i^2 E''_i}{m''} + 2\gamma  \sum_{j=1}^d \sin^2(\pi x_j) \biggr) dx_1\cdots dx_d\\
&\ge \log E''_i - \int_{[0,1]^d}\log \biggl(\frac{h_i^2 E''_i}{m''} + 2  \sum_{j=1}^d \sin^2(\pi x_j) \biggr) dx_1\cdots dx_d\\
&\ra \log E'' - C(1) \ \ \text{as } i\ra\infty. 
\end{align*}
Therefore, by \eqref{eplim} and \eqref{eeee}, this shows that for all sufficiently large $i$,
\begin{equation}\label{ememem}
\Theta(E''_i, m'', h_i) > \Theta(E_{h_i}, m_{h_i}, h_i).
\end{equation} 
But by the definition of $E''_i$, we know that $(E''_i, m'')\in \mr(E_0, m_0, h_i)$. This contradicts the definition of $(E_{h_i}, m_{h_i})$ as a maximizer of $\Theta(E,m,h)$ in $\mr(E_0, m_0, h)$. 
\end{proof}

\section{Proofs of Theorems \ref{elimitthm}, \ref{orbital} and \ref{solitonconv}}\label{threethm}
\begin{proof}[Proof of Theorem \ref{elimitthm}]
Lemma \ref{emax} gives the formula for $E_{\max}(m,h)$. Lemma~\ref{emin0} and Lemma \ref{emfacts3} show that $-\infty < E_{\min}(m,h)< 0$. The subadditive inequality follows from Lemma \ref{cc5}. Lastly, Corollary \ref{elimit} shows that $E_{\min}(m,h) \ra E_{\min}(m)$ as $h \ra 0$, and that the convergence is uniform over compact subsets of $(0,\infty)$. 
\end{proof}

\begin{proof}[Proof of Theorem \ref{orbital}]
Take any sequence of function $f_k$ on $\zz^d$ such that $M_h(f_k) \ra m$ and $H_h(f_k)\ra E_{\min}(m,h)$ as $k \ra\infty$. Let $\alpha_k$ be a constant such that $M(\alpha_k f_k)=m$. Then $\alpha_k \ra 1$ and therefore $H_h(\alpha_k f_k)$ also tends to $E_{\min}(m,h)$.  Theorem \ref{conccomp} now guarantees the existence of a subsequence $\alpha_{k_j}f_{k_j}$ converging to a limit $f\in \ms(m,h)$ in the $\tl^q$ pseudometric for every $q\in [2,\infty]$. Now, $f_k$ is uniformly $L^2$ bounded, and hence uniformly $L^q$ bounded for every $q\in [2,\infty]$. Therefore, since $\alpha_{k_j}\ra1$, $f_{k_j}$ also tends to $f$ in $\tl^q$. This proves compactness of $\ms(m,h)$. To prove that the set is non-empty, simply note that by the definition of $E_{\min}(m,h)$, there exists a sequence $f_k$ satisfying $M_h(f_k)=m$ for all $k$ and $H_h(f_k)\ra E_{\min}(m,h)$.  
\end{proof}

\begin{proof}[Proof of Theorem \ref{solitonconv}]
This is a direct consequence of Theorem \ref{tfthm}. 
\end{proof}

\section{Proof of Theorem \ref{discretelimit}}\label{discretelimitproof}
To be notationally compatible with the theorems of Sections \ref{upperbound}, \ref{lowerbound} and~\ref{radiatingsec}, we will write $E_0$ instead of $E$ and $m_0$ instead of $m$. As in those sections, the numbers $p$, $d$, $h$, $E_0$ and $m_0$ will be fixed throughout this section and will be called the `fixed parameters'. Any positive constant that depends only on the fixed parameters will be denoted simply by $C$, instead of $C(p,d,h,E_0, m_0)$. If the constant depends on additional parameters $a,b,\ldots$, then it will be denoted by $C(a,b,\ldots)$.

Fix $n$, and recall the definition of $\delta$-soliton from Section \ref{upperbound}. Recall also the random function $\phi$ defined in Section \ref{gaussian}, and the objects $\Theta$, $\mm$ and $\mr$ defined in Section \ref{variational}. Lastly, recall that the set $\ms(m,h)$ denotes the set of discrete ground state solitons of mass $m$ at grid size $h$.

Given a function $f: V_n \ra\cc$, let $f^e : \zz^d \ra \cc$ be the extension of $f$ to $\zz^d$, defined as
\[
f^e(x) =
\begin{cases}
f(x) &\text{ if } x\in V_n,\\
0 &\text{ if } x\not \in V_n. 
\end{cases}
\]
Having defined $f^e$, say that $f$ is an ``improved $\delta$-soliton'' if 
there exists $g:\zz^d \ra \cc$ such that
\begin{enumerate}
\item[(a)] $\|f^e-g\|_\infty \le \delta$, and 
\item[(b)] there exists $(E^*,m^*)\in \mm(E_0, m_0,h)$ such that 
\begin{align*}
&|(E_0-E^*)-H_{h}(g)|\le \delta  \ \text{ and } \\ 
&|(m_0-m^*)-M_{h}(g)|\le \delta.
\end{align*}
\end{enumerate}
\begin{lmm}\label{klemma}
If $E_{\min}(m_0,h)< E_0 < dm_0/h^2$, then the set $K$ in the statement of Theorem \ref{discretelimit} can be alternatively described as
\begin{equation}\label{khdef}
\begin{split}
K = K(h) &:= \{m'\in [0,m]: m' = m_0-m^* \\
&\qquad \textup{ for some } (E^*,m^*) \in \mm(E_0, m_0,h)\}. 
\end{split}
\end{equation}
\end{lmm}
\begin{proof}
Suppose that $m'\in [0,m_0]$ maximizes
\[
q(m) := \log(m_0-m) - \Psi_d\biggl(\frac{2h^2 (E_0-E_{\min}(m,h))}{m_0-m}\biggr).
\]
Let $m^*=m_0-m'$ and $E^* = E_0-E_{\min}(m',h) = E^+(m^*,h)$. Note that  $E^+(m^*,h)$ is necessarily positive, for otherwise $q(m')$ would be $-\infty$, which is impossible since it maximizes $q(m)$ in the interval $[0,m_0]$ and $q(m)>-\infty$ for $m$ sufficiently close to $m_0$. This shows that $(E^*, m^*)\in \mr(E_0, m_0, h)$. Now take any $(E,m)\in \mm(E_0, m_0, h)$. Then by Lemma \ref{rlmm}, $E= E^+(m, h)$, and therefore
\begin{align*}
\Theta(E,m,h) &= \log m - \Psi_d\biggl(\frac{2h^2(E_0-E_{\min}(m_0-m, h))}{m}\biggr)\\
&= q(m_0-m)\le q(m') = \Theta(E^*,m^*,h).  
\end{align*} 
Thus, $(E^*, m^*)\in \mm(E_0, m_0,h)$. 

Next, suppose that we are given $m'$ such that $m'=m_0-m^*$ for some $(E^*,m^*)\in \mm(E_0, m_0,h)$. By Lemma \ref{rlmm}, 
\[
\Theta(E^*, m^*, h) = q(m'). 
\]
Take any $m\in [0,m_0]$. We have to show that $q(m')\ge q(m)$. If $q(m)=-\infty$, there is nothing to prove. If not, then 
\[
0 < E_0-E_{\min}(m,h) = E^+(m_0-m, h). 
\]
Let $m_1:= m_0-m$ and $E_1:= E^+(m_0-m,h)$. The above display proves that 
\[
\max\{E^-(m_1, h), 0\} \le E_1 = E^+(m_1,h),
\]
and hence $(E_1,m_1)\in \mr(E_0, m_0, h)$. Thus,
\[
q(m') = \Theta(E^*, m^*, h) \ge \Theta(E_1, m_1, h) = q(m).
\]
This completes the proof of the lemma. 
\end{proof}
\begin{lmm}\label{impsolit}
For any $\eta > 0$, there exists $\delta = \delta(\eta) > 0$ depending only on $\eta$ and the fixed parameters (and not on $n$), such that if $f:V_n \ra\cc$ is an improved $\delta$-soliton, then there exists $v\in \ms(m',h)$ for some $m'\in K$, such that $\tl^\infty(f^e, v) \le \eta$. 
\end{lmm}
\begin{proof}
We will argue by contradiction. Suppose that the statement of the theorem is false. Then there exists an $\eta > 0$, such that for every positive integer $k$, there exist $n_k$ and a function $f_k :V_{n_k} \ra\cc$ such that $f_k$ is an improved $k^{-1}$-soliton, but  $\tl^\infty(f_k^e, v)> \eta$ for all $v\in \cup_{m'\in K} \ms(m',h)$. 

For each $k$, let $g_k$ be a function on $\zz^d$ satisfying the requirements (a) and~(b) in the definition of improved $\delta$-soliton (with $\delta = k^{-1}$ and $f= f_k$). Let $(E^*_k, m^*_k)$ be the corresponding element of $\mm(E_0, m_0, h)$. We will show that $f^e_k$ approaches an element of $\cup_{m'\in K} \ms(m', h)$ in $\tl^\infty$ pseudometric through a subsequence, which will give us the necessary contradiction. 

Since $\mm(E_0, m_0,h)$ is a compact set (Lemma \ref{rlmm}), we may assume without loss of generality that $(E^*_k,m^*_k)$ approaches a limit $(E^*,m^*)\in \mm(E_0, m_0,h)$ as $k\ra\infty$. Let $E' := E_0-E^*$ and $m' := m_0-m^*$. Then $M_h(g_k) \ra m'$ and $H_h(g_k)\ra E'$. But by Lemma \ref{rlmm}, $E'= E_{\min}(m',h)$. Therefore, by Theorem~\ref{orbital}, $g_k$ approaches some $g\in \ms(m',h)$ in the $\tl^\infty$ pseudometric through a subsequence. Since $\|f^e_k-g_k\|_\infty \ra 0$ as $k\ra\infty$, this shows that $f^e_k$ also approaches $g$ in the $\tl^\infty$ pseudometric through the same subsequence. But by Lemma \ref{klemma}, $m'\in K$. This completes the argument. 
\end{proof}

Given a function $f: V_n \ra\cc$, let $f_\tau$ denote a random translate of $f$, that is,
\[
f_\tau(x) := f(x+\tau),
\]
where $\tau$ is uniformly distributed on $V_n$ and the addition on the right-hand side is addition modulo $n$ in each coordinate. 
\begin{lmm}\label{ftaulmm}
If $f$ is a $\delta$-soliton, then, provided that $n > C(\delta)$, 
\[
\pp(f_\tau \textup{ is an improved $2\delta$-soliton}) \ge 1- Cn^{-1}. 
\]
\end{lmm}
\begin{proof}
Let $g:V_n \ra\cc$ be a function satisfying the requirements (a) and (b) in the definition of $\delta$-soliton. Then clearly $f_\tau$ is also a $\delta$-soliton, with $g_\tau$ serving the role of $g$. Let $\partial V_n$ denote the boundary of $V_n$ in $\zz^d$. Since $\tau$ is uniformly distributed on $V_n$, it is easy to see that 
\[
\ee(M_{h,n}(g_\tau, \partial V_n)) = \frac{|\partial V_n|}{|V_n|} M_{h,n}(g) \le Cn^{-1}. 
\]
Therefore by Markov's inequality,
\begin{equation}\label{pmhng}
\pp(M_{h,n}(g_\tau, \partial V_n) > n^{-1/2}) \le Cn^{-1/2}. 
\end{equation}
Now, $M_h(g_\tau^e) = M_{h,n}(g_\tau)$ and $N_h (g_\tau^e) = N_{h,n}(g_\tau)$. Also,  it is easy to verify that 
\[
|G_h(g_\tau^e) - G_{h,n}(g_\tau)| \le C M_{h,n}(g_\tau, \partial V_n). 
\]
Thus, if $M_{h,n}(g_\tau, \partial V_n)\le n^{-1/2}$, then $f_\tau$ is an improved $\delta'$-soliton, where $\delta' = \delta + Cn^{-1/2}$. By \eqref{pmhng}, this completes the proof. 
\end{proof}
\begin{proof}[Proof of Theorem \ref{discretelimit}]
First, assume that 
\begin{equation}\label{focusingcondi}
E_{\min}(m_0, h) < E_0 < \frac{1}{2}E_{\max}(m_0,h).
\end{equation}
By Lemma \ref{emax}, $E_{\max}(m_0,h)= 2dm_0/h^2$. Therefore, we are in the setting of Theorem \ref{lower}. Let 
\begin{align*}
S &:= S_{\ep, h,n} (E_0, m_0) \\
&= \{v\in \cc^{V_n} : |M_{h,n}(v)-m_0|\le \ep, \; |H_{h,n}(v)-E_0|\le \ep\},
\end{align*}
as defined in Section \ref{basic}. Let $f = f_{\ep,h,n}$ be a random  function chosen uniformly from $S$. Let
\[
A := \{v\in S: v \text{ is not a $\delta$-soliton}\}. 
\]
Then by Theorem \ref{upper} and Theorem \ref{lower},
\begin{align*}
\limsup_{\ep \ra0} \limsup_{n\ra\infty} \frac{\log\pp(\phi \in A)-\log\pp(\phi\in S)}{n^d} < 0.  
\end{align*}
By Lemma \ref{basiclmm}, this shows that 
\[
\limsup_{\ep \ra0} \limsup_{n\ra\infty} \frac{\log \pp(f\in A)}{n^d} < 0. 
\]
In particular, 
\begin{equation}\label{fsoliton}
\lim_{\ep \ra 0} \lim_{n\ra\infty} \pp(f \text{ is a $\delta$-soliton}) =1. 
\end{equation}
But by Lemma \ref{ftaulmm}, 
\[
\pp(f_\tau \text{ is an improved $2\delta$-soliton})\ge (1-Cn^{-1})\;\pp(f \text{ is a $\delta$-soliton}).
\]
However, $f_\tau$ has the same distribution as $f$. Combined with \eqref{fsoliton}, this shows that 
\[
\lim_{\ep \ra 0} \lim_{n\ra\infty} \pp(f \text{ is an improved $2\delta$-soliton}) =1. 
\]
Since this is true for any $\delta> 0$, Lemma \ref{impsolit} completes the proof of Theorem~\ref{discretelimit} for the $\tl^\infty$ pseudometric, under the condition \eqref{focusingcondi}. 

To prove the result for general $q\in (2,\infty)$, simply observe that for any such $q$ and any $v:\zz^d \ra\cc$,
\begin{equation}\label{q2inf}
\|v\|^q_q \le \|v\|_\infty^{q-2} \|v\|_2^2,
\end{equation}
and that $\|f\|_2\le C$ by definition. 

When $E_0 \ge \frac{1}{2}E_{\max}(m_0, h)$, Theorem~\ref{discretelimit} is a direct consequence of Theorem~\ref{radiating} and Lemma \ref{basiclmm}. 
\end{proof}

\section{Proof of Theorem \ref{ourmain}}\label{ourmainproof}
As in Section \ref{discretelimitproof}, we will write $E_0$ instead of $E$ and $m_0$ instead of $m$. The convention about the notation $C$ will also be the same. 

Let $K(h)$ be defined as in \eqref{khdef}. 
Fix $q\in (2,\infty]$ and $\delta > 0$. 
\begin{lmm}\label{ourmainlmm}
Whenever $h< C(q,\delta)$, for all $v\in \cup_{m'\in K(h)} \ms(m',h)$
\[
\tl^q(\tilde{v}, Q_{\lambda(m_0)}) \le \delta/2,
\]
where $\tilde{v}$ is the continuum image of $\tilde{v}$ at grid size $h$. 
\end{lmm}
\begin{proof}
Take a sequence $h_k$ decreasing to $0$. 
For each $k$, let $v_k$ be an element of $\cup_{m'\in K(h_k)}\ms(m', h_k)$.  Let $m_k' := M_{h_k}(v_k)$ and $m_k := m_0-m'_k$. By definition of $v_k$, $m_k'\in K(h_k)$, and by definition of $K(h_k)$, there exists $E_k$ such that $(E_k, m_k)\in \mm(E_0, m_0,h)$. Therefore by Theorem \ref{varlim}, $\lim_{k\ra\infty} m_k=0$. Consequently, $M_{h_k}(v_k)$ tends to $m_0$. Since we know that $v_k$ is a discrete ground state soliton for every $k$, therefore by Theorem~\ref{solitonconv}, 
\[
\lim_{k\ra\infty}\tl^q(\tilde{v}_k, Q_{\lambda(m_0)}) =0. 
\] 
A simple argument by contradiction now completes the proof. 
\end{proof}
\begin{proof}[Proof of Theorem \ref{ourmain}]
Fix $q\in (2,\infty]$ and $\delta > 0$. Take any $h$ and any 
\[
v\in \bigcup_{m'\in K(h)} \ms(m',h).
\]
Note that 
\begin{align*}
\tl^q(\tf_{\ep,h,n},\; Q_{\lambda(m_0)}) &\le \tl^q(\tf_{\ep,h,n},\; \tilde{v}) + \tl^q(\tilde{v},\; Q_{\lambda(m_0)})\\
&= \tl^q(f_{\ep,h,n},\; v) + \tl^q(\tilde{v},\; Q_{\lambda(m_0)}).
\end{align*}
Thus, if $\tl^q(\tilde{v},\; Q_{\lambda(m_0)})\le \delta/2$ 
for all $v\in \cup_{m'\in K(h)} \ms(m',h)$, then 
\begin{align*}
\pp\bigl(\tl^q(\tf_{\ep,h,n},\; Q_{\lambda(m_0)}) > \delta\bigr) \le \pp\bigl(\inf_{m'\in K(h)} \inf_{v\in \ms(m',h)}\tl^q(f_{\ep,h,n},\; v) > \delta/2\bigr).
\end{align*}
Now, $E_{\min}(m_0)< E_0 <\infty$ by assumption. Therefore, Corollary \ref{elimit} and Lemma \ref{emax} show that for all sufficiently small $h$,
\[
E_{\min}(m_0, h) < E_0 < \frac{dm_0}{h^2}=\frac{1}{2}E_{\max}(m_0,h). 
\]
Therefore by Theorem \ref{discretelimit} and Lemma \ref{ourmainlmm}, for all $h< C(q,\delta)$, 
\begin{equation}\label{limitzero}
\limsup_{\ep \ra0} \limsup_{n\ra\infty} \pp\bigl(\tl^q(\tf_{\ep,h,n},\; Q_{\lambda(m_0)}) > \delta\bigr) =0. 
\end{equation}
This completes the proof of the first part of Theorem \ref{ourmain}. For the second part, first fix $q\in (2,\infty]$ and $\delta > 0$. Fix some $k > 0$. Choose $h_k$ so small that $h_k< 1/k$ and \eqref{limitzero} holds with $h=h_k$. Given $h_k$, choose $\ep_k$ so small that $\ep_k < 1/k$ and 
\[
\limsup_{n\ra\infty} \pp\bigl(\tl^q(\tf_{\ep_k,h_k,n},\; Q_{\lambda(m_0)}) > \delta\bigr) < 1/k. 
\]
Finally, given $h_k$ and $\ep_k$, choose $n_k$ so large that $n_kh_k > k$ and 
\[
\pp\bigl(\tl^q(\tf_{\ep_k,h_k,n_k},\; Q_{\lambda(m_0)}) > \delta\bigr) < 2/k.
\]
This shows that 
\[
\lim_{k\ra\infty} \pp\bigl(\tl^q(\tf_{\ep_k,h_k,n_k},\; Q_{\lambda(m_0)}) > \delta\bigr) =0.
\]
Since such a sequence $(\ep_k, h_k, n_k)$ exists for any $\delta >0$ (with $q$ fixed), one can extract a sequence that works simultaneously for all $\delta > 0$ by a diagonal argument. In particular, there exists a sequence $(\ep_k, h_k, n_k)$ such that for all $\delta >0$, 
\[
\lim_{k\ra\infty} \pp\bigl(\tl^\infty(\tf_{\ep_k,h_k,n_k},\; Q_{\lambda(m_0)}) > \delta\bigr) =0.
\]
But again, for any $v:\rr^d \ra\cc$ and any $q\in (2,\infty]$, we have the inequality~\eqref{q2inf}. Since the $L^2$ norm of $\tf_{\ep,h,n}$ is uniformly bounded (because $\|\tf_{\ep,h,n}\|^2_2 = M_{h,n}(f_{\ep,h,n})\in [m_0-\ep,m_0+\ep]$), this completes the proof of the second assertion of Theorem \ref{ourmain}. 
\end{proof}

\section{Proof of Theorem \ref{src}}\label{srcproof}
Take any $\nu \in \mm$ and any bounded measurable function $\phi$ on $S$. By Birkhoff's ergodic theorem (see e.g.~\cite[Theorem 1.14]{walters82}) and the ergodicity of the map $T_1$ with respect to the measure $\nu$, 
\begin{align*}
\nu\biggl\{f\in S : \lim_{l\ra \infty} & \frac{1}{l}\sum_{r=1}^l \phi(T_r f)= \int_S\phi(v) d\nu(v)\biggr\}=1.
\end{align*}
Now fix $\phi$ to be the function
\[
\phi(f) = \int_0^1 1_{\bigl\{\tl^\infty(\widetilde{T_tf}, \, Q_{\lambda(m)})\, > \, \delta\bigr\}}\, dt.
\]
Then for any $r$, 
\begin{align*}
\phi(T_r f) &= \int_0^1 1_{\bigl\{\tl^\infty(\widetilde{T_{t+r}f}, \, Q_{\lambda(m)})\, > \, \delta\bigr\}}\, dt\\
&= \int_r^{r+1} 1_{\bigl\{\tl^\infty(\widetilde{T_tf}, \, Q_{\lambda(m)})\, > \, \delta\bigr\}}\, dt.
\end{align*}
Consequently,
\[
\frac{1}{l}\sum_{r=1}^l \phi(T_r f) = \frac{1}{l}\int_0^l 1_{\bigl\{\tl^\infty(\widetilde{T_tf}, \, Q_{\lambda(m)})\, > \, \delta\bigr\}}\, dt.
\]
The above steps show that for any $\nu\in \mm$, 
\begin{equation}\label{srceq}
\begin{split}
&\text{if } \int_S\phi(v)\, d\nu(v) < \delta, \\
&\text{then $\nu$ satisfies SRC with error $\delta$.}
\end{split}
\end{equation} 
For each $j$, let $g_j:\rr\ra\rr$ be the function
\[
g_j(x) = 
\begin{cases}
1 &\text{ if } x > \delta + j^{-1},\\
y &\text{ if } x = \delta + yj^{-1} \text{ for some } y\in [0,1],\\
0 &\text{ if } x< \delta. 
\end{cases}
\]
Note that (a) $g_j$ is a continuous function, (b) $g_j(x) \le 1_{\{x > \delta\}}$ for all $x$ and $j$, and (c) for each $x$, $g_j(x)$ increases to $1_{\{x> \delta\}}$ as $j\ra\infty$. 

Define the map $\xi: S \ra \rr$ as 
\[
\xi(v) := \tl^\infty(\tilde{v}, Q_{\lambda(m)}). 
\]
It is easy to see that $\xi$ is a continuous function on $S$. For each $j$, let 
\[
\phi_j(v) := \int_0^1 g_j(\xi(T_t v)) dt.
\]
By our previous observations about $g_j$ and the continuity of $\xi$ and $T_t$ (as remarked in Section \ref{srcsec}), we see that (a) for each $j$, $\phi_j$ is a continuous function taking value in $[0,1]$, (b) $\phi_j(v)\le \phi(v)$ for each $v$ and $j$, and (c) for each $v$, $\phi_j(v)$ increases to $\phi(v)$ as $j\ra\infty$. 

By \eqref{ergdec} and the continuity of $\phi_j$, we have that for each $j$, 
\begin{equation}\label{phi1}
\int_{\mm}\biggl(\int_S\phi_j(v)\, d\nu(v)\biggr) d\tau(\nu) = \int_S\phi_j(v)\, d\mu(v).
\end{equation}
Since $\phi_j\le \phi$, therefore 
\begin{equation}\label{phi2}
\int_S\phi_j(v)\, d\mu(v)\le \int_S\phi(v)\, d\mu(v)
\end{equation}
for all $j$. Since $\phi_j \ra\phi$ pointwise, therefore by Fatou's lemma from measure theory and \eqref{phi1} and \eqref{phi2},
\begin{equation}\label{twoeqs}
\int_{\mm}\biggl(\int_S\phi(v)\, d\nu(v)\biggr) d\tau(\nu)\le \liminf_{j\ra\infty} \int_S\phi_j(v)\, d\mu(v) \le \int_S\phi(v)\, d\mu(v)
\end{equation}
Again, since $T_t$ preserves $\mu$,
\begin{align}
\int_S\phi(v) d\mu(v)&= \int_S \int_0^1 1_{\bigl\{\tl^\infty(\widetilde{T_tv}, \, Q_{\lambda(m)})\, > \, \delta\bigr\}}\, dt \,d\mu(v)\nonumber \\
&= \int_0^1 \int_S  1_{\bigl\{\tl^\infty(\widetilde{T_tv}, \, Q_{\lambda(m)})\, > \, \delta\bigr\}}\,d\mu(v)\, dt \nonumber \\
&= \int_0^1 \int_S  1_{\bigl\{\tl^\infty(\tilde{v}, \, Q_{\lambda(m)})\, > \, \delta\bigr\}}\,d\mu(v)\, dt \nonumber \\
&= \int_S  1_{\bigl\{\tl^\infty(\tilde{v}, \, Q_{\lambda(m)})\, > \, \delta\bigr\}}\,d\mu(v). \label{invariance}
\end{align}
By \eqref{twoeqs}, this gives
\begin{equation}\label{invariance3}
\begin{split}
\int_{\mm}\biggl(\int_S\phi(v)\, d\nu(v)\biggr) d\tau(\nu)&\le \int_S1_{\bigl\{\tl^\infty(\tilde{v}, \, Q_{\lambda(m)})\, > \, \delta\bigr\}}\, d\mu(v).
\end{split}
\end{equation}
But by Theorem \ref{ourmain}, 
\begin{equation}\label{ourmain2form}
\lim_{h\ra 0} \limsup_{\ep \ra 0} \limsup_{n\ra \infty}\int_S1_{\bigl\{\tl^\infty(\tilde{v}, \, Q_{\lambda(m)})\, > \, \delta\bigr\}}\, d\mu(v) = 0.  
\end{equation}
The displays \eqref{srceq}, \eqref{invariance3} and Markov's inequality show that 
\begin{align*}
&\tau\bigl\{\nu\in \mm: \nu \text{ satisfies SRC with error } \delta\bigr\} \\
&\ge \tau\biggl\{\nu\in \mm:\int_S\phi(v)\, d\nu(v) < \delta\biggr\}\\
&\ge 1 - \frac{1}{\delta} \int_{\mm}\biggl(\int_S\phi(v)\, d\nu(v)\biggr) d\tau(\nu)\\
&\ge 1 - \frac{1}{\delta} \int_S1_{\bigl\{\tl^\infty(\tilde{v}, \, Q_{\lambda(m)})\, > \, \delta\bigr\}}\, d\mu(v).
\end{align*}
Together with \eqref{ourmain2form}, this completes the proof of Theorem \ref{src}.

\vskip.2in
\noindent{\bf Acknowledgments.} The author thanks Pierre Germain, Partha Dey, Terence Tao, Jalal Shatah, Phil Sosoe, Kay Kirkpatrick, Julien Barr\'e, Lai-Sang Young, Carlos Kenig, Persi Diaconis, Fraydoun Rezakhanlou, Stefano Olla, Raghu Varadhan and the anonymous referee for useful comments. The manuscript of Rapha\"el \cite{raphael08} (which was brought to my attention by Pierre Germain) and Tao's book~\cite{tao06} have been of immeasurable help.

\end{document}